\documentclass[11pt,a4paper,reqno]{amsart}
\makeatletter
\def\l@subsection{\@tocline{2}{0pt}{2.5pc}{5pc}{}}
\def\l@subsubsection{\@tocline{3}{0pt}{4.5pc}{7pc}{}}
\makeatother
\usepackage[utf8]{inputenc}
\usepackage[T1]{fontenc}
\usepackage{lmodern}
\usepackage[american]{babel}
\usepackage{csquotes}
\usepackage{microtype}
\AtBeginDocument{\sbox0{$\left(\frac{1}{2}\right)$}}
\usepackage{bbm}

\usepackage{amsmath, amsfonts, amssymb, amsthm, mathtools}
\usepackage{polynom}
\usepackage{environ}
\usepackage{etoolbox}

\usepackage{geometry}
\usepackage{fancyhdr, afterpage, lastpage}
\headwidth=\textwidth
\usepackage{graphicx}
\usepackage[svgnames]{xcolor}
\usepackage{tikz}
\usepackage{tikz-cd}
\usepackage{rank-2-roots}
\usepackage{tabularray}
\usepackage{booktabs}
\usepackage{multicol}
\usepackage{float}
\usepackage{placeins}
\usepackage{epigraph}
\usepackage[shortlabels]{enumitem}

\usepackage{listings}
\definecolor{codegreen}{rgb}{0,0.6,0}
\definecolor{codegray}{rgb}{0.5,0.5,0.5}
\definecolor{codepurple}{rgb}{0.58,0,0.82}
\definecolor{backcolour}{rgb}{0.97,0.97,0.97}

\lstdefinestyle{pythonstyle}{
    backgroundcolor=\color{backcolour},   
    commentstyle=\color{codegreen},
    keywordstyle=\color{magenta},
    numberstyle=\tiny\color{codegray},
    stringstyle=\color{codepurple},
    basicstyle=\ttfamily\footnotesize,
    breakatwhitespace=false,         
    breaklines=true,                 
    captionpos=b,                    
    keepspaces=true,                 
    numbers=left,                    
    numbersep=5pt,                  
    showspaces=false,                
    showstringspaces=false,
    showtabs=false,                  
    tabsize=4,
    frame=none
}
\usepackage[colorlinks=true, urlcolor=blue, linkcolor=DarkBlue, citecolor=DarkGreen]{hyperref}

\theoremstyle{plain}
\newtheorem{theorem}{Theorem}[section]
\newtheorem{proposition}[theorem]{Proposition}
\newtheorem{corollary}[theorem]{Corollary}
\newtheorem{lemma}[theorem]{Lemma}

\theoremstyle{definition}
\newtheorem{definition}[theorem]{Definition}

\NewCommandCopy{\proofqedsymbol}{\qedsymbol}
\newcommand{\exampleqedsymbol}{$\triangle$}

\theoremstyle{remark}
\newtheorem{remark}{Remark}[section]
\AtBeginEnvironment{remark}{%
  \pushQED{\qed}\renewcommand{\qedsymbol}{\exampleqedsymbol}%
}
\AtEndEnvironment{remark}{\popQED\endexample}

\newtheorem{example}{Example}[section]
\AtBeginEnvironment{example}{%
  \pushQED{\qed}\renewcommand{\qedsymbol}{\exampleqedsymbol}%
}
\AtEndEnvironment{example}{\popQED\endexample}

\newcommand{\inv}[1]{{#1}^{-1}}
\newcommand{\sgn}{\mathrm{sgn}}

\providecommand{\ZZ}{\mathbb{Z}}

\providecommand{\RR}{\mathbb{R}}

\providecommand{\CC}{\mathbb{C}}

\newcommand{\abs}[1]{\left\vert #1 \right\vert}

\newcommand{\id}{\textup{id}}

\newcommand{\dif}{\mathrm{d}}
\newcommand{\eps}{\varepsilon}

\newcommand{\dw}{\dif_W}
\newcommand{\dce}{\dif_\mathrm{CE}}

\newcommand{\F}{\mathbf{F}_W}

\newcommand{\Th}{\boldsymbol{\theta}_W}

\newcommand{\g}{\mathfrak{g}}
\newcommand{\h}{\mathfrak{h}}
\newcommand{\kk}{\mathfrak{k}}
\newcommand{\m}{\mathfrak{m}}

\newcommand{\ttt}{\mathfrak{t}}

\newcommand{\orb}[1]{\mathcal{O}_{#1}}

\newcommand{\SU}[1]{\mathrm{SU}(#1)}
\newcommand{\SUk}{\SU{k}}

\newcommand{\sll}[1]{\mathfrak{sl}(#1,\CC)}
\newcommand{\slk}{\sll{k}}
\newcommand{\suu}[1]{\mathfrak{su}(#1)}
\newcommand{\suk}{\suu{k}}

\newcommand{\killing}[2]{\langle #1, #2 \rangle}

\DeclareMathOperator{\Ad}{Ad}

\DeclareMathOperator{\ad}{ad}

\DeclareMathOperator{\tr}{tr}
\DeclareMathOperator{\Hom}{Hom}

\newcommand{\cp}[1]{\CC\mathbb{P}^{#1}}

\title[]{Explicit Transgression Primitives on Flag Manifolds}

\author{Greg L. A. Weiler}
\thanks{While working on this manuscript, the author was supported as a member and associate member by the DFG RTG 2491 ``Fourier Analysis and Spectral Theory''.}
\address{Mathematisches Institut, Georg-August-Universit\"at G\"ottingen, Bunsenstra\ss e 3--5, 37073 G\"ottingen, Germany}
\email{greg.weiler@mathematik.uni-goettingen.de}

\date{\today}

\begin{document}

\begin{abstract}
For the canonical principal fibration $K \to K/T$ of a compact semisimple Lie group over its flag manifold, Borel's classical presentation of the cohomology ring as the coinvariant algebra $H^*(K/T; \CC) \cong S(\h^\vee)/I_W$ abstractly establishes the vanishing of the positive-degree characteristic forms $p_j(F_{\ttt}) = \tr(F_{\ttt}^j)$ in de Rham cohomology. However, the classical proof is non-constructive, and standard algebraic projection from the universal contractible Weil algebra $W(\kk)$ yields closed absolute primitives on $K$ that fail to descend to the base. 

In this paper, we explicitly construct odd-degree relative transgression primitives $\alpha_j \in (\Lambda^{2j-1}\m^\vee)^T \cong {}^K\Omega^{2j-1}(K/T)$ satisfying $\dif_\m \alpha_j = -p_j(F_{\ttt})$ directly on the flag manifold. By introducing a 1-parameter family of connections $\theta_s = \theta_{\ttt} + s\theta_\m$ interpolating between the canonical connection and the flat Maurer--Cartan form, the interpolating curvature factors as $F_s = (1-s)(F_{\ttt} - s\theta_\m^2)$. Expanding this power via complete non-commutative symmetric word sums and evaluating the parameter integrals via Euler's Beta integral yields a universal closed-form formula for all degrees $j \ge 1$. 

We demonstrate that the explicit primitives exhibit a structural transition: for degree $j = 2$, the second primitive $\alpha_2 = -\frac{1}{3}\tr(\theta_\m^3)$ is purely unperturbed by curvature; for $j = 3$, the third primitive $\alpha_3$ is strictly linear in $F_{\ttt}$, where cyclic trace permutations commute trivially. At degree $j = 4$, quadratic curvature $F_{\ttt}^2$ enters alongside the soldering form $\theta_\m^2$, where the non-vanishing commutator $[F_{\ttt}, \theta_\m^2] \neq 0$ shatters cyclic trace symmetry into independent necklaces, producing the irreducible invariant $\tr(\theta_\m \wedge F_{\ttt} \wedge \theta_\m^2 \wedge F_{\ttt})$.

For the rank-two root system $A_2$ ($\mathfrak{sl}_3(\CC)$), we ground these forms in concrete Euclidean geometry: $T$-invariant cochains correspond bijectively to closed directed loops on the planar hexagonal root lattice, the second primitive $\alpha_2 = \Omega_- - \Omega_+$ manifests as the oriented difference of triangular loops, and the relative differential $\dif_\m$ acts as an edge-fracturing operator, providing a visual witness to the vanishing of the second Chern class. Finally, while character theory encounters an additive multiplicity barrier at $k=4$ due to non-trivial Specht module multiplicities in the cochain complex, we resolve the coinvariant algebra globally via a  differential graded algebra Koszul resolution $K_\bullet(p_2, \dots, p_k)$, directly recovering the classical Poincaré polynomial $P_{K/T}(t) = \prod_{j=2}^k \frac{1-t^{2j}}{1-t^2}$ and Euler characteristic $\chi(K/T) = k!$ without matrix-rank ambiguities.
\end{abstract}

\maketitle

\tableofcontents

\section{Introduction and Main Results}

\subsection{The Transgression Problem on Flag Manifolds}

For a compact connected Lie group $K$ with maximal torus $T$, Borel's classical theorem \cite{Borel1953SurLC} establishes that the complex cohomology of the flag manifold $K/T$ is isomorphic to the coinvariant algebra of the Weyl group: $H^*(K/T; \CC) \cong S(\h^\vee)/I_W$. Under the Chern--Weil homomorphism, this presentation is realized by the canonical connection $\theta_{\ttt}$ on the principal bundle $K \to K/T$, whose curvature $F_{\ttt} \in \Omega^2(K/T, \ttt)$ maps the positive-degree Weyl-invariant polynomials to zero in de Rham cohomology: $[p_j(F_{\ttt})] = 0 \in H^{2j}_{\mathrm{dR}}(K/T; \CC)$.

Our goal in this paper is to explicitly link Borel's symmetric picture to the exterior picture. Unlike in the world of rings, where invariant generators can simply be quotiented out, the corresponding cochains in the differential complex must vanish as cohomology classes. As such, we seek coboundaries corresponding to the discarded generators.

The Weil algebra $W(\kk) = S^\bullet(\kk^\vee) \otimes \Lambda^\bullet(\kk^\vee)$ is equipped with the universal connection $\boldsymbol{\theta}_W \in \kk \otimes W^1\kk$ and curvature $\F \in \kk \otimes W^2\kk$. These satisfy the Weil structural equation and the second Bianchi identity:
\[ \dw \boldsymbol{\theta}_W = \F - \frac{1}{2}[\boldsymbol{\theta}_W, \boldsymbol{\theta}_W], \qquad \dw \F = [\F, \boldsymbol{\theta}_W] . \]
For each fundamental invariant $f_j(\F) = \tr(\F^j) \in (S^j\kk^\vee)^\kk$, the associated \textit{universal Chern--Simons cochain} is defined via the standard homotopy operator $h$ as 
\[Tf_j := h(f_j(\F)) \in W^{2j-1}\kk.\]
 Evaluating the homotopy integral isolates the term linear in $\dif t$, yielding:
\begin{equation}\label{csform2}
	Tf_j = j \int_0^1 \tr\left( \boldsymbol{\theta}_W \wedge \mathbf{F}_{tW}^{j-1} \right) \dif t ,
\end{equation}
where $\mathbf{F}_{tW} = t\F - \frac{1}{2}(t - t^2)[\boldsymbol{\theta}_W, \boldsymbol{\theta}_W]$ denotes the curvature of the intermediate connection $\boldsymbol{\theta}_{tW} = t\boldsymbol{\theta}_W$ (cf.\ Subsection~\ref{sec:cs}). This cochain satisfies the universal transgression equation:
\begin{equation}
	\dw(Tf_j) = f_j(\F) .
\end{equation}
Applying the projection $\pi \colon W(\kk) \to \Lambda^\bullet \kk_\CC^\vee$ (in our case $\kk_\CC=\slk$) maps $Tf_j$ to the \textit{absolute primitive} on the compact group $K$:
\[ \eta_j := \pi(Tf_j) \in (\Lambda^{2j-1}\kk_\CC^\vee)^\kk \cong {}^K\Omega^{2j-1}(K, \CC)^K , \]
which is $\Ad_K$-invariant by the cyclic invariance of the trace and closed ($\dif_\wedge \eta_j = 0$).

However, the central objective of this paper is to understand the differential geometry and topology directly on the flag manifold $K/T$. While topological arguments (such as the splitting of the trivial bundle $K \times_T \CC^k$) abstractly establish the vanishing of the characteristic classes $cw_K^\CC(f_j) = [f_j(F_{\ttt})] = 0 \in H^{2j}_\mathrm{dR}(K/T; \CC)$, they do not provide explicit primitives realizing the exactness on the base. The absolute primitives $\eta_j$ cannot serve this purpose: they live on the total space $K$ and are not horizontal with respect to the $T$-fibration. Indeed, contracting $\eta_j$ with a vertical generator $A \in \ttt$ yields $\iota_A \eta_j= j\tr(A \theta_\mathrm{MC}^{2j-2}) \not\equiv 0$ (for instance, taking $A = h_\alpha$, we have $(\iota_{h_\alpha} \eta_2)(e_\alpha, e_{-\alpha}) = 2\tr(h_\alpha^2) \neq 0$). In particular, $\eta_j$ is not basic, precluding it from descending to $K/T$.

To construct explicit, canonical odd-degree primitives directly on the flag manifold $K/T$ and to identify them with directed closed loops on the root lattice, we must adapt the transgression machinery to the reductive splitting $\kk = \ttt \oplus \m$.

\subsection{The Weil Algebra Obstruction}

Given the universal transgression $\dw(Tf_j) = f_j(\F)$ in the Weil algebra $W(\kk)$, one might ask whether the relative primitives on $K/T$ could be obtained purely algebraically. Evaluating the universal cochain $Tf_j$ on the reductive splitting $\kk = \ttt \oplus \m$ via a naive homomorphism $\psi \colon W(\kk) \to (\Lambda^\bullet \m^\vee)^T \otimes \kk$ setting $\boldsymbol{\theta}_W \mapsto \theta_\m$ and $\F \mapsto -F_{\ttt}$ fails because $\psi$ is not a chain map: the non-trivial action of $\ttt$ on $\m$ prevents $\psi$ from commuting with the differentials ($\dif_\m \circ \psi \neq \psi \circ \dw$, see Remark~\ref{rmk:psi_w_failure}).

Because such static homomorphisms fail to transport the differential across the reductive splitting $\kk = \ttt \oplus \m$ (detailed in Section~\ref{sec:geocan}), we use a dynamic geometric interpolation in the affine space of connections on $K \to K/T$ in order to capture the exact primitives.

\subsection{Statement of Principal Results}

The primary contributions of this paper resolve the transgression problem, its non-commutative structure, its visual root-lattice realization, and its homological resolution across four main theorems:

\paragraph{\textit{1. Universal Relative Transgression Primitives.}}
Overcoming the static evaluation defect requires interpolating dynamically between the canonical connection $\theta_0 = \theta_{\ttt}$ and the flat Maurer--Cartan form $\theta_1 = \theta_{\mathrm{MC}} = \theta_{\ttt} + \theta_\m$. On the principal bundle $K \to K/T$, the family $\theta_s = \theta_{\ttt} + s\theta_\m$ has horizontal curvature factoring as $F_s = (1-s)(F_{\ttt} - s\theta_\m^2)$. Evaluating the transgression integral yields universal closed-form primitives:%

\begin{theorem}[Universal Closed Formula for Relative Primitives]\label{thm:intro_universal_formula}
	For each $j \ge 1$, the differential form
	\[ \alpha_j := j \int_0^1 \tr(\theta_\m \wedge F_s^{j-1}) \, \dif s \in \Omega^{2j-1}_{\mathrm{bas}}(K) \cong \Omega^{2j-1}(K/T) \]
	descends to a left-$K$-invariant form on the flag manifold $K/T$ satisfying the transgression boundary equation:
	\[ \dif_\m \alpha_j = -p_j(F_{\ttt}) = -\tr(F_{\ttt}^j) \in (\Lambda^{2j}\m^\vee)^T . \]
	Moreover, $\alpha_j$ admits the exact closed-form expansion:
	\begin{equation}\label{eq:intro_closed_formula}
		\alpha_j = \sum_{k=1}^{j-1} (-1)^k\inv{\binom{j+k}{k}} \tr\big(\theta_\m \wedge S_{j-1, k}(F_{\ttt}, \theta_\m^2)\big) ,
	\end{equation}
	where $S_{n, k}(A, B)$ denotes the complete non-commutative word sum over all $\binom{n}{k}$ words of length $n$ with $k$ factors of $B$ and $n-k$ factors of $A$.
\end{theorem}

%
\paragraph{\textit{2. Curvature Perturbations and Necklaces.}}
A non-commutative threshold occurs at degree $j=4$: for $j \le 3$, all constituent words in the curvature and soldering form commute under the trace, whereas at $j=4$, the non-vanishing commutator $[F_{\ttt}, \theta_\m^2] \neq 0$ appears. The explicit forms we derived for the absolute primitives $\eta_j\in\Lambda^{2j-1}\g_\CC^\vee$ and the relative transgression primitives $\alpha_j\in(\Lambda^{2j-1}\m^\vee)^T$ explain this in terms of a perturbation by the canonical curvature $F_{\ttt}$. From the absolute transgression on $K$, we have:
\[\eta_j=\frac{j}{2j-1}\tr(\theta_\mathrm{MC}^{2j-1}) ,\]
and from Theorem~\ref{thm:intro_universal_formula}, we have:
\[\alpha_j = \sum_{k=1}^{j-1} (-1)^k\inv{\binom{j+k}{k}} \tr\big(\theta_\m \wedge S_{j-1, k}(F_{\ttt}, \theta_\m^2)\big) .\]
Isolating the final summand, this becomes:
\[ \alpha_j = (-1)^{j-1}\inv{\binom{2j-1}{j-1}} \tr(\theta_\m^{2j-1}) + \sum_{k=1}^{j-2} 
(-1)^k\inv{\binom{j+k}{k}}
\tr\big(\theta_\m \wedge S_{j-1, k}(F_{\ttt}, \theta_\m^2)\big) . \]
The leading term is proportional to the bare projection $\eta_j|_\m$, while the lower-order summands ($k < j - 1$) provide the necessary perturbation for the transgression equation $\dif \alpha_j = -p_j(F_{\ttt})$ to be fulfilled. We therefore are able to obtain the desired Chern class, which the bare projection $\eta_j$ could not accomplish. In particular, $\dif_\m(\eta_j|_\m)$ contained no curvature factors, highlighting the necessity of the perturbation. For the flag manifold of $\mathrm{SU}(3)$ ($k=3$), the third primitive vanishes ($\alpha_3 \equiv 0$), so non-vanishing curvature perturbations first manifest in non-trivial primitives for $k \ge 4$. This threshold becomes manifest when trying to compute the ranks of the Chevalley--Eilenberg differentials on the reduced complex. Indeed, the decomposition of the complex into $S_k$-irreps block-diagonalizes the diferentials, but does not determine the diagonal blocks for $k\ge4$.

This resolves our initial goal of producing explicit primitives realizing Borel's isomorphism in terms of the relative Chevalley--Eilenberg complex of the flag manifold $K/T$. In a first attempt to bridge the exterior world of forms and the symmetric world of coinvariants, we used the Weil algebra. Here, we found canonical primitives, namely the unviversal Chern--Simons cochains $Tf_j = h(f_j(\F)) \in W^{2j-1}\g$ (cf.\ Proposition~\ref{cs}), which satisfy the Chern--Simons equation $\dw(Tf_j)=f_j(\F)=\tr\F^j$. The naive attempt to obtain the desired primitives by projecting these forms down to the candidates $\eta_j=\pi(Tf_j)$ was unsuccessful. Indeed, projection to $K$ yields the closed absolute cocycle $\dif_\wedge \eta_j = 0$.
It was by examining the geometry of the principal bundle $K\to K/T$ and the algebra of the reductive splitting $\kk=\ttt\oplus\m$ in great detail that we were able to construct and compute the explicit primitives by applying the Chern--Simons recipe to the principal bundle ingredients. A posteriori, the relative transgression primitives $\alpha_j$ are obtained by perturbing the candidates $\eta_j$ with the help of the curvature of the canonical connection $\theta_{\ttt}$.

At degree $j = 4$, we explicitly compute the non-commutative threshold of the perturbation:

\begin{theorem}[The Fourth Primitive and Necklace Splitting]\label{thm:intro_fourth_primitive}
	For $j = 4$, the geometric transgression primitive evaluates to the invariant $7$-form:
	\begin{equation}
		\alpha_4 = -\frac{2}{5}\tr(\theta_\m^3 \wedge F_{\ttt}^2) - \frac{1}{5}\tr(\theta_\m \wedge F_{\ttt} \wedge \theta_\m^2 \wedge F_{\ttt}) 
		+ \frac{1}{5}\tr(\theta_\m^5 \wedge F_{\ttt}) - \frac{1}{35}\tr(\theta_\m^7) ,
	\end{equation}
	satisfying $\dif_\m \alpha_4 = -\tr(F_{\ttt}^4)$. Degree $j = 4$ is the minimal degree where the non-vanishing commutator $[F_{\ttt}, \theta_\m^2] \neq 0$ breaks the cyclic trace symmetry of the word sums, causing words of identical letter content to split into distinct trace classes and producing the irreducible necklace $\tr(\theta_\m \wedge F_{\ttt} \wedge \theta_\m^2 \wedge F_{\ttt})$.
\end{theorem}

\paragraph{\textit{3. Combinatorial Geometry on the $A_2$ Root Lattice.}}

For the rank-two system $A_2$ ($\mathfrak{sl}_3(\CC)$), the relative Chevalley--Eilenberg complex $(\Lambda^\bullet \m^\vee)^T$ admits an intriguing visual dictionary on the planar hexagonal root lattice:

\begin{theorem}[Root Lattice Loops and Edge-Fracturing]\label{thm:intro_root_lattice}
	On the $A_2$ root lattice, $T$-invariant cochains correspond bijectively to closed directed polygonal loops. In particular:
	\begin{enumerate}[(i)]
		\item The second primitive $\alpha_2 \in (\Lambda^3 \m^\vee)^T$ is given by the oriented difference of fundamental triangular loops:
		\[ \alpha_2 = \Omega_- - \Omega_+ = \omega^{-\alpha_1} \wedge \omega^{-\alpha_2} \wedge \omega^{\alpha_3} - \omega^{\alpha_1} \wedge \omega^{\alpha_2} \wedge \omega^{-\alpha_3} . \]
		\item The relative differential $\dif_\m$ acts as an edge-fracturing operator, shattering the triangular loops along their edges into the three parallelogram configurations of $\tr(F_{\ttt}^2)$, visually witnessing the transgression equation:
		\[ \dif_\m \alpha_2 = \dif_\m \Omega_- - \dif_\m \Omega_+ = -\frac{1}{2}\tr(F_{\ttt}^2) - \frac{1}{2}\tr(F_{\ttt}^2) = -\tr(F_{\ttt}^2) = -p_2(F_{\ttt}) . \]
		\item The third primitive vanishes ($\alpha_3 \equiv 0$) due to the absence of closed $5$-step loops ($C^5 = \{0\}$).
	\end{enumerate}
\end{theorem}

\paragraph{\textit{4.  Koszul DGA and Poincaré Polynomials.}}
While an additive character-theoretic approach via the Frobenius characteristic map block-diagonalizes the cochain complex into Specht modules, it encounters the aforementioned threshold at rank $k = 4$ ($\mathfrak{sl}_4$), where Specht module multiplicities exceed $1$ (yielding undetermined scalar matrix blocks like $\mathrm{Hom}^{S_4}(V_{\mathbf{3}}^{\oplus 3}, V_{\mathbf{3}}^{\oplus 5}) \cong \CC^{5 \times 3}$). We resolve this obstruction by modeling the multiplicative DGA structure directly:

\begin{theorem}[ Koszul Resolution and Poincaré Polynomials]\label{thm:intro_koszul}
	Let $K_\bullet(p_2, \dots, p_k) := S(\h^\vee) \otimes \Lambda(V_{\mathrm{odd}})$, where $V_{\mathrm{odd}}=\bigoplus_{j=2}^k\CC\eta_j$, be the  differential graded algebra generated by the restricted fundamental invariants $p_j = \tr(X^j)|_\h$ in Koszul degree $0$ and odd symbols $\eta_j$ in Koszul degree $1$ with differential $\dif_\kappa(1 \otimes \eta_j) = p_j \otimes 1$.
	\begin{enumerate}[(i)]
		\item The sequence $p_2, \dots, p_k$ forms a regular sequence in $S(\h^\vee)$, and $K_\bullet(p_2, \dots, p_k)$ is a finite free acyclic resolution of the coinvariant algebra $S(\h^\vee)/I_W \cong H^*(K/T; \CC)$.
		\item Applying the Euler--Poincaré principle degree-by-degree to the homogeneous components of this resolution recovers the classical Poincaré polynomial and Euler characteristic:
		\begin{align*}
			P_{K/T}(t) &= \prod_{j=2}^k \frac{1 - t^{2j}}{1 - t^2} = \prod_{j=2}^k \big(1 + t^2 + t^4 + \dots + t^{2(j-1)}\big),\\
			\chi(K/T) &= P_{K/T}(1) = k! ,
		\end{align*}
		completely bypassing matrix-rank ambiguities.
	\end{enumerate}
\end{theorem}

\subsection{Organization of the Paper}

The remainder of the paper is organized as follows:

In \textbf{Section~\ref{sec:geocan}}, we examine the differential geometry of the canonical fibration $K \to K/T$. We formalize the reductive splitting $\kk = \ttt \oplus \m$, decompose the Maurer--Cartan form into the vertical connection $\theta_{\ttt}$ and horizontal soldering form $\theta_\m$, prove the trace orthogonality of the splitting (Lemma~\ref{lem:trace_orthogonality}), and review the Chern--Weil derivation of Borel's coinvariant presentation $H^*(K/T; \CC) \cong S(\h^\vee)/I_W$ (Theorem~\ref{thm:borel_isomorphism}).

In \textbf{Section~\ref{sec:weil}}, we investigate the universal transgression mechanism within the contractible Weil algebra $W(\kk)$. We construct the universal Chern--Simons cochains $Tf_j \in W^{2j-1}(\kk)$ (Proposition~\ref{cs}) and evaluate their projection to absolute closed primitives $\eta_j = \pi(Tf_j)$ on the compact group $K$ (Proposition~\ref{primform}). We then demonstrate the static evaluation defect (Remark~\ref{rmk:psi_w_failure}), proving why naive algebraic homomorphic evaluations fail to produce basic primitives on $K/T$.

In \textbf{Section~\ref{sec:uni}}, we develop the dynamic transgression machinery on the flag manifold. We introduce the 1-parameter connection family $\theta_s = \theta_{\ttt} + s\theta_\m$, establish the linear curvature factorization $F_s = (1-s)(F_{\ttt} - s\theta_\m^2)$ (Remark~\ref{rmk:curvature_factorization}), and prove the Geometric Transgression Theorem (Theorem~\ref{thm:geometric_transgression}). We then prove the universal closed-form formula for $\alpha_j$ via Euler's Beta integral (Theorem~\ref{thm:general_primitive}) and evaluate the second primitive $\alpha_2$ (Corollary~\ref{cor:j2_primitive}) and third primitive $\alpha_3$ (Corollary~\ref{cor:j3_primitive}).

We subsequently investigate the non-commutative threshold at degree $j = 4$. We evaluate the fourth primitive $\alpha_4$ (Corollary~\ref{cor:j4_primitive}) and detail the algebraic mechanism of necklace splitting arising from $[F_{\ttt}, \theta_\m^2] \neq 0$. We then formalize the conceptual interpretation of transgression primitives as curvature-perturbed Maurer--Cartan forms.

In \textbf{Section~\ref{sec:root}}, we specialize to the rank-two root system $A_2$ ($\mathfrak{sl}_3(\CC)$). We construct the pictorial loop dictionary on the planar hexagonal lattice (Remark~\ref{loops0}), express the canonical curvature in coordinates, and realize the second primitive $\alpha_2 = \Omega_- - \Omega_+$ as oriented triangle loops. We prove that the relative differential $\dif_\m$ acts as an edge-fracturing operator that shatters triangles into parallelograms, verifying $\dif_\m \alpha_2 = -p_2(F_{\ttt})$ in coordinates and demonstrating the vanishing of $\alpha_3 \equiv 0$ (Remark~\ref{rmk:a2_cubic_triviality}).

In \textbf{Section~\ref{sec:homo}}, we resolve the homological problem. We establish the Weyl equivariance of the complex and construct the master generating function $\mathcal{F}(t,x)$ via the Frobenius characteristic map, verifying the $k=3$ Betti numbers and exposing the additive multiplicity barrier at $k=4$ (Example~\ref{k4decomp}). We then introduce the Koszul differential graded algebra $K_\bullet(p_2, \dots, p_k)$, prove that the restricted invariants form a regular sequence (Proposition~\ref{regseqH}), and derive the Poincaré polynomial and Euler characteristic of the flag manifold via the Euler--Poincaré principle (Proposition~\ref{prop:poincare_flag}), concluding with Kostant's harmonic highest-weight solution (Theorem~\ref{kolialco}).

\section{Geometry of the Canonical Fibration \texorpdfstring{$K \to K/T$}{K -> K/T}}\label{sec:geocan}

\subsection{Reductive Splitting and Invariant Differential Forms}{\color{blue}\label{sec:reductive_splitting}}

Let $K = \SUk$ and let $T \subset K$ be the maximal torus of diagonal unitary matrices. The flag manifold $K/T$ is a reductive homogeneous space whose de Rham cohomology is canonically computed by the subcomplex of left-$K$-invariant, right-$T$-basic differential forms. We fix the standard Lie-theoretic setup:
\begin{equation}\label{setup}
     K=\SUk\quad T=K\cap B= \mathrm{S}(\mathrm{U}(1)^k)\quad \kk=\suk\quad \ttt=\mathrm{s}((i\RR)^k),
\end{equation}
where $\dim_\CC(\slk)=k^2-1$, 
$\dim_\RR(\kk)=k^2-1$, and $\dim_\RR(\ttt)=k-1$. Consequently $\dim_\RR(\kk/\ttt)=(k^2-1)-(k-1)=k(k-1)$.

The Killing form on $\slk$ and therefore also on $\kk=\suk\subset\slk$ is given by $\killing{\xi_1}{\xi_2}=2k\tr{\xi_1\xi_2}$ and can be used to identify the Lie algebras with their duals.\\
By Chapter 9 Theorem 6 in \cite{onishchik1994topology}, the Chevalley-Eilenberg complex of $\kk$ takes the following shape:
\[(\kk^\vee)^{T}_{\ttt}\xrightarrow{\dif_1}\left(\Lambda^2(\kk^\vee)_{\ttt}\right)^{T}\xrightarrow{\dif_2}\quad\cdots\quad\xrightarrow{\dif_{n-1}}\left(\Lambda^n (\kk^\vee)_{\ttt}\right)^{T}\xrightarrow{\dif_n}\quad\cdots\quad ,\]
with the degree $n$ Chevalley-Eilenberg differential $\dif_n$ given explicitly, for $\omega\in\left(\Lambda^n (\kk^\vee)_{\ttt}\right)^{T}$ and $\xi_1,\dots,\xi_{n+1}\in\kk$, by the formula 
\begin{equation}\label{CEdifdef}
	\dif\omega(\xi_1,\dots,\xi_{n+1})=\sum_{i<j}(-1)^{i+j}\omega([\xi_i,\xi_j],\xi_1,\dots,\hat{\xi_i},\dots,\hat{\xi_j},\dots,\xi_{n+1}) ,
\end{equation}
see Corollary 7.7.3 in \cite{Weibel_1994}.

\begin{remark}[Maurer--Cartan]\label{mc1}
    The differential $\dif_1$ is the dualization on forms of the Lie bracket on vector fields, and the higher degree differentials come from the Maurer--Cartan equation $\dif\theta+\frac{1}{2}[\theta\wedge\theta]=0$. %
\end{remark}

By using the isomorphisms  
\begin{equation}\label{isos}
     (\kk^\vee)^{T}_{\ttt}\cong((\kk/\ttt)^\vee)^{T}\cong((\ttt^\perp)^\vee)^{T} 
\end{equation}
 and putting $\m:=\ttt^\perp$ the complex can also be written as follows:
\begin{equation}\label{invcomp}          (\m^\vee)^{T}\xrightarrow{\dif_1}\left(\Lambda^2\m^\vee\right)^{T}\xrightarrow{\dif_2}\quad\cdots\quad\xrightarrow{\dif_{n-1}}\left(\Lambda^n \m^\vee\right)^{T}\xrightarrow{\dif_n}\quad\cdots\quad . 
\end{equation}

\begin{remark}[$k=2$]
    In the special case
    $k=2$, $K/T\cong\cp{1}$ is a type AIII symmetric space, with the involution of $K$ given by conjugation with the matrix $
    \begin{psmallmatrix}
        i & 0\\
        0 & -i
    \end{psmallmatrix}
    $. In that case, Theorems 6, 7 and the corresponding Corollary in \S9.10 of \cite{onishchik1994topology} imply that  all differentials vanish and that for all $n$ the spaces in the complex of invariant forms are identical to the cohomology spaces: $\prescript{K}{}(\Omega^n(\orb{e}(K)))\cong H^n(\orb{e}(K),\RR)$. \\
    However, for $k\neq2$ we do not have a symmetric space and cannot apply these arguments. 
    For $k=2$, $\alpha_2 \equiv 0$ vanishes for degree reasons ($\Lambda^3 \m^\vee = \{0\}$), confirming that non-trivial transgression is an intrinsically non-symmetric phenomenon occurring for $k \ge 3$.
\end{remark}

\begin{proposition}[Equivalence of Invariant de Rham and Chevalley--Eilenberg Cohomology \cite{Chevalley:1948zz, onishchik1994topology}]\label{prop:ce_derham_equiv}
    For a homogeneous space $G/H$ with connected stabilizer $H$, the cohomology of the complex of invariant differential forms is the relative Chevalley-Eilenberg cohomology:
    \[H^n_\text{inv}(\g/\h)\cong H^n_\text{CE}(\g,\h) ,\]
    where the left hand side is the cohomology computed via the complex of invariant alternating forms, and the right hand side is the above mentioned relative Chevalley-Eilenberg cohomology.
\end{proposition}

 A differential form $\omega \in \Omega^\bullet(K)$ descends to a smooth form on the flag manifold $K/T$ if and only if it is basic with respect to the principal $T$-action, meaning it is horizontal (i.e., $\iota_A \omega = 0$ for all fundamental vertical vector fields $A \in \ttt$) and invariant (i.e., $R_t^* \omega = \omega$ for all $t \in T$). More precisely, the pullback along the bundle projection defines an isomorphism of differential graded algebras:
\[ \pi_K^* \colon \Omega^\bullet(K/T) \xrightarrow{\ \sim\ } \Omega^\bullet_\mathrm{bas}(K) . \]

Recall the canonical isomorphism between left-invariant differential forms on $K$ and the exterior algebra given by evaluation at the identity:

$$\mathrm{ev}_1 \colon \prescript{K}{}{\Omega^\bullet(K)} \xrightarrow{\ \sim\ } \Lambda^\bullet\kk^\vee .$$

Under this map, the geometric condition of horizontality ($\iota_A \omega = 0$ for all vertical generators $A \in \ttt$) translates into the algebraic condition that the multilinear form vanishes on $\ttt$. In terms of the Maurer--Cartan splitting $\theta_\mathrm{MC} = \theta_{\ttt} + \theta_{\m}$, this condition is enforced by the soldering form: because $\ker(\theta_{\m})_k = \dif L_k(\ttt)$ (as established in Proposition~\ref{Q2c0} below), any form expressed purely in terms of $\theta_{\m}$ is automatically horizontal, while $\theta_{\ttt}$ detects the vertical components along the fibers.

Restricting $\mathrm{ev}_1$ to basic forms and imposing right $T$-invariance yields the chain of isomorphisms:

$$\prescript{K}{}{\Omega^\bullet(K)_\mathrm{bas}} = \prescript{K}{}{\Omega^\bullet(K)_{\ttt}}^T \xrightarrow[\sim]{\mathrm{ev}_1\mid} (\Lambda^\bullet\kk^\vee)_{\ttt}^T \cong (\Lambda^\bullet\m^\vee)^T ,$$

recovering the relative Chevalley--Eilenberg complex (\ref{invcomp}) that computes the invariant de Rham forms $\Omega^\bullet(K/T)^K$.

In Subsection~\ref{sec:bundle_geometry}, we examine the principal bundle geometry of the canonical fibration $K \to K/T$, decomposing the Maurer--Cartan form into the vertical connection $\theta_{\ttt}$ and the horizontal soldering form $\theta_\m$.

\subsection{Connections, Soldering Forms, and Curvature}{\color{blue}\label{sec:bundle_geometry}}

The relative Chevalley--Eilenberg complex $(\Lambda^\bullet \m^\vee)^T$ acquires its geometric meaning through the differential geometry of the canonical principal $T$-bundle $\pi_K \colon K \to K/T$. We now examine how the reductive splitting $\kk = \ttt \oplus \m$ decomposes the Maurer--Cartan form into the vertical canonical connection $\theta_{\ttt}$ and the horizontal soldering form $\theta_\m$.
	
Because $K = \SUk$ and its maximal torus $T$ are compact real Lie groups, all differential forms on this smooth fibration take values in the real Lie algebras $\kk = \mathrm{Lie}(K)$ and $\ttt = \mathrm{Lie}(T)$, while their complexifications $\g = \kk_\CC = \slk$ and $\h = {\ttt}_\CC$ govern the polynomial invariants and complex cohomology. Decomposing the global Maurer--Cartan form $\theta_\mathrm{MC} \in \Omega^1(K, \kk)$ along the reductive splitting $\kk = {\ttt} \oplus \m$ gives:
\[\theta_\mathrm{MC} = \theta_{\ttt} + \theta_\m \in \Omega^1(K, \ttt \oplus \m) \text{}.\]
As we shall see with Proposition \ref{Q2c0} below, this decomposition separates the geometry of the fibration into a vertical connection $\theta_{\ttt} := \operatorname{pr}_{\ttt}(\theta_\mathrm{MC})$ along the fibers and a horizontal soldering form $\theta_\m := \operatorname{pr}_\m(\theta_\mathrm{MC})$ across the base.

\begin{proposition}\label{Q2c0}
	Let $\pi_K \colon K \to K/T$ be the principal $T$-bundle defined by the canonical projection, and let $\theta_\mathrm{MC} = \theta_{\ttt} + \theta_\m \in \Omega^1(K, \kk)$ be the decomposition of the Maurer--Cartan form along the reductive splitting $\kk = \ttt \oplus \m$. Then:
	\begin{enumerate}[(i)]
		\item $K \to K/T$ is isomorphic to a reduction of the principal frame bundle $\mathrm{L}(K/T) \to K/T$ to the structure group $T$ via the isotropy representation $\lambda = \Ad\vert{}_T \colon T \to \mathrm{GL}(\m)$. The reduction is given explicitly by the bundle embedding
		\[f \colon K \hookrightarrow \mathrm{L}(K/T), \qquad k \longmapsto \big( X \in \m \mapsto \dif\pi_K(\dif L_k(X)) \in \mathrm{T}_{kT}(K/T) \big) \text{};\]
		
		\item $\theta_{\ttt} := \operatorname{pr}_{\ttt}(\theta_\mathrm{MC}) \in \Omega^1(K, \ttt)$ is the canonical connection on $K \to K/T$, with curvature
		\[F_{\ttt} = \dif \theta_{\ttt} = -\frac{1}{2}[\theta_\m, \theta_\m]_{\ttt} \in \Omega^2(K, \ttt) \text{};\]
		
		\item $\theta_\m := \operatorname{pr}_\m(\theta_\mathrm{MC}) \in \Omega^1_\mathrm{hor}(K, \m)^{(T, \Ad(T))}$ is a soldering form on $K \to K/T$; more precisely, it coincides with the pullback of the canonical soldering form $\sigma$ on $\mathrm{L}(K/T)$ along the reduction map $f$:
		\[f^*\sigma = \theta_\m \text{}.\]
	\end{enumerate}
	In particular, we may refer to $\theta_{\ttt}$ as the vertical (connection) component, and $\theta_\m$ as the horizontal (soldering) component of $\theta_\mathrm{MC}$.
\end{proposition}

\begin{proof}
	Recall that $K/T$ is a reductive homogeneous space in the sense that there is a splitting of $\kk = \mathrm{Lie}(K)$ into $\Ad(T)$-subrepresentations: $\kk = \ttt \oplus \m$. The proposition can now be derived from several standard properties of principal bundles.
	\begin{enumerate}[(i)]
		\item 
		For any reductive homogeneous space $K/T$, the canonical isomorphism $\mathrm{T}(K/T) \cong K \times_{\Ad(T)} \m$ induces an associated bundle isomorphism for the frame bundle $\mathrm{L}(K/T) \cong K \times_{\widehat{\Ad(T)}} \mathrm{GL}(\m)$ (see, e.g., \cite{kobayashi1969}). Embedding $K$ into the associated bundle via $k \mapsto [k, I]$ realizes $K$ as a reduction of $\mathrm{L}(K/T)$ to the structure group $T$ via the isotropy representation $\Ad|_T \colon T \to \mathrm{GL}(\m)$. Evaluating this identification fiberwise yields the embedding $f(k) = \dif\pi_K \circ \dif L_k|_\m \colon \m \xrightarrow{\sim} \mathrm{T}_{kT}(K/T)$.
		
		\item 
		The projection 
		$\theta_{\ttt} := \operatorname{pr}_{\ttt} \circ \theta_{\mathrm{MC}}$ is right-invariant under $T$ because $R_t^* \theta_{\mathrm{MC}} = \Ad(t^{-1}) \circ \theta_{\mathrm{MC}}$, reductivity makes $\operatorname{pr}_{\ttt}$ commute with $\Ad(t^{-1})$, and $\Ad(t^{-1})$ acts trivially on the abelian Lie algebra $\ttt$. To compute the curvature, we apply the splitting $\theta_{\mathrm{MC}} = \theta_{\ttt} + \theta_{\m}$ to the Maurer--Cartan equation $\dif\theta_{\mathrm{MC}} + \frac{1}{2}[\theta_{\mathrm{MC}}, \theta_{\mathrm{MC}}] = 0$ before projecting onto $\ttt$:
		\[\dif\theta_{\m} + \dif\theta_{\ttt} = -\frac{1}{2}[\theta_{\m} + \theta_{\ttt}, \theta_{\m} + \theta_{\ttt}] \text{}.\]
		The left-hand side projects to $\dif\theta_{\ttt}$ because $\dif\theta_{\m}$ takes values strictly in $\m$ (so that $\operatorname{pr}_{\ttt}(\dif\theta_{\m}) = 0$). 
		
		The bracket on the right-hand side becomes
		\[[\theta_{\m}, \theta_{\m}] + [\theta_{\m}, \theta_{\ttt}] + [\theta_{\ttt}, \theta_{\m}] + [\theta_{\ttt}, \theta_{\ttt}] \text{},\]
		where the last term vanishes because $\ttt$ is abelian and the two middle terms live in $\m$ by reductivity ($[\ttt, \m] \subset \m$). Consequently, the Maurer--Cartan equation projects to
		\[\dif\theta_{\ttt} = -\frac{1}{2}[\theta_{\m}, \theta_{\m}]_{\ttt} \text{},\]
		and therefore
		\[F_{\ttt} := \dif\theta_{\ttt} + \frac{1}{2}[\theta_{\ttt}, \theta_{\ttt}] = \dif\theta_{\ttt} = -\frac{1}{2}[\theta_{\m}, \theta_{\m}]_{\ttt} \text{}.\]
		
		\item 
		Horizontality is obvious: the kernel of $\theta_{\m}$ on vertical vectors is by definition the subspace $\ttt$. As above, equivariance follows from reductivity, which implies that $\operatorname{pr}_{\m}$ and $\Ad(T)$ commute. 

		Finally, to see that $\theta_\m$ is the pullback of the canonical soldering form $\sigma \in \Omega^1(\mathrm{L}(K/T), \m)$, let $\xi \in \mathrm{T}_k K$. By definition of the frame embedding $f$, the frame $f(k) \colon \m \xrightarrow{\sim} \mathrm{T}_{kT}(K/T)$ satisfies $f(k)(X) = \dif\pi_K(\dif L_k(X))$. Using $\pi_{\mathrm{L}} \circ f = \pi_K$, the pullback evaluates to:
		\[ (f^*\sigma)_k(\xi) = f(k)^{-1}(\dif\pi_K(\xi)) = \operatorname{pr}_\m(\theta_\mathrm{MC}(\xi)) = \theta_\m(\xi) \text{}, \]
		which confirms that $f^*\sigma = \theta_\m$ and concludes the proof. 
	\end{enumerate}
\end{proof}

To understand the relationship between the exterior derivative $\dif_K$ on the total space and the relative differential $\dif_\m$ on the base, we decompose the left-invariant forms into a horizontal--vertical bigrading:
\[{}^K\Omega^n(K) \cong \bigoplus_{p+q=n} \left( \Lambda^p \m^\vee \otimes \Lambda^q \ttt^\vee \right) =: \bigoplus_{p+q=n} {}^K\Omega^{p,q}(K) \text{}.\]
In this grading, the canonical soldering form is purely horizontal, $\theta_\m \in {}^K\Omega^{1,0}(K, \m)$, while the connection form is vertical, $\theta_{\ttt} \in {}^K\Omega^{0,1}(K, \ttt)$.

Although $\theta_{\ttt}$ is a vertical $1$-form, its exterior derivative is purely horizontal: by the $\ttt$-projection of the Maurer--Cartan equation,
\[\dif_K \theta_{\ttt} = -\frac{1}{2}[\theta_\m, \theta_\m]_{\ttt} = F_{\ttt} \in {}^K\Omega^{2,0}(K, \ttt) \text{}.\]

In the proof of Proposition~\ref{Q2c0}(ii), we decomposed the Maurer--Cartan equation $\dif\theta_\mathrm{MC} + \frac{1}{2}[\theta_\mathrm{MC}, \theta_\mathrm{MC}] = 0$ along the reductive splitting $\kk = \ttt \oplus \m$ onto the vertical $\ttt$-component. Projecting it analogously onto $\m$ yields:
\[\dif_K \theta_\m = -\frac{1}{2}[\theta_\m, \theta_\m]_\m - [\theta_{\ttt}, \theta_\m] \text{}.\]
Here, $-\frac{1}{2}[\theta_\m, \theta_\m]_\m \in {}^K\Omega^{2,0}(K, \m)$ is purely horizontal, whereas $[\theta_{\ttt}, \theta_\m] \in {}^K\Omega^{1,1}(K, \m)$ is of mixed type, encoding the vertical variation of $\theta_\m$ along the $T$-fibers ($\iota_A[\theta_{\ttt}, \theta_\m] = [A, \theta_\m]$ for $A \in \ttt$). 

Passing to the relative complex $(\Lambda^\bullet \m^\vee)^T \cong {}^K\Omega^{\bullet, 0}(K)^T$ projects out all forms with non-zero vertical degree, i.e., $q > 0$. Consequently, the mixed term $[\theta_{\ttt}, \theta_\m]$ vanishes under the horizontal projection $\operatorname{pr}^{(2,0)}$, yielding the restricted structure equation for the soldering form:
\begin{equation}\label{solderred}
	\dif_\m \theta_\m = \operatorname{pr}^{(2,0)}(\dif_K \theta_\m) = -\frac{1}{2}[\theta_\m, \theta_\m]_\m \text{}.
\end{equation}

Equation~\eqref{solderred} shows that the relative differential $\dif_\m$ is determined entirely by the horizontal Lie bracket on $\m$, which governs the edge-fracturing operator on the root lattice in Section~\ref{sec:root}. In Subsection~\ref{sec:trace_ortho_cw}, we establish the trace orthogonality of the reductive splitting, providing the algebraic vanishing mechanism required for the Chern--Weil construction and Borel's isomorphism in Subsection~\ref{sec:borel_cw_descent}.

\subsection{Trace Orthogonality and Invariant Polynomials}\label{sec:trace_ortho_cw}

A crucial algebraic property for our upcoming boundary computations on this complex is the trace orthogonality of the splitting.

\begin{lemma}[Trace Orthogonality of the Reductive Splitting]\label{lem:trace_orthogonality}
	Let $\kk = \mathfrak{su}(k)$ with Cartan subalgebra $\ttt$ and reductive complement $\m = \ttt^\perp$ defined via the Killing form in~\eqref{isos}. Then $\ttt$ and $\m$ are orthogonal with respect to the matrix trace:
	\[ \tr(H \cdot X) = 0 \quad \text{for all } H \in \ttt \text{ and } X \in \m . \]
	Consequently, for any differential forms $\omega \in \Omega^p(K, \ttt)$ and $\eta \in \Omega^q(K, \m)$, their wedge-trace vanishes: $\tr(\omega \wedge \eta) \equiv 0$. In particular:
	\[ \tr(\theta_\m \wedge F_{\ttt}^m) = 0 \quad \text{for all } m \ge 1, \qquad \text{and} \qquad \tr(F_{\ttt} \wedge \dif_\m \theta_\m) = 0 . \]
\end{lemma}
\begin{proof}
	By the definition $\m := \ttt^\perp$ in~\eqref{invcomp}  and the relation $B(\xi_1, \xi_2) = 2k\tr(\xi_1\xi_2)$, the trace orthogonality $\tr(H \cdot X) = 0$ holds for all $H \in \ttt$ and $X \in \m$. 
	
	Recall that $\ttt$ consists of diagonal matrices, and so the orthogonal complement $\m$ consists of off-diagonal matrices. Recall further that in coordinates, taking powers of forms like $F_{\ttt}$ corresponds to wedging the form components and multiplying the matrix components. If we think of $\ttt$-valued forms as diagonal matrices whose components are 1-forms, the power corresponds to the diagonal matrix whose entries are the wedge powers of the original form's component 1-forms.
	
	 For any $m \ge 1$, the matrix power $F_{\ttt}^m$ is diagonal, while $\theta_\m$ and $\dif_\m\theta_\m = -\frac{1}{2}[\theta_\m, \theta_\m]_\m$ take values in $\m$ and are off-diagonal. Since the product of a diagonal matrix and an off-diagonal matrix has vanishing diagonal entries, the traces vanish.
\end{proof}

We are working with the principal $T$-bundle $\pi_K \colon K \to K/T$, where $T = \mathrm{S}(\mathrm{U}(1)^k) \cong (S^1)^{k-1}$. Recall from (\ref{setup}) that $\ttt := \mathrm{Lie}(T) = \mathrm{s}((i\RR)^k)$. An element of $\ttt$ is represented by a diagonal matrix with weights summing to zero: $\mathrm{diag}(L_1, \dots, L_k)$ with $\sum L_j = 0$. In particular, $\ttt^\vee$ and its complexification $\h^\vee$ are of (real, respectively complex) dimension $k-1$ and spanned by the weight functionals $L_1, \dots, L_{k-1}$.

Each integral weight functional $\lambda \in \ttt^\vee$ defines a character $\chi_\lambda \colon T \to \mathrm{U}(1)$ by $\chi_\lambda(\exp \xi) = e^{\lambda(\xi)}$.
Consider the associated complex line bundle:
\[ \mathbb{L}_\lambda := K \times_{\chi_\lambda} \CC \longrightarrow K/T , \]
where the representation space $\CC$ can be identified with the root space $\slk_\lambda$ when $\lambda$ is a root.

Equipping $K \to K/T$ with the canonical connection $\theta_{\ttt}$ and its curvature $F_{\ttt} = -\frac{1}{2}[\theta_{\m}, \theta_{\m}]_{\ttt}$ from Proposition \ref{Q2c0}, the induced curvature form on $\mathbb{L}_\lambda$ is $\lambda(F_{\ttt}) = \dif(\lambda(\theta_{\ttt})) \in \Omega^2(K/T, i\RR)$. We therefore obtain:
\begin{equation}\label{c1Llambda}
	c_1(\mathbb{L}_\lambda) = -\frac{1}{2\pi i}[\lambda(F_{\ttt})] \in H^2(K/T; \CC) .
\end{equation}

By standard Chern--Weil theory (see, e.g., \cite{morita2001}), the connection $\theta_{\ttt}$ on the principal $T$-bundle $K \to K/T$ induces the real Chern--Weil homomorphism:
\[ cw_K \colon S^d(\ttt^\vee)^T \longrightarrow H^{2d}_\mathrm{dR}(K/T; \RR), \qquad p \longmapsto [p(F_{\ttt})] . \]
Because the maximal torus $T$ is abelian, the adjoint action $\Ad(T)$ acts trivially on $\ttt = \mathrm{Lie}(T)$, so every polynomial on $\ttt$ is automatically $T$-invariant: $S(\ttt^\vee)^T = S(\ttt^\vee)$. Tensoring with $\CC$ over $\RR$ yields on the domain $S(\ttt^\vee) \otimes_\RR \CC \cong S(\ttt_\CC^\vee) = S(\h^\vee)$, and on the codomain $H^*_\mathrm{dR}(K/T; \RR) \otimes_\RR \CC \cong H^*_\mathrm{dR}(K/T; \CC)$ via the de Rham theorem. Extending $cw_K$ by $\CC$-linearity defines the graded $\CC$-algebra homomorphism 
\[cw_K^\CC \colon S(\h^\vee) \to H^*_\mathrm{dR}(K/T; \CC) ,\]
 mapping polynomials of algebraic degree $d$ to de Rham classes of degree $2d$.

On linear polynomials $\lambda \in S^1(\h^\vee) \cong \h^\vee$, the complexified Chern--Weil homomorphism associated with the connection $\theta_{\ttt}$ evaluates to
\begin{equation}\label{cw_lambda}
	cw_{K}^\CC(\lambda) = \left[ \lambda(F_{\ttt}) \right] = -2\pi i \cdot c_1(\mathbb{L}_\lambda) \in H^2(K/T; \CC) .
\end{equation}
Because the weight functionals $L_1, \dots, L_{k-1}$ generate the polynomial algebra $S(\h^\vee)$ in degree $1$, their images $cw_{K}^\CC(L_j) = -2\pi i \cdot c_1(\mathbb{L}_{L_j})$ generate the image of $cw_{K}^\CC$ in degree $2$. Since $-2\pi i \neq 0$ and the cohomology ring $H^*(K/T; \CC)$ is generated in degree $2$ by the first Chern classes $c_1(\mathbb{L}_{L_j})$, the homomorphism $cw_{K}^\CC$ is surjective.

To determine the ring structure of $H^*(K/T; \CC)$, it therefore remains to identify the kernel of $cw_{K}^\CC$, which brings us to the geometric realization of Borel's isomorphism.

In Subsection~\ref{sec:borel_cw_descent}, we prove that this kernel is precisely the Weyl-invariant ideal $I_W = \langle p_2, \dots, p_k \rangle$, completing the Chern--Weil derivation of Borel's isomorphism.

\subsection{The Chern--Weil Homomorphism and Borel's Isomorphism}\label{sec:borel_cw_descent}

\begin{proposition}\label{oddvanlem}
	The odd cohomology of the flag manifold vanishes:
	\[H^{\mathrm{odd}}_\mathrm{dR}(K/T; \CC) = 0 \text{}.\]
\end{proposition}

\begin{proof}
	Under the Bruhat decomposition, the complex flag manifold decomposes into an even-dimensional cell complex whose cells (the Schubert cells) have real dimension $2\ell(w)$ for $w \in W$ (see, e.g., \cite{Borel1953SurLC, bott1982}):
	\[K/T \cong G/B = \coprod_{w \in W} BwB/B.\]
	Because there are no odd-dimensional cells, all odd cellular chain groups vanish, so $H^{\mathrm{odd}}_\mathrm{dR}(K/T; \CC) = 0$. Alternatively, this vanishing can be deduced from finite-dimensional Hodge theory on $(\Lambda^\bullet \m_\CC^\vee)^T$ via Kostant's harmonic weights (as established in Subsection~\ref{sec:kostant_hodge}).
\end{proof}

\begin{theorem}[Borel's Isomorphism via Chern--Weil Theory]\label{thm:borel_isomorphism}
	Let $\pi_K \colon K \to K/T$ be the canonical principal $T$-bundle equipped with the canonical connection $\theta_{\ttt} = \operatorname{pr}_{\ttt}(\theta_\mathrm{MC}) \in \Omega^1(K, \ttt)$ and curvature $F_{\ttt} = -\frac{1}{2}[\theta_{\m}, \theta_{\m}]_{\ttt} \in \Omega^2(K, \ttt)$. 
	
	The complexified Chern--Weil homomorphism associated with $\theta_{\ttt}$,
	\begin{align*}
		cw_{K}^\CC \colon S(\h^\vee) &\longrightarrow H^*_\mathrm{dR}(K/T; \CC) \\
		p &\longmapsto [p(F_{\ttt})], \text{}
	\end{align*}
	satisfies the following properties:
	\begin{enumerate}[(i)]
		\item The kernel of $cw_{K}^\CC$ is the Weyl-invariant ideal generated by the positive-degree invariant polynomials:
		\[\ker(cw_{K}^\CC) = I_W := S^+(\h^\vee)^W \cdot S(\h^\vee) = \langle p_2, \dots, p_k \rangle \text{},\]
		where $p_j = \mathrm{res}_\h(f_j)$ are the restricted fundamental trace invariants $f_j(X) = \tr(X^j)$.
		
		\item Under the degree-doubling convention $\deg(p) = 2\deg_\mathrm{poly}(p)$ for $p \in S(\h^\vee)$, $cw_{K}^\CC$ descends to an isomorphism of graded $\CC$-algebras:
		\[\overline{cw}_K^\CC \colon S(\h^\vee)/I_W \xrightarrow{\ \cong\ } H^*_\mathrm{dR}(K/T; \CC) \text{},\]
		recovering Borel's presentation of the cohomology of the flag manifold as the coinvariant algebra.
		
		\item The Chern--Weil map $cw_K^\CC$ is surjective, and the cohomology ring $H^*_\mathrm{dR}(K/T; \CC)$ is generated in degree $2$ by the first Chern classes of the associated line bundles: $c_1(\mathbb{L}_{L_j}) = -\frac{1}{2\pi i}cw_K^\CC(L_j)$ for $j = 1, \dots, k-1$.
	\end{enumerate}
\end{theorem}

\begin{proof}
	We establish statements (i)--(iii) in order:
	\begin{enumerate}[(i)]
		\item To prove the identification of the kernel $\ker(cw_K^\CC) = I_W$, we proceed in three steps: first, we establish the inclusion $I_W \subseteq \ker(cw_K^\CC)$ via the splitting of the trivial bundle associated with the fundamental representation; second, we prove that the induced map on the coinvariant algebra $\overline{cw}_K^\CC \colon S(\h^\vee)/I_W \to H^*_\mathrm{dR}(K/T; \CC)$ is injective by evaluating it on the Weyl discriminant; third, we conclude via dimension counting that $\overline{cw}_K^\CC$ is an isomorphism, which forces $\ker(cw_K^\CC) = I_W$.
		
		\paragraph{Step 1: Inclusion ($I_W \subseteq \ker(cw_K^\CC)$).}
		Consider the standard representation of $K = \SUk$ on $V = \CC^k$, and form the associated complex vector bundle over the flag manifold:
		\[E := K \times_T \CC^k \longrightarrow K/T \text{}.\]
		The map $\Phi \colon K \times_T \CC^k \to (K/T) \times \CC^k$ defined by $[k, v] \mapsto (kT, k \cdot v)$ is a smooth vector bundle isomorphism over $K/T$. Indeed, for any $t \in T \subset K$, we have $[kt, t^{-1}v] \mapsto (ktT, kt(t^{-1}v)) = (kT, kv)$, so $\Phi$ is well-defined, smooth, and fiberwise a linear isomorphism. Because $E$ is trivial, all its positive-degree Chern classes vanish:
		\[c(E) = 1 \in H^*_\mathrm{dR}(K/T; \CC) \text{}.\]
		To see this, view the trivial bundle as the pullback of a vector bundle over a point and use the functoriality of the Chern class construction to commute with the pullback along the point map.
			
		On the other hand, restricting the representation of $K$ on $\CC^k$ to the maximal torus $T = \mathrm{S}(\mathrm{U}(1)^k)$ diagonalizes $V$ into the direct sum of 1-dimensional weight spaces:
		\[\CC^k = \bigoplus_{j=1}^k \CC_{L_j} \text{},\]
		where $L_j \in \ttt^\vee$ again are the coordinate weight functionals. Since the associated bundle construction commutes with direct sums,
		$E$ splits into the Whitney sum of the line bundles $\mathbb{L}_{L_j} \to K/T$:
		\[E \cong \bigoplus_{j=1}^k \mathbb{L}_{L_j} \text{}.\]

		Applying the Whitney sum formula to this decomposition yields:
		\[c(E) = \prod_{j=1}^k c(\mathbb{L}_{L_j}) = \prod_{j=1}^k \big(1 + c_1(\mathbb{L}_{L_j})\big) = 1 + \sum_{m=1}^k e_m\big(c_1(\mathbb{L}_{L_1}), \dots, c_1(\mathbb{L}_{L_k})\big) \text{},\]
		where $e_m$ denotes the $m$-th elementary symmetric polynomial. Because $c(E) = 1$, the characteristic classes corresponding to the elementary symmetric polynomials vanish for all $m = 1, \dots, k$:
		\[e_m\big(c_1(\mathbb{L}_{L_1}), \dots, c_1(\mathbb{L}_{L_k})\big) = 0 \in H^{2m}_\mathrm{dR}(K/T; \CC) \text{}.\]

		By equation (\ref{cw_lambda}), evaluating the Chern--Weil map on the linear generators yields $cw_K^\CC(L_j) = -2\pi i \cdot c_1(\mathbb{L}_{L_j})$. Because $cw_K^\CC \colon (S(\h^\vee), \cdot) \to (H^*_\mathrm{dR}(K/T; \CC), \cup)$ is a homomorphism of graded $\CC$-algebras, evaluating it on the elementary symmetric polynomial $e_m \in S^m(\h^\vee)$ gives:
		\[cw_K^\CC(e_m(L_1, \dots, L_k)) = (-2\pi i)^m e_m\big(c_1(\mathbb{L}_{L_1}), \dots, c_1(\mathbb{L}_{L_k})\big) = 0 \text{},\]
		where the polynomial expression $e_m(c_1(\mathbb{L}_{L_1}), \dots, c_1(\mathbb{L}_{L_k}))$ is evaluated in cohomology using the cup product. By Newton's identities, the power sums $p_m = \sum_{j=1}^k L_j^m = \mathrm{res}_\h(f_m)$ generate the same invariant subring as the elementary symmetric polynomials: $S(\h^\vee)^W = \CC[p_2, \dots, p_k]$. Consequently, $cw_K^\CC(p_m) = 0$ for all $m = 2, \dots, k$, which proves that the entire Weyl-invariant ideal satisfies:
		\[I_W = \langle p_2, \dots, p_k \rangle \subseteq \ker(cw_K^\CC) \text{}.\]
		
		\paragraph{Step 2: Injectivity of $\overline{cw}_K^\CC$.}
		Because $I_W \subseteq \ker(cw_K^\CC)$, the Chern--Weil homomorphism factors through the coinvariant algebra to define a graded algebra homomorphism:
		\[\overline{cw}_K^\CC \colon S(\h^\vee)/I_W \longrightarrow H^*_\mathrm{dR}(K/T; \CC) \text{}.\]
		To prove that $\overline{cw}_K^\CC$ is injective, we analyze the ring structure of the coinvariant algebra $A := S(\h^\vee)/I_W$. By a statement of Chevalley from \cite{cheval} also used in the proof of Proposition~\ref{regseqH} below, $S(\h^\vee)$ is a free module over $S(\h^\vee)^W$ of rank $\vert{}W\vert{} = k!$.

		By Theorem 1.5 of Kane \cite{Kane_1994}, a finite linear group in characteristic zero is generated by pseudoreflections if and only if its ring of coinvariants satisfies Poincaré duality. Since the Weyl group $W \cong S_k$ is a reflection group acting on $\h^\vee$, the coinvariant algebra $A := S(\h^\vee)/I_W$ satisfies Poincaré duality in the sense of equations (1.3) and (1.4) of \cite{Kane_1994}: there exists an integer $N = \vert{}\Phi^+\vert{} = \frac{1}{2}k(k-1)$ such that $A^N \cong \CC$, $A^i = 0$ for all $i > N$, and the multiplication pairing
		\[A^d \otimes A^{N-d} \longrightarrow A^N \cong \CC \text{}\]
		is non-singular for each $0 \le d \le N$. Furthermore, as shown on page 87 of \cite{Kane_1994}, the non-trivial top-degree element is given by the product $\Omega = \prod_s a_s$ over all reflections $s \in W$, where each $a_s \in \h^\vee$ is a linear form defining the corresponding reflection hyperplane. For the Weyl group $W$ acting on $\h^\vee$, these hyperplanes are defined by the positive roots $\alpha \in \Phi^+$, so $\Omega$ is given by the Weyl discriminant:
		\[\mathcal{D} := \prod_{\alpha \in \Phi^+} \alpha \in S^N(\h^\vee) \text{}.\]
		By the non-singularity of the Poincaré duality pairing (condition (1.4) in \cite{Kane_1994}), every non-zero ideal of $A$ contains the 1-dimensional top component $A^N = \CC \cdot \mathcal{D}$. Therefore, the kernel $\ker(\overline{cw}_K^\CC)$ is trivial if and only if it does not contain $\mathcal{D}$:
		\[\ker(\overline{cw}_K^\CC) = \{0\} \iff \overline{cw}_K^\CC(\mathcal{D}) \neq 0 \in H^{2N}_\mathrm{dR}(K/T; \CC) \text{}.\]
		
		We proceed to show that $\overline{cw}_K^\CC(\mathcal{D})\neq0$. Evaluating $\overline{cw}_K^\CC$ on the Weyl discriminant $\mathcal{D} = \prod_{\alpha \in \Phi^+} \alpha$ using equation (\ref{cw_lambda}) yields:
		\begin{equation}\label{cwev}
			 \overline{cw}_K^\CC(\mathcal{D}) = \prod_{\alpha \in \Phi^+} \overline{cw}_K^\CC(\alpha) = \prod_{\alpha \in \Phi^+} \big(-2\pi i \cdot c_1(\mathbb{L}_\alpha)\big) = (-2\pi i)^N \prod_{\alpha \in \Phi^+} c_1(\mathbb{L}_\alpha) \text{},
		\end{equation}
		where the product denotes the iterated cup product.			
		
		By Proposition \ref{Q2c0}(i), the reduction of the frame bundle $\mathrm{L}(K/T)$ to the structure group $T$ yields the canonical bundle isomorphism $\mathrm{T}(K/T) \cong K \times_{\Ad(T)} \m$. Intersecting the Cartan decomposition $\m_\CC \cong \bigoplus_{\alpha \in \Phi} \slk_\alpha$ with the real compact form $\kk = \mathfrak{su}(k)$ decomposes the real complement $\m$ into 2-dimensional $\Ad(T)$-invariant subspaces:
		\[\m = \bigoplus_{\alpha \in \Phi^+} \m_\alpha, \qquad \text{where } \m_\alpha := \m \cap (\slk_\alpha \oplus \slk_{-\alpha}) \text{}.\]
		
		Under the invariant complex structure $J$ on $K/T$, the holomorphic tangent space at the identity coset is given by $\mathrm{T}^{1,0}_{1T}(K/T) = \m^- = \bigoplus_{\alpha \in \Phi^+} \g_{-\alpha}$. Since each negative root space $\g_{-\alpha}$ is a 1-dimensional complex $T$-representation with character $\chi_{-\alpha}$, and the associated bundle functor commutes with direct sums (as noted in Step 1), the holomorphic tangent bundle decomposes into the Whitney sum of line bundles:
		\begin{equation}\label{vbdec}
			\mathrm{T}^{1,0}(K/T) \cong K \times_{\Ad(T)} \left( \bigoplus_{\alpha \in \Phi^+} \g_{-\alpha} \right) \cong \bigoplus_{\alpha \in \Phi^+} (K \times_{\chi_{-\alpha}} \CC) = \bigoplus_{\alpha \in \Phi^+} \mathbb{L}_{-\alpha} \text{}.
		\end{equation}
		The underlying real vector bundle of the holomorphic tangent bundle is the real tangent bundle:
		\[\left(\mathrm{T}^{1,0}(K/T)\right)^\RR=\mathrm{T}(K/T) \text{}.\]
		Using the fact that the top Chern class of a complex vector bundle equals the Euler class of the underlying real vector bundle (see, e.g., \cite{bott1982} 20.10.6), we can deduce:
		\begin{equation}\label{cxrvb}
			c_N(\mathrm{T}^{1,0}(K/T)) = e(\mathrm{T}(K/T)) \text{}.
		\end{equation}
		In combination with the Whitney sum formula, equations (\ref{vbdec}) and (\ref{cxrvb}) allow us to rewrite the Euler class as follows:
		\[e(\mathrm{T}(K/T)) = \prod_{\alpha \in \Phi^+} c_1(\mathbb{L}_{-\alpha}) = \prod_{\alpha \in \Phi^+} \big(-c_1(\mathbb{L}_\alpha)\big) = (-1)^N \prod_{\alpha \in \Phi^+} c_1(\mathbb{L}_\alpha) \text{},\]
		where we used the dual relationship $c_1(\mathbb{L}_{-\alpha}) = -c_1(\mathbb{L}_\alpha)$.
		Multiplying by $(-1)^N$ yields $\prod_{\alpha \in \Phi^+} c_1(\mathbb{L}_\alpha) = (-1)^N e(\mathrm{T}(K/T))$. Combining this with the evaluation on the discriminant from equation (\ref{cwev}), we obtain:
		\[\overline{cw}_K^\CC(\mathcal{D}) =  (2\pi i)^N e(\mathrm{T}(K/T)) \text{}.\]

		By the Gauss--Bonnet--Chern theorem, which we may apply because $K/T$ is compact, oriented (due to its almost-complex structure) and of even dimension $2N$,  evaluating the Euler class on the fundamental cycle $[K/T]$ yields the Euler characteristic:
		\[\int_{K/T} \overline{cw}_K^\CC(\mathcal{D}) = (2\pi i)^N \int_{K/T} e(\mathrm{T}(K/T)) = (2\pi i)^N \chi(K/T) \text{}.\]
		By the classical theorem of Hopf and Samelson in \cite{HopfSamelson1940}, the Euler characteristic of the flag manifold $K/T$ equals the order of its Weyl group:
		\[\chi(K/T) = [N_K(T) : T] = \vert{}W\vert{} = k! > 0 \text{}.\]
	
		The integral then evaluates to:
		\[\int_{K/T} \overline{cw}_K^\CC(\mathcal{D}) = (2\pi i)^N k! \neq 0 \text{}.\]
		Consequently, the cohomology class $\overline{cw}_K^\CC(\mathcal{D})$ cannot be zero, which implies that $\ker(\overline{cw}_K^\CC) = \{0\}$, and hence that $\overline{cw}_K^\CC$ is injective.

		Because the induced map on the quotient $S(\h^\vee)/I_W$ has trivial kernel, we conclude that
		$\ker(cw_K^\CC) = I_W,$
		which establishes statement (i).
		
		\item To establish (ii), we begin by comparing the complex dimensions of its domain and codomain: As recalled from Chevalley in Step 2 above, the coinvariant algebra has dimension $\dim_\CC\big(S(\h^\vee)/I_W\big) = \vert{}W\vert{} = k!$.
		
		On the cohomology side, the Chern--Weil map assigns elements of $S^d(\h^\vee)$ to classes of real degree $2d$, so the image of $\overline{cw}_K^\CC$ is concentrated in even degrees:
		\[\operatorname{Im}(\overline{cw}_K^\CC) \subseteq H^{\mathrm{even}}_\mathrm{dR}(K/T; \CC) \text{}.\]
		Because $\overline{cw}_K^\CC$ is injective and the coinvariant algebra $S(\h^\vee)/I_W$ has complex dimension $\vert{}W\vert{} = k!$, the dimension formula implies:
		\[\dim_\CC H^{\mathrm{even}}_\mathrm{dR}(K/T; \CC) \ge \dim_\CC(\operatorname{Im}(\overline{cw}_K^\CC) ) = \vert{}W\vert{} \text{}.\]
		On the other hand, the Euler characteristic is by definition the alternating sum of the even and odd Betti numbers. Combining this with the value $\chi(K/T) = \vert{}W\vert{}$ established in Step 2 via \cite{HopfSamelson1940} yields:
		\[\vert{}W\vert{}=\chi(K/T) = \dim_\CC H^{\mathrm{even}}_\mathrm{dR}(K/T; \CC) - \dim_\CC H^{\mathrm{odd}}_\mathrm{dR}(K/T; \CC) \text{},\]
		which can be rearranged to:
		\[\dim_\CC H^{\mathrm{odd}}_\mathrm{dR}(K/T; \CC) = \dim_\CC H^{\mathrm{even}}_\mathrm{dR}(K/T; \CC) - \vert{}W\vert{} \text{}.\]
		Consequently, the injection $\overline{cw}_K^\CC$ is surjective onto $H^{\mathrm{even}}_\mathrm{dR}(K/T; \CC)$ (and thus defines an isomorphism onto the full cohomology ring) if and only if the odd cohomology of the flag manifold vanishes, $H^{\mathrm{odd}}_\mathrm{dR}(K/T; \CC) = 0,$ 
		which is Proposition \ref{oddvanlem} above.
		We conclude that $\dim_\CC H^{\mathrm{even}}_\mathrm{dR}(K/T; \CC) = \vert{}W\vert{} = k!$, so $\dim_\CC H^*_\mathrm{dR}(K/T; \CC) = \vert{}W\vert{}$.
		
		This implies that the induced map $\overline{cw}_K^\CC \colon S(\h^\vee)/I_W \to H^*_\mathrm{dR}(K/T; \CC)$ is an isomorphism of vector spaces. Because $\overline{cw}_K^\CC$ is a bijective homomorphism of $\CC$-algebras, its inverse is automatically an algebra homomorphism, making $\overline{cw}_K^\CC$ an isomorphism of $\CC$-algebras. Under the degree-doubling convention assigning degree $2d$ to $S^d(\h^\vee)$, $\overline{cw}_K^\CC$ maps $(S(\h^\vee)/I_W)_{2d}$ onto $H^{2d}_\mathrm{dR}(K/T; \CC)$ and preserves the grading degree-by-degree, establishing that it is an isomorphism of graded $\CC$-algebras. This proves (ii).
		
		\item Because the projection $\operatorname{pr} \colon S(\h^\vee) \twoheadrightarrow S(\h^\vee)/I_W$ is surjective and $\overline{cw}_K^\CC$ is an isomorphism, the composite $cw_K^\CC = \overline{cw}_K^\CC \circ \operatorname{pr}$ is surjective. Finally, because the polynomial algebra $S(\h^\vee)$ is generated in degree $1$ by $\h^\vee$, the isomorphic ring $H^*_\mathrm{dR}(K/T; \CC)$ is generated in degree $2$ by the images of the coordinate weight functionals: the first Chern classes $c_1(\mathbb{L}_{L_j}) = -\frac{1}{2\pi i}cw_K^\CC(L_j)$ for $j = 1, \dots, k-1$. This proves (iii).
	\end{enumerate}
\end{proof}

\begin{remark}[Newton's Identities and the Chern--Weil Map]\label{newton4}
	This geometric setup is closely related to Newton's identities relating power sums to elementary symmetric polynomials. The total Chern class of the splitting bundle $E$ is obtained via the determinant:
	\[c(E) = \left[ \det\left( I - \frac{1}{2\pi i} F_{\ttt} \right) \right] = 1 + c_1(E) + c_2(E) + \dots + c_k(E) \in H^*(K/T; \CC) \text{}.\]
	Here, each component $c_m(E)$ corresponds to the scaled $m$-th elementary symmetric polynomial $\left(-\frac{1}{2\pi i}\right)^m e_m$ evaluated on the curvature $F_{\ttt}$, which we can interpret as a diagonal matrix of real-valued 2-forms. Its eigenvalues are then the 2-forms $L_1(F_{\ttt}), \dots, L_k(F_{\ttt}) \in \Omega^2(K/T, i\RR)$.

	On the other hand, our algebraic generators in the Weil algebra and the coinvariant ideal $I_W$ are the invariant power-sum traces $p_j(X) = \operatorname{tr}(X^j)$. The bridge between these two geometric representations is the determinant-trace identity $\det (I+tf)=\exp\tr\log(I+tf)$ (cf.\ Lemma~\ref{lem:det_trace_id}). Substituting the curvature matrix $f = -\frac{1}{2\pi i}F_{\ttt}$ into this identity yields the generating series for the Chern classes:
	\[\sum_{m=0}^k c_m(E) t^m = \left[ \det\left( I - \frac{t}{2\pi i}F_{\ttt} \right) \right] = \left[ \exp\left( \sum_{j=1}^\infty \frac{(-1)^{j-1}}{j} \left(-\frac{1}{2\pi i}\right)^j \operatorname{tr}(F_{\ttt}^j) t^j \right) \right] \text{}.\]
	The left hand side has scaled elementary symmetric polynomials as coefficients, whereas the right hand side's coefficients are the power-sum polynomials. Expanding this exponential degree-by-degree in powers of $t$ therefore recovers Newton's identities in de Rham cohomology. Under this new viewpoint, we can interpret Newton's identities as converting the curvature traces $[\operatorname{tr}(F_{\ttt}^j)]$ into the Chern classes $c_m(E)$.
	
	This sheds further light on the link between our algebraic invariants and the geometry of $K/T$. The Chern--Weil map in Theorem \ref{thm:borel_isomorphism} is defined via the trace generators $p_j = \operatorname{tr}(X^j)$ and leads to $I_W = \langle p_2, \dots, p_k \rangle$ evaluated on the weight functionals $(L_1, \dots, L_k)$; and Newton's identities imply that this ideal coincides with $\langle e_2, \dots, e_k \rangle$, recalling the vanishing of which we observed upon evaluation on the positive Chern classes of the splitting bundle $E$ in Step~1 of the proof of (i).\\
\end{remark}

\begin{remark}[Topological Vanishing vs. Geometric Primitives]\label{rmk:transgression_geometric}
	While Step~1 of the proof of Theorem~\ref{thm:borel_isomorphism} establishes the ideal inclusion $I_W \subseteq \ker(cw_K^\CC)$ topologically via the Chern classes of the splitting bundle $E$, it does so non-constructively. It does not produce explicit differential-form primitives witnessing the exactness of the characteristic forms $p_j(F_{\ttt}) = \tr(F_{\ttt}^j)$ on the base manifold $K/T$. 
	
	Constructing these canonical odd-degree primitives explicitly on the flag manifold via relative transgression and identifying them with directed closed loops on the root lattice forms the central objective of the subsequent sections. In Section~\ref{sec:weil}, we demonstrate why algebraic projection within the Weil algebra fails to produce basic primitives on $K/T$, motivating the dynamic connection interpolation developed in Section~\ref{sec:uni}.
\end{remark}

\section{The Weil Algebra and the Absolute Primitives}\label{sec:weil}

\subsection{Universal Chern--Simons Forms in \texorpdfstring{$W(\kk)$}{W(k)}}\label{sec:cs}

In Section~\ref{sec:geocan}, we established Borel's isomorphism $H^*(K/T; \CC) \cong S(\h^\vee)/I_W$ topologically via the Chern classes of the splitting bundle, demonstrating that the characteristic forms $p_j(F_{\ttt}) = \tr(F_{\ttt}^j)$ vanish in cohomology. To construct canonical odd-degree differential forms witnessing this exactness directly on the base $K/T$, one must first understand the universal transgression mechanism in the Weil algebra $W(\kk)$ and identify the precise algebraic obstruction that prevents universal Chern--Simons forms from descending under static evaluation.

\begin{definition}[Weil Algebra]\label{def:weil-algebra}
    The \textit{Weil algebra} $W(\kk)$ of the compact Lie algebra $\kk = \mathfrak{su}(k)$ (with complexification $W(\g) \cong W(\kk) \otimes_\RR \CC$ for $\g = \mathfrak{sl}(k, \CC)$) is the graded commutative algebra
    \[W(\g):=S^\bullet(\g^\vee)\otimes\Lambda^\bullet(\g^\vee),\]
    where we assign the degree $1$ to the elements of $\{1\}\otimes\Lambda^\bullet(\g^\vee)$ and the degree $2$ to the elements of $S^\bullet(\g^\vee)\otimes\{1\}$.
\end{definition}

Let us introduce some notation for the generators. Given a form $l\in\g^\vee$, we define:
\[\check{l}:=1\otimes l\in W^1\g;\quad \Bar{l}:=l\otimes1\in W^2\g .\]

We will interpret degree $1$ elements in terms of connections and degree $2$ elements in terms of curvatures. Indeed, in the notation of Proposition 6.10 in \cite{Meinrenken2013}, we have $\check{l}=\theta_W(l)$. This is the definition of the so-called \textit{Weil connection}
\[\theta_W:\g^\vee\rightarrow W^1\g\]
whose curvature is given by 
\[F^{\theta_W}:\g^\vee\rightarrow W^2\g;\:l\mapsto\Bar{l}-1\otimes\dif_\wedge l=:\hat{l} .\]
This curvature variable $\hat{l}$ is computationally more convenient than $\Bar{l}$.

There are several differentials, or, more precisely, super derivations at play here. The differential whose restriction we used previously in our complex of invariant forms and which we called the Chevalley-Eilenberg differential, namely the one given in degree $1$ by the dual of the Lie bracket on $\g$, acts on the exterior subalgebra of $W\g$ and will be denoted $\dif_\wedge$ here. In particular, $\dif_\wedge:\g^\vee\rightarrow\Lambda^2\g^\vee$ is defined by the relation \[\langle d_\wedge l, \xi_1 \wedge \xi_2 \rangle = - \langle l, [\xi_1,\xi_2] \rangle .\] 
The \textit{Koszul differential} of $W\g$ maps odd generators to even generators, it is defined by
\begin{equation}\label{kosz}
    \dif_\text{K}\check{l}:=\hat{l},\quad \dif_\text{K}\hat{l}=0 .
\end{equation}
Consider the insertion operator $\iota_\xi$ for $\xi\in\g$. It lowers the exterior degree by $1$ without affecting the symmetric degree. We call an element of the Weil algebra \textit{horizontal} if it lies in the kernel of $\iota_\xi$ for all $\xi\in\g$.  Exterior elements are not horizontal: 
\[\iota_\xi(\check{l})=1\otimes\iota_\xi(l)=\langle l, \xi \rangle\cdot1 .\] However, symmetric elements are horizontal by definition.
We may therefore interpret (\ref{kosz}) as saying that $\dif_\text{K}$ implements a horizontal shift. 
The generalized Chevalley-Eilenberg differential on $W\g$ is defined by
\[\dif_{\text{CE}}\check{l}:=1\otimes\dif_\wedge l,\quad \dif_{\text{CE}}\hat{l}=\sum_i(\ad_{e_i}^\vee l)\otimes\omega^i ,\]
where $(e_i)_i$ denotes a basis of $\g$ and $(\omega^i)_i$ is its dual basis.
In particular the generalized Chevalley-Eilenberg differential implements the exterior differential on the exterior part and the coadjoint action on the symmetric part. In the latter part of the definition we consider $W\g$ as the Chevalley-Eilenberg complex of $\g$ with coefficients in the $\g$-module $S(\g^\vee)$.\\
Summing up these two super derivations yields the so called \textit{Weil differential}:
\[\dif_\text{W}:=\dif_\text{K}+\dif_\text{CE} .\]
It follows from the above that
\[\dif_\text{W}\check{l}=\Bar{l},\quad \dif_\text{W}\Bar{l}=0,\quad \dif_\text{W}\hat{l}=\dif_\text{CE}\hat{l} .\]
Observe that $\dif_\text{W}^2-\dif_\text{K}^2-\dif_\text{CE}^2=0$ implies that $\dif_\text{K}$ and $\dif_\text{CE}$ anticommute.\\
Another crucial tool is the projection $\pi:W\g\rightarrow\Lambda\g^\vee$ defined as
\[\pi(\check{l})=l,\quad\pi(\Bar{l})=0 ,\] and consequently $\pi(\hat{l})=-\dif_\wedge l$. It can easily be seen that the projection is a chain map w.r.t. $\dif_{\text{CE}}$ and $\dif_\wedge$. Note that it is not a chain map w.r.t. $\dif_W$ and $\dif_\wedge$, since $(\pi \circ \dif_W)(\check{l}) = \pi(\Bar{l}) = 0 \neq \dif_\wedge(l)=(\dif_\wedge\circ\pi)(\check{l})$ in general.

\begin{remark}[Towards Recovery of Maurer--Cartan]\label{mc3}
    Recall from above that the \textit{Weil connection} is the linear map $\theta_W: \g^\vee \to W^1\g$ defined by $\theta_W(l) := \check{l}$.
    We identify this map with the $\g$-valued form
    \[\boldsymbol{\theta}_W := \sum_{a} e_a \otimes \check{l^a}\in \g \otimes W^1\g ,\]
    where $\{e_a\}$ is a basis of $\g$ and $\{l^a\}$ its dual basis.\\
    Analogously, we interpret the curvature $F^{\theta_W}:\g^\vee\rightarrow W^2\g$ given by $F^{\theta_W}(l)=\hat{l}$ as the $\g$-valued curvature \[\F := \sum e_a \otimes \hat{l^a}\in \g \otimes W^2\g .\]
\end{remark}

\begin{lemma}[Weil Structural Equation]\label{weilselem}
    The Weil connection satisfies the following equation:
    \begin{equation}\label{structural_eq}
        \dw \boldsymbol{\theta}_W = \F - \frac{1}{2} [\boldsymbol{\theta}_W, \boldsymbol{\theta}_W] .
    \end{equation}
\end{lemma}
This is essentially Proposition 6.7 from \cite{Meinrenken2013}. 
\begin{proof}
    From our definitions we have that \[\hat{l}=\dif_\mathrm{K}(\check{l})=\dif_W(\check{l})-\dif_\mathrm{CE}(\check{l}) .\]
    Lifting this to the $\g$-valued forms, i.e. tensoring from the left with $e_a$ and summing over $a$, this becomes
    \begin{equation}\label{strucb}
        \F=\dw \boldsymbol{\theta}_W-\dce\boldsymbol{\theta}_W .
    \end{equation}
    The only thing left to show is that $\dce\boldsymbol{\theta}_W=- \frac{1}{2} [\boldsymbol{\theta}_W, \boldsymbol{\theta}_W]$. But this is exactly the same computation explained in Remark \ref{mc1}.
\end{proof}

By differentiating (\ref{structural_eq}), we also get the \textit{second Bianchi identity} $\dw\F=[\F,\boldsymbol{\theta}_W]$: 
\begin{align*}
    \dw\F&=\dw^2\boldsymbol{\theta}_W+\frac{1}{2}\dw[\boldsymbol{\theta}_W,\boldsymbol{\theta}_W]\\
    &=[\F-\frac{1}{2}\dw[\boldsymbol{\theta}_W,\boldsymbol{\theta}_W],\boldsymbol{\theta}_W]\\
    &=[\F,\boldsymbol{\theta}_W]-\frac{1}{2}[[\boldsymbol{\theta}_W,\boldsymbol{\theta}_W],\boldsymbol{\theta}_W] ,
\end{align*}
where the last summand vanishes due to the graded Jacobi identity.

\begin{proposition}[Recovery of Maurer--Cartan]\label{mc4}
    The image of the Weil connection under the map $\pi$ is the Maurer--Cartan form $\theta_\mathrm{MC} \in \g \otimes \Lambda^1\g^\vee$, and the structural equation projects  to the Maurer--Cartan equation:
    \begin{equation}\label{mceq}
    	\dif_\wedge\theta_\mathrm{MC}+\frac{1}{2}[\theta_\mathrm{MC},\theta_\mathrm{MC}]=0
    \end{equation}
\end{proposition}

\begin{proof}
    The fact that $\pi(\boldsymbol{\theta}_W)=\theta_\mathrm{MC}$ follows immediately from $\pi(\check{l})=l$.
    As an algebra morphism, $\pi$ preserves the wedge product and therefore the brackets in the sense that $\pi([\boldsymbol{\theta}_W, \boldsymbol{\theta}_W]) = [\pi(\boldsymbol{\theta}_W), \pi(\boldsymbol{\theta}_W)]$. Furthermore, since $(\pi \circ \dif_W)(\check{l}) = \pi(\Bar{l}) = 0$, applying $\pi$ to the left hand side of (\ref{structural_eq}) yields $0$. Finally, $\pi(\hat{l})=-\dif_\wedge l$ implies that $\pi(\F)=-\dif_\wedge(\sum_a{e_a}\otimes l^a)=-\dif_\wedge\theta_\mathrm{MC}$. In total, the structural equation (\ref{structural_eq}) is therefore indeed projected to $0=-\dif_\wedge\theta_\mathrm{MC}-\frac{1}{2}[\theta_\mathrm{MC},\theta_\mathrm{MC}]$, the Maurer--Cartan equation up to sign.
\end{proof}
\begin{remark}
    Alternatively, one could prove Proposition \ref{mc4} by applying $\pi$ to equation (\ref{strucb}) and exploiting the chain map property.
\end{remark}

Restricting the Weil algebra and its projection $\pi$ to $\g$-invariants establishes the transgression diagram relating polynomial coadjoint invariants to primitive exterior forms:
Let $S\g^\vee_+$ denote polynomials with vanishing constant term.
Reducing to the invariant parts of the involved algebras allows us to establish the following commutative diagram, see the beginning of 6.13 in \cite{Meinrenken2013}:
\begin{equation}\label{diagram}
    \begin{tikzcd}
    	& {\mathcal{J}\subset {(W\g)}^\g} \\
    	{(S\g^\vee_+)^\g} & {(\Lambda\g^\vee)^\g}
    	\arrow["{\dif_\text{W}\mid_{\mathcal{J}}}"', from=1-2, to=2-1]
    	\arrow["\pi", from=1-2, to=2-2]
    	\arrow["\tau", from=2-1, to=2-2]
    \end{tikzcd} 
\end{equation}

Recall that for any finite-dimensional vector space $V$, the exterior algebra $\Lambda V$ carries a canonical graded super Hopf algebra structure $(\Lambda V, m_0, i_0, \Delta_0, \varepsilon_0, s_0)$ whose coproduct $\Delta_0$ is uniquely determined by declaring the degree-$1$ generators to be primitive: $\Delta_0(v) = v \otimes 1 + 1 \otimes v$ for $v \in V$ (see Section~10.3.1 in \cite{Meinrenken2013}).

By Theorem 10.1 in \cite{Meinrenken2013}, we may also view the invariant exterior algebra $(\Lambda\g^\vee)^\g$ as a graded super Hopf algebra $((\Lambda\g^\vee)^\g,m,i,\Delta,\eps,s)$ with $m,i,\eps$ and $s$ being the restrictions of the corresponding objects $m_0,i_0,\eps_0$ and $s_0$ on $\Lambda\g^\vee$. Only the coproduct $\Delta_0$ must be adapted so that its image lies in $(\Lambda\g^\vee)^\g\otimes(\Lambda\g^\vee)^\g$: Let $\pi \colon \Lambda\g \to (\Lambda\g)^\g$ be the projection along $\g \cdot \Lambda\g$, and let $j \colon (\Lambda\g)^\g \hookrightarrow \Lambda\g$ be the natural inclusion. We define the restricted coproduct $\Delta$ via:
\[\Delta = (\pi \otimes \pi) \circ \Delta_0 \circ j .\] 
This dualizes the multiplication $m=\pi\circ m_0\circ(j\otimes j)$.

\begin{remark}[The Role of the Projection $\pi$]\label{rmk:role_of_pi}
    In the full exterior algebra $\Lambda\g$, the primitive elements for $\Delta_0$ are  the degree $1$ elements (the space $\g$ itself). Higher degree elements, such as our asolute primitives $\eta_j \in \Lambda^{2j-1}\g$, are not primitive under $\Delta_0$; their coproduct contains many mixed-degree terms (e.g., in $\Lambda^{2j-2}\g \otimes \g$). 
    
    However, for a semisimple Lie algebra, there are no invariant 1-forms ($\g^\g = 0$). Consequently, the projection $\pi \colon \Lambda\g \to (\Lambda\g)^\g$  annihilates these mixed-degree cross-terms. When evaluating $\Delta(\eta_j) = (\pi \otimes \pi)(\Delta_0(\eta_j))$, all intermediate tensor products vanish, leaving only $\eta_j \otimes 1 + 1 \otimes \eta_j$. The projection $\pi$ is the precise algebraic mechanism that promotes the higher-degree transgression cocycles into the primitive generators of the invariant Hopf algebra.
\end{remark}

The subset $\mathcal{J}$ in the diagram (\ref{diagram}) is defined as follows (see Definition 6.15 in \cite{Meinrenken2013}):
\begin{definition}[Cochains of Transgression]
    Given an element $p\in(S\g^\vee)^\g$, a cochain of transgression of $p$ is an odd element $C\in(W\g)^\g$ such that $\dw(C)=p$. We denote the set of cochains of transgression by $\mathcal{J}$.
\end{definition}
If we consider $(\Lambda\g^\vee)^\g$ with the trivial differential, the projection $\pi$ restricts to a morphism of graded $\g$-differential algebras.\\
The transgression map is defined according to Proposition 6.17 of \cite{Meinrenken2013}, which asserts the existence of a well-defined linear map $\tau:(S^+\g^\vee)^\g\rightarrow(\Lambda\g^\vee)^\g$ mapping $p\mapsto\tau(p)$, satisfying the following condition:
\[\forall C\in \mathcal{J}\text{ s.t. }p=\dif_W(C):\quad\tau(p)=\pi(C) .\]
It is shown further that if $\deg(p)=r>0$, then $\deg(\tau(p))=2r-1$ and that $\tau$ vanishes on the subspace $((S^+\g^\vee)^\g)^2$. 
For computations it is convenient to write $\tau=\pi\circ h$ for the map $h:W\g\rightarrow W\g$ defined as the composition $(I\otimes1)\circ k$ of the maps
\begin{align*}
    k: W\g&\longrightarrow\CC[t,\dif t]\otimes W\g\\
    \check{l}&\longmapsto t\otimes\check{l},\\
    \Bar{l}&\longmapsto \dif t\otimes\check{l}+t\otimes\Bar{l}
\end{align*}	
(where the assignment $\check{l}\mapsto t\otimes\check{l}$ is interpreted as homotopy $\theta_t=(1-t)\theta_0+t\theta_1$ between $\theta_0=0$ and $\theta_1=\theta_W$),
and
\[I\otimes1:\CC[t,\dif t]\otimes W\g \rightarrow W\g ,\]
where $I$ denotes integration in the variable $t$ over the unit interval. We easily verifiy that $h$ satisfies the homotopy relation
\begin{equation}\label{homotopy}
    \dw h + h\dw=\id
\end{equation}
on generators:
\begin{align*}
    h\dw(\check{l})+\dw h(\check{l})&=h(\Bar{l})+0=\int_0^1\dif t\cdot\check{l}=\check{l};\\
    h\dw(\Bar{l})+\dw h(\Bar{l})&=0+\dw(\check{l})=\Bar{l} .
\end{align*}

\begin{remark}\label{khat}
	From the definition of the map $k$ above, we can deduce its value on the curvature variables $\hat{l}$:
	\begin{align*}
		k(\hat{l})=k(\bar{l}-1\otimes\dif_\wedge l) &= \dif t\otimes\check{l}+t\otimes\bar{l}-t^2\otimes(1\otimes \dif_\wedge l)\\
		&=\dif t\otimes\check{l}+t\otimes\hat{l}+(t-t^2)\otimes(1\otimes \dif_\wedge l) ,
	\end{align*}
	where we used that $1\otimes \dif_\wedge l$ is quadratic in $\check{l}$ in order to extract the factor $t^2$.
\end{remark}

\begin{remark}
    The above decomposition $\tau=\pi\circ h$ uses the so-called standard homotopy operator $h$, but any $\g$-equivariant homotopy works since such homotopies map any $p\in(S^+\g^\vee)^\g$ to a cochain of transgression, see Chapters 5 and 6 of \cite{Meinrenken2013} for further details. In particular one can interpret $h$ as a deformation into $\mathcal{J}\subset(W\g)^\g$ (as is remarked just before Proposition 6.18 of \cite{Meinrenken2013}).
\end{remark}

Recalling (\ref{structural_eq}), we can write the curvature of the intermediate connection $\theta_t=t\theta_W$ as
\begin{equation}\label{ftw}
	\mathbf{F}_{tW}= \dw \boldsymbol{\theta}_{tW} + \frac{1}{2} [\boldsymbol{\theta}_{tW}, \boldsymbol{\theta}_{tW}]=t\dw \boldsymbol{\theta}_W + \frac{1}{2}t^2 [\boldsymbol{\theta}_{W}, \boldsymbol{\theta}_{W}]= t \F - \frac{1}{2}(t - t^2)[\boldsymbol{\theta}_W, \boldsymbol{\theta}_W] ,
\end{equation}
which, using the same logic as in the proof of Proposition \ref{mc4}, projects to
\begin{equation}\label{ftproj}
    \pi(\mathbf{F}_{tW})=t^2\frac{1}{2}[\theta_\mathrm{MC},\theta_\mathrm{MC}]=-t^2\dif_\wedge\theta_\mathrm{MC} ,
\end{equation} 
where we used the Maurer--Cartan equation in the final step.
	
On the full Lie algebra $\g = \slk$, the classical ad-invariant polynomials are given by the trace polynomials:
\[f_j(X) = \tr(X^j) \in S^j(\g^\vee)^G \quad (j = 2, \dots, k) ,\]
where $f_1(X) = \tr(X)$ vanishes on traceless matrices. 

To bridge these global coadjoint invariants with Borel's setting on the Cartan subalgebra $\h$, we use the \textit{Chevalley Restriction Theorem} (see, e.g., Theorem 3.1.38 in \cite{chrissginz}). The canonical subspace inclusion $\h \hookrightarrow \g$ induces the surjective restriction homomorphism of polynomial algebras:
\begin{equation}\label{resdef}
	\mathrm{res}: S(\g^\vee) \twoheadrightarrow S(\h^\vee), \quad f \longmapsto f|_\h .
\end{equation}
The Chevalley Restriction Theorem establishes that this map restricts to a graded algebra isomorphism on the invariant subrings:
\begin{equation}\label{chevres}
	\mathrm{res}\big|_{S(\g^\vee)^G}: S(\g^\vee)^G \xrightarrow{\ \sim\ } S(\h^\vee)^W .
\end{equation}
Under this isomorphism, each global trace invariant $f_j$ restricts to the corresponding power-sum polynomial on $\h$:
\[p_j := \mathrm{res}(f_j) = f_j\vert{}_\h = \sum_{i=1}^k L_i^j \in S^j(\h^\vee)^W \quad (j = 2, \dots, k) .\]

\begin{definition}[Chern--Simons Cochain]
    We define the \textit{universal Chern--Simons cochain} associated to $f_j$ as the transgression cochain $Tf_j := h(f_j(\F)) \in W^{2j-1}\g$.
\end{definition}

\begin{remark}
    The evaluation of $f_j\in S^j(\g^\vee)^\g$ on the curvature $F^{\theta_W}:\g^\vee\rightarrow W^2\g$ is to be understood as follows. By definition, $f_j$ is a $k$-multilinear map $\g \times \dots \times \g \to \CC$, and $\F = \sum e_a \otimes \hat{l^a}$. Due to their even degrees, the generators $\hat{l^a}\in W^2\g$ are central in the Weil algebra (an even number of negative signs is added when permuting the wedge product). Consequently, the coefficients of $\F$ generate a commutative subalgebra to which the polynomial $f_j$ may be extended via
    \[f_j(\F) = \sum_{a_1, \dots, a_j} f_j(e_{a_1}, \dots, e_{a_j}) \, \hat{l^{a_1}} \wedge \dots \wedge \hat{l^{a_k}}\in W^{2j}\g .\]
\end{remark}

\begin{proposition}\label{cs}
    The universal Chern--Simons cochain is a \textit{Chern--Simons form} for the connection $\boldsymbol{\theta}_{W}$ in the Weil algebra, that is:
    \begin{equation}\label{csform}
    	\dw(Tf_j)=f_j(\F)=\tr\F^j .
    \end{equation}
\end{proposition}

\begin{remark}\label{cyclermk}
    Technically, a Lie algebra ${\g}$ possesses only a Lie bracket $[-,-]$, not an associative multiplication. Expressions like $\F^j$ and the expansion of the commutator $[\F, \boldsymbol{\theta}_W] = \F \wedge \boldsymbol{\theta}_W - \boldsymbol{\theta}_W \wedge \F$ constitute a standard abuse of notation. These operations are justified by the fact that the invariant polynomial $f_j(X) = \tr(X^j)$ implicitly evaluates elements of ${\g}$ via a chosen finite-dimensional representation $\rho: {\g} \to \mathfrak{gl}(V)$, which in our case is the adjoint representation of $\mathfrak{sl}(k)$. This representation extends to Lie algebra-valued forms via the map $\rho \otimes \operatorname{id}: \g \otimes W\g \to \mathfrak{gl}(V) \otimes W\g$. The target space $\mathfrak{gl}(V) \otimes W\g$ is an associative differential graded algebra, equipped with a product given by matrix composition on the first tensor factor and the exterior wedge product on the second, but written like the standard exterior product. Working  within this associative algebra, the power $\F^j$ is rigorously defined as the $j$-fold associative product of $\rho(\F)$. Furthermore, because $\rho$ is a Lie algebra homomorphism, the Lie bracket of $\g$-valued forms pushes forward  to the graded commutator, yielding
    \begin{align*}
        \rho([\F, \boldsymbol{\theta}_W]) &= \rho(\F) \wedge \rho(\boldsymbol{\theta}_W) -(-1)^{\abs{\F}\abs{\boldsymbol{\theta}_W}} \rho(\boldsymbol{\theta}_W) \wedge \rho(\F)\\
        &=\rho(\F) \wedge \rho(\boldsymbol{\theta}_W) - \rho(\boldsymbol{\theta}_W) \wedge \rho(\F) .
    \end{align*}
    Finally, the matrix trace $\tr: \mathfrak{gl}(V) \to \CC$ extends to a map $\tr \otimes \operatorname{id}$ that acts exclusively on the first tensor factor. Because the Weil differential acts exclusively on the second, the two operators commute. This extended trace also satisfies the graded cyclic property with respect to the associative wedge product. More precisely, for homogeneous matrix-valued forms $\Phi$ and $\Psi$, one has $\tr(\Phi \wedge \Psi) = (-1)^{\vert{}\Phi\vert{}\vert{}\Psi\vert{}} \tr(\Psi \wedge \Phi)$. In the following, the representation $\rho$ is omitted from the notation, and these associative operations are written directly on the elements of $\g \otimes W\g$.
\end{remark}
\begin{proof}
    Recall the homotopy property (\ref{homotopy}). Evaluating it on $f_j(\F)$ yields
    \[\dw (Tf_j)+h(\dw (f_j(\F)))=f_j(\F) .\]
    It remains to show that the polynomial $f_j(\F)$ is closed.
    
    By the second Bianchi identity, $\dw\F=[\F,\boldsymbol{\theta}_W]$.
    By the graded Leibniz rule, we have
    \[\dw(\F^j) = \sum_{k=0}^{j-1} (-1)^{\vert{}\F\vert{}^k} \F^k \wedge (\dw \F) \wedge \F^{j-1-k}\]
    where $\F^j := \underbrace{\F \wedge \dots \wedge \F}_{j \text{ times}}$ and the sign is positive since the curvature has degree $2$. As explained in Remark \ref{cyclermk}, the trace acts on the first and the Weil differential acts on the second component of $\g\otimes W\g$, so they commute. We therefore write
    \[\dw \tr(\F^j) = \tr\left(\sum_{k=0}^{j-1}  \F^k \wedge (\dw \F) \wedge \F^{j-1-k}\right)=j \tr\left( \F^{j-1} \wedge (\dw \F) \right) ,\]
    where we used the linearity and the invariance of the trace under cyclic permutations in the last step (see Remark \ref{cyclermk}). Plugging in the Bianchi identity and writing the bracket as a commutator, i.e. \[[\F,\boldsymbol{\theta}_W]=\F \wedge \boldsymbol{\theta}_W - (-1)^{\vert{}\F\vert{}\vert{}\boldsymbol{\theta}_W\vert{}} \boldsymbol{\theta}_W \wedge \F=\F \wedge \boldsymbol{\theta}_W - \boldsymbol{\theta}_W \wedge \F ,\]
    we obtain:
    \[\dw f_j(\F) = j \left( \tr(\F^j \wedge \boldsymbol{\theta}_W) - \tr(\F^{j-1} \wedge \boldsymbol{\theta}_W \wedge \F) \right) .\]
    Applying the graded cyclic property of the trace one more time we conclude that
    \[\dw f_j(\F) = j \left[ \tr(\F^j \wedge \boldsymbol{\theta}_W) - \tr(\F^j \wedge \boldsymbol{\theta}_W) \right] = 0 .\]
\end{proof}

Proposition~\ref{cs} establishes that the universal Chern--Simons cochains $Tf_j \in W^{2j-1}(\kk)$ function as universal transgression primitives on the contractible Weil algebra. In Subsection~\ref{sec:absolute_primitives}, we project these cochains to the compact group $K$ via the canonical projection $\pi \colon W(\kk) \to \Lambda^\bullet \kk^\vee$, recovering the absolute primitives $\eta_j = \pi(Tf_j)$.

\subsection{Absolute Primitives on the Compact Lie Group}\label{sec:absolute_primitives}

The canonical projection $\pi \colon W(\kk) \to \Lambda^\bullet \kk^\vee$ maps the universal Chern--Simons cochain $Tf_j$ to the differential form $\eta_j$ on the compact group $K$. By the classical Hopf--Koszul--Samelson theorem, these forms constitute the primitive generators of the invariant cohomology $H^*(K; \CC)$.

\begin{theorem}[Hopf--Koszul--Samelson \cite{Meinrenken2013}]\label{hks}
	The inclusion $P(\g)\hookrightarrow(\Lambda^\bullet\g^\vee)^\g$ extends to an isomorphism of graded super algebras:
	\[\Lambda^\bullet P(\g)\longrightarrow (\Lambda^\bullet\g^\vee)^\g,\]
	where the grading on the exterior power of the space of primitives is determined by the condition that the inclusion $P(\g)\hookrightarrow\Lambda^\bullet P(\g)$ preserves degrees.
\end{theorem}

\begin{theorem}[Transgression Theorem \cite{Meinrenken2013}]\label{transth}
	If $\g$ is a complex reductive Lie algebra, the transgression map $\tau:{(S^\bullet\g^\vee_+)^\g} \rightarrow{(\Lambda^\bullet\g^\vee)^\g}$ induces an isomorphism
	\[S^\bullet(\g^\vee)_+^\g/(S^\bullet(\g^\vee)_+)^2\cong P(\g)\]
	between the space of indecomposable elements of $(S^\bullet\g^\vee)^\g$ and the space of primitive elements of $(\Lambda^\bullet\g^\vee)^\g$.
\end{theorem}

\begin{example}\label{etaf}
	If $(S^\bullet\g^\vee)^\g \cong \CC[f_{d_1}, \dots, f_{d_r}]$, then the elements $\eta_{d_i} = \tau(f_{d_i})$ are algebraically independent primitive cocycles.
\end{example}

\begin{proposition}[Computation of Primitive Forms]\label{primform}
    The absolute primitive $\eta_j:=\tau(f_j)=\pi(Tf_j)$ of the fundamental invariant $f_j(\F)=\tr(\F^j)$ is given by the trace of the Maurer--Cartan form:
    \[\eta_j=\frac{j}{2j-1}\tr(\theta_\mathrm{MC}^{2j-1})\in(\Lambda^{2j-1}\kk^\vee)^K.\]
\end{proposition}

\begin{proof}
    The proof is simply an application of the decomposition $\tau=\pi\circ h$ for the standard homotopy operator introduced above.
    Lifting the computation from Remark \ref{khat} to $\g\otimes W\g$, we obtain
    \[k(\F)=\dif t\otimes\boldsymbol{\theta}_W+t\otimes\F-\frac{1}{2}(t-t^2)[\boldsymbol{\theta}_W,\boldsymbol{\theta}_W]=\dif t\otimes\boldsymbol{\theta}_W+1\otimes\mathbf{F}_{tW},\]
    where we used that $\dce(\check{l})=1\otimes\dif_\wedge l$ and that $\dce\boldsymbol{\theta}_W=-\frac{1}{2}[\boldsymbol{\theta}_W,\boldsymbol{\theta}_W]$, a fact we showed in the proof of the Weil structural equation, as well as equation (\ref{ftw}).
    Now $k$ can be pulled into the polynomial $f_j(\F)=\tr(\F^j)$ since it's an algebra morphism, in particular $k(f_j(\F)) = \tr(k(\F)^j) = \tr\big((\mathrm{d}t \otimes \boldsymbol{\theta}_W + 1 \otimes \mathbf{F}_{tW})^j\big)$. Recall that by definition, $h(f_j(\F))=(I\otimes1)\circ k(f_j(\F))$. The integral $I$ over the unit interval essentially serves to isolate the term linear in $\dif t$ from $k(f_j(\F))$:
    \[h(f_j(\F)) = j \int_0^1 \mathrm{d}t \tr(\boldsymbol{\theta}_W \wedge \mathbf{F}_{tW}^{j-1}) ,\]
    and applying the projection to this yields
    \begin{align*}
        \eta_j=\tau(f_j(\F))=\pi\circ h(f_j(\F))&=j \int_0^1 \left(\theta_\mathrm{MC}\wedge(t^2\theta^2_\mathrm{MC})^{j-1}\right)\dif t\\
        &=j\left(\int_0^1 t^{2j-1} \dif t\right)\tr\left(\theta_\mathrm{MC}^{2j-1}\right) ,
    \end{align*}
    since $\pi(\mathbf{F}_{tW})=-t^2\dif_\wedge\theta_\mathrm{MC}$ by equation (\ref{ftproj}), $\pi(\boldsymbol{\theta}_W)=\theta_\mathrm{MC}$ by the proof of Proposition \ref{mc4} and where we changed the order of the projection and the integral since they act respectively on the two different factors of $\CC[t,\dif t]\otimes W\g$.\\
    The integral evaluates to $1/(2j-1)$. This concludes the proof.
\end{proof}

\begin{corollary}[Cocycle Property of Absolute Primitives]\label{primformcor}
    The absolute primitive $\eta_j$ is a cocycle in the Chevalley-Eilenberg complex of $\g$, that is,
    \begin{equation*}
        \dif_\wedge \eta_j = \dif_\wedge\left(\frac{j}{2j-1} \tr\left(\theta_\mathrm{MC}^{2j-1}\right)\right) = 0 .
    \end{equation*}
\end{corollary}
\begin{proof}
    Since the trace acts only on the Lie algebra component, we may pull the differential into it:
    \begin{equation}\label{pullineq}
        \dif_\wedge \eta_j = \frac{j}{2j-1} \tr\big(\dif_\wedge (\theta_\mathrm{MC}^{2j-1})\big)
    \end{equation}
    Combining the definition of the Lie bracket of $1$-forms with the Maurer--Cartan equation (\ref{mceq}) yields $\dif_\wedge\theta_\mathrm{MC}=-\theta_\mathrm{MC}^2$.
    By the graded Leibniz rule, applying the differential $\dif_\wedge$ to $\theta_\mathrm{MC}^{2j-1}$ and plugging in the above leads to:
    \begin{align*} 
		\mathrm{d}_\wedge(\theta_\mathrm{MC}^{2j-1}) &= \sum_{i=0}^{2j-2} (-1)^i \theta_\mathrm{MC}^i \wedge (\mathrm{d}_\wedge \theta_\mathrm{MC}) \wedge \theta_\mathrm{MC}^{2j-2-i}\\
		& = \sum_{i=0}^{2j-2} (-1)^{i+1} \theta_\mathrm{MC}^{2j}.
	\end{align*}
    The sum contains $2j-1$ identical summands with alternating signs, leaving only $\dif_\wedge(\theta_\mathrm{MC}^{2j-1}) = -\theta_\mathrm{MC}^{2j}$.
    Plugging this into equation (\ref{pullineq}), we arrive at:
    \[\dif_\wedge \eta_j = -\frac{j}{2j-1} \tr(\theta_\mathrm{MC}^{2j}) .\]
    But this vanishes because of the graded cyclic property of the trace:
    \[\tr(\theta_\mathrm{MC}^{2j}) = \tr(\theta_\mathrm{MC} \wedge \theta_\mathrm{MC}^{2j-1}) = (-1)^{1 \cdot (2j-1)} \tr(\theta_\mathrm{MC}^{2j-1} \wedge \theta_\mathrm{MC}) = -\tr(\theta_\mathrm{MC}^{2j}) .\]
    This concludes the proof.
\end{proof}

While the universal Chern--Simons cochain satisfies the non-trivial transgression relation $\dw(Tf_j) = f_j(\F)$ in $W(\kk)$, its projection $\eta_j$ to the compact group $K$ is strictly closed ($\dif_\wedge \eta_j = 0$). Thus, projection to $K$ eliminates the curvature target completely, causing a transgression defect that precludes $\eta_j$ from serving as a relative primitive for the flag manifold $K/T$. In Subsection~\ref{sec:static_failure}, we examine why this defect cannot be circumvented through naive algebraic evaluation on the reductive splitting.

\subsection{Failure of Static Algebraic Evaluation}\label{sec:static_failure}

Step~1 of Theorem~\ref{thm:borel_isomorphism} abstractly establishes the vanishing of the characteristic classes $cw_K^\CC(f_j) = [f_j(F_{\ttt})] = 0 \in H^{2j}_\mathrm{dR}(K/T; \CC)$, but it does not provide explicit primitives realizing the exactness on the base. The absolute primitives $\eta_j$ cannot serve this purpose: they live on the total space $K$ and are not horizontal with respect to the $T$-fibration. Indeed, contracting $\eta_j$ with a vertical generator $A \in \ttt$ yields $\iota_A \eta_j = j\tr(A \theta_\mathrm{MC}^{2j-2}) \not\equiv 0$ (for instance, taking $A = h_\alpha$, we have $(\iota_{h_\alpha} \eta_2)(e_\alpha, e_{-\alpha}) = 2\tr(h_\alpha^2) \neq 0$). In particular, $\eta_j$ is not basic, precluding it from descending to $K/T$.

\begin{remark}[Obstruction to Purely Algebraic Evaluation]\label{rmk:psi_w_failure}
	Because the universal connection $\Th$ and curvature $\F$ in $W(\kk)$ mirror the algebraic relations of the canonical connection $\theta_\ttt$ and soldering form $\theta_\m$ under the reductive splitting $\kk = \ttt \oplus \m$, one might ask whether the relative transgression equation $\dif_\m \alpha_j = -p_j(F_\ttt)$ (Theorem~\ref{thm:geometric_transgression}) can be obtained via an algebra homomorphism from the universal Chern--Simons formula \eqref{csform}.

	More precisely, the relative primitives on $K/T$ might be constructed by evaluating the universal Chern--Simons cochain $Tf_j \in W(\kk)$ on the reductive splitting via a naive algebra homomorphism $\psi \colon W(\kk) \to (\Lambda^\bullet \m^\vee)^T \otimes \kk$ defined on generators by $\boldsymbol{\theta}_W \mapsto \theta_\m$ and $\F \mapsto -F_{\ttt}$.
	
	However, this cannot succeed because $\psi$ fails to be a chain map, i.e., $\dif_\m \circ \psi \neq \psi \circ \dif_W$. Indeed, applying $\psi$ to the universal Bianchi identity $\dw \F = [\F, \boldsymbol{\theta}_W]$ (Lemma~\ref{weilselem}) yields:
	\[\psi(\dif_W \F) = [-F_{\ttt}, \theta_\m] \neq 0, \qquad \text{whereas} \qquad \dif_\m(\psi(\F)) = -\dif_\m F_{\ttt} = 0 ,\]
	where we used that the Cartan subalgebra acts non-trivially on the root spaces. 
	Consequently, $\dif_\m(\psi(Tf_j)) \neq \psi(\dif_W Tf_j) = f_j(-F_{\ttt})$, preventing the transfer of the transgression relation via naive evaluation. We circumnavigated this algebraic obstruction via the interpolating connection family $\theta_s$.\\
\end{remark}

Because both static projection ($\pi$) and naive algebraic evaluation ($\psi$) fail to produce basic primitives on $K/T$, we instead interpolate dynamically in the affine space of connections on $K \to K/T$, leading to the 1-parameter connection family $\theta_s$ introduced in Section~\ref{sec:uni}.

\section{Relative Transgression Primitives}\label{sec:uni}

\subsection{The Interpolating Connection Family}

Having established in Section~\ref{sec:weil} that static evaluation in the Weil algebra fails to produce basic primitives on $K/T$, we now construct the canonical odd-degree relative transgression primitives $\alpha_j \in (\Lambda^{2j-1}\m^\vee)^T$ directly on the flag manifold by interpolating dynamically between the canonical connection and the flat Maurer--Cartan form.

Consider therefore the
1-parameter family of $\kk$-valued 1-forms on $K$ interpolating between the connection $\theta_0 := \theta_{\ttt} \in \Omega^1(K, \ttt)$ and the flat Maurer--Cartan form $\theta_1 := \theta_\mathrm{MC} = \theta_{\ttt} + \theta_{\m} \in \Omega^1(K, \kk)$:
\[\theta_s := \theta_{\ttt} + s\theta_{\m}, \quad s \in [0, 1] .\]

Note that $\theta_s$ is not $\ttt$-valued for $s > 0$, so it cannot define a connection $1$-form on our principal $T$-bundle $K \to K/T$. However, consider again the extended principal $K$-bundle $P_K := K \times_T K \to K/T$ (with canonical reduction $f \colon K \hookrightarrow P_K, k \mapsto [k, e]$), whose associated vector bundle via the fundamental representation is the trivial splitting bundle $E = K \times_T \CC^k \cong (K/T) \times \CC^k$ encountered in Step~1 of Theorem~\ref{thm:borel_isomorphism}.

By standard principal bundle theory \cite{kobayashi1969}, the 
canonical connection $\theta_0 = \theta_{\ttt} \in \Omega^1(K, \ttt)$ induces a unique connection $1$-form $\tilde{\theta}_0 \in \Omega^1(P_K, \kk)$ satisfying $f^* \tilde{\theta}_0 = \theta_{\ttt}$. Because the space of connections on $P_K$ is an affine space modeled on the space of horizontal and equivariant $1$-forms $\mathcal{A}(P_K) \cong \Omega^1_\mathrm{hor}(K, \kk)^{(T, \Ad\vert{}_T)}$, adding  $s\theta_\m$ defines a connection $\tilde{\theta}_s$ on $P_K$ for each $s \in [0, 1]$, whose restriction to the subbundle $K$ is $f^* \tilde{\theta}_s = \theta_s$. Indeed,  $s\theta_\m$ is the horizontal projection of the equivariant Maurer--Cartan form under a reductive splitting. 
In a slight abuse of terminology, we may therefore speak of the curvature $F_s := \dif_K \theta_s + \frac{1}{2}[\theta_s, \theta_s]$ of $\theta_s$. It satisfies the following property.

\begin{proposition}[Curvature of the Interpolating Family]\label{prop:interpolating_curvature}
	Consider the $1$-parameter family of $\kk$-valued $1$-forms on $K$:
	\[\theta_s := \theta_{\ttt} + s\theta_\m, \quad s \in [0, 1],\]
	interpolating between the canonical connection $\theta_0 = \theta_{\ttt}$ and the flat Maurer--Cartan form $\theta_1 = \theta_\mathrm{MC}$. 
	The curvature $F_s = \dif_K \theta_s + \frac{1}{2}[\theta_s, \theta_s]$ of the interpolating connection $\theta_s$ on $K$ contains only horizontal components and evaluates to:
	\[F_s = (1 - s^2)F_{\ttt} + \frac{1}{2}(s^2 - s)[\theta_\m, \theta_\m]_\m .\]
	In particular, $\iota_A F_s = 0$ for all vertical generators $A \in \ttt$.
\end{proposition}
\begin{proof}
	Expanding $F_s$ and using the fact that $\ttt$ is abelian ($[\theta_{\ttt}, \theta_{\ttt}] = 0$, as noted in Proposition~\ref{Q2c0}(ii)) gives:
	\[F_s = \dif\theta_{\ttt} + s\big(\dif\theta_{\m} + [\theta_{\ttt}, \theta_{\m}]\big) + \frac{1}{2}s^2[\theta_{\m}, \theta_{\m}] .\]
	By the proof of Proposition~\ref{Q2c0}(ii), decomposing the Maurer--Cartan equation $\dif\theta_\mathrm{MC} + \frac{1}{2}[\theta_\mathrm{MC}, \theta_\mathrm{MC}] = 0$ along the reductive splitting $\kk = \ttt \oplus \m$ yields the projection:
	\[\dif\theta_{\ttt} = F_{\ttt} = -\frac{1}{2}[\theta_{\m}, \theta_{\m}]_{\ttt},\]
	and, as mentioned above, projecting to $\m$ instead of $\ttt$ similarly yields:
	\[\qquad \dif\theta_{\m} + [\theta_{\ttt}, \theta_{\m}] = -\frac{1}{2}[\theta_{\m}, \theta_{\m}]_{\m} .\]
	Note that $[\theta_{\ttt}, \theta_{\m}]\in\m$ by the reductivity of the splitting.
	Substituting these relations and decomposing $[\theta_{\m}, \theta_{\m}] = -2F_{\ttt} + [\theta_{\m}, \theta_{\m}]_{\m}$ eliminates the connection form $\theta_{\ttt}$ entirely:
	\begin{align*}
		F_s &= F_{\ttt} - \frac{1}{2}s[\theta_{\m}, \theta_{\m}]_{\m} + \frac{1}{2}s^2\big(-2F_{\ttt} + [\theta_{\m}, \theta_{\m}]_{\m}\big)\\
		&= (1 - s^2)F_{\ttt} + \frac{1}{2}(s^2 - s)[\theta_{\m}, \theta_{\m}]_{\m}.
	\end{align*}
	Consequently, $F_s$ contains only horizontal components, so $\iota_A F_s = 0$ for all vertical generators $A \in \ttt$.
\end{proof}

\begin{remark}[Factorization of the Interpolating Curvature]\label{rmk:curvature_factorization}
	The interpolating curvature $F_s$ admits a particularly concise linear factorization. Recalling that $[\theta_\m, \theta_\m] = 2\theta_\m^2 = -2F_{\ttt} + [\theta_\m, \theta_\m]_\m$, projecting to $\m$ yields the algebraic identity:
	\begin{equation}\label{eq:bracket_m_factored}
		[\theta_\m, \theta_\m]_\m = 2\theta_\m^2 + 2F_{\ttt} .
	\end{equation}
	Substituting \eqref{eq:bracket_m_factored} into the curvature formula of Proposition~\ref{prop:interpolating_curvature} yields:
	\begin{align*}
		F_s &= (1 - s^2)F_{\ttt} + (s^2 - s)(\theta_\m^2 + F_{\ttt})\\
		&= (1 - s^2 + s^2 - s)F_{\ttt} - s(1 - s)\theta_\m^2 \\
		&= (1 - s)(F_{\ttt} - s\theta_\m^2) .
	\end{align*}
	This factorization cleanly isolates the scalar parameter $(1-s)$ from the differential form factor $(F_{\ttt} - s\theta_\m^2)$, which enables the Beta-integral evaluation in Theorem~\ref{thm:general_primitive}.
\end{remark}

\subsection{The Geometric Transgression Theorem}

\begin{theorem}[Geometric Transgression on the Flag Manifold]\label{thm:geometric_transgression}
	For each $j \ge 1$, the differential form
	\[ \alpha_j := j \int_0^1 \tr(\theta_\m \wedge F_s^{j-1}) \, \dif s \in \Omega^{2j-1}(K)\]
	is basic with respect to the principal fibration $\pi_K \colon K \to K/T$. It descends to a global $(2j-1)$-form on the flag manifold $K/T$ satisfying the transgression equation:
	\[ \dif \alpha_j = -p_j(F_{\ttt}) .\]
	Here, $\dif$ denotes the de Rham differential on $\Omega^\bullet(K/T)$, which is canonically identified with the relative Chevalley--Eilenberg differential $\dif_\m$ under the DGA isomorphism $\pi_K^* \colon \Omega^*(K/T) \xrightarrow{\sim}\Omega^*_{\mathrm{bas}}(K)\cong (\Lambda^\bullet \m^\vee)^T$. 
\end{theorem}

\begin{proof}
	Consider the $1$-parameter family of connections $\theta_s = \theta_{\ttt} + s\theta_\m$ on $K$ with curvature $F_s = \dif_K \theta_s + \frac{1}{2}[\theta_s, \theta_s]$. Differentiating the curvature with respect to $s$ yields:
	\[ \frac{\dif}{\dif s} F_s = \dif_K\left(\frac{\dif \theta_s}{\dif s}\right) + \left[\theta_s, \frac{\dif \theta_s}{\dif s}\right] = \dif_K \theta_\m + [\theta_s, \theta_\m] =: \dif^{\theta_s}\theta_\m , \]
	where $\dif^{\theta_s}$ denotes the exterior covariant derivative associated with $\theta_s$. 
	
	Differentiating the characteristic form $\tr(F_s^j)$ with respect to $s$ and applying the cyclic invariance of the trace (cf.\ Remark~\ref{cyclermk}) (evaluated here on wedges of even degree forms) gives:
	\[ \frac{\dif}{\dif s}\tr(F_s^j) = j\tr\left(\frac{\dif F_s}{\dif s} \wedge F_s^{j-1}\right) = j\tr\left(\dif^{\theta_s}\theta_\m \wedge F_s^{j-1}\right) . \]
	On the other hand, the second Bianchi identity $\dif^{\theta_s} F_s = \dif_K F_s + [\theta_s, F_s] = 0$ implies $\dif^{\theta_s}(F_s^{j-1}) = 0$. By the graded Leibniz rule and the cyclic property of the trace ($\tr([\theta_s, \omega]) = 0$, recall e.g. Remark \ref{cyclermk}), the ordinary exterior derivative $\dif_K$ on the total space $K$ satisfies:
	\begin{align*}
		\dif_K \tr(\theta_\m \wedge F_s^{j-1}) &= \tr(\dif^{\theta_s}\theta_\m \wedge F_s^{j-1}) - \tr(\theta_\m \wedge \dif^{\theta_s}(F_s^{j-1})) - \tr\big([\theta_s, \theta_\m \wedge F_s^{j-1}]\big)\\
		&= \tr(\dif^{\theta_s}\theta_\m \wedge F_s^{j-1}) .
	\end{align*}
	Equating both expressions yields the following identity on $K$:
	\[ \frac{\dif}{\dif s}\tr(F_s^j) = j \, \dif_K \tr(\theta_\m \wedge F_s^{j-1}) . \]
	
	Integrating both sides with respect to $s \in [0, 1]$ and applying the Fundamental Theorem of Calculus:
	\[ \tr(F_1^j) - \tr(F_0^j) = \dif_K\left( j \int_0^1 \tr(\theta_\m \wedge F_s^{j-1}) \, \dif s \right) = \dif_K \alpha_j . \]
	At $s = 1$, the connection $\theta_1 = \theta_\mathrm{MC}$ is flat by the Maurer--Cartan equation, so $F_1 = 0$ and $\tr(F_1^j) = 0$.
	At $s = 0$, the connection restricts to the Cartan connection $\theta_0=\theta_{\ttt}$ with curvature $F_0=F_{\ttt} \in \Omega^2(K/T, \ttt)$. Because $F_{\ttt}$ takes values in $\ttt$, the evaluation of the trace polynomial coincides with the restricted Borel polynomial: $f_j(F_{\ttt}) = (f_j\vert{}_{\ttt})(F_{\ttt}) = p_j(F_{\ttt})$.
	We may therefore write:
	\[\dif_K \alpha_j = -p_j(F_{\ttt}) .\]
	
	To verify that $\alpha_j$ descends to $K/T$, we check that it is basic under the right action of $T$:
	\begin{enumerate}[(i)]
		\item {Horizontality:} For every vertical generator $A \in \ttt$, we have $\iota_A \theta_\m = 0$ by definition of the projection $\theta_\m = \operatorname{pr}_\m(\theta_\mathrm{MC})$, and $\iota_A F_s = 0$ by Proposition~\ref{prop:interpolating_curvature}. Since the contraction operator is a graded derivation commuting with integration over $s$, $\iota_A \alpha_j = 0$.
		\item {$T$-Invariance:} For each $t \in T$, the Maurer--Cartan form satisfies $R_t^* \theta_\mathrm{MC} = \Ad(t^{-1}) \circ \theta_\mathrm{MC}$. Since $\kk = \ttt \oplus \m$ is reductive, the projection commutes with $\Ad(t^{-1})$, giving $R_t^* \theta_\m = \Ad(t^{-1})\theta_\m$ and $R_t^* F_s = \Ad(t^{-1})F_s$. The $\Ad$-invariance of the trace ensures that $R_t^* \tr(\theta_\m \wedge F_s^{j-1}) = \tr(\theta_\m \wedge F_s^{j-1})$, so $R_t^* \alpha_j = \alpha_j$.
	\end{enumerate}
	Therefore, $\alpha_j \in \Omega^{2j-1}_\mathrm{bas}(K) \cong \Omega^{2j-1}(K/T)$, and the boundary equation $\dif \alpha_j = -p_j(F_{\ttt})$ holds on the flag manifold.
\end{proof}

Note that the form $\alpha_j \in \Omega^{2j-1}(K)$ is left-$K$-invariant. Indeed, this follows from the left-invariance of $\theta_\m=\mathrm{pr}_\m\circ\theta_\mathrm{MC}$.
Under the DGA isomorphisms
\[ \big({}^K\Omega^*_\mathrm{bas}(K), \dif_K\big) \xrightarrow[\pi_K^*]{\sim} \big({}^K\Omega^*(K/T), \dif\big) \cong \big((\Lambda^\bullet \m^\vee)^T, \dif_\m\big),\]
the transgression equation becomes:
\[ \dif_\m \alpha_j = -p_j(F_{\ttt}) \in (\Lambda^{2j}\m^\vee)^T . \]

While Theorem~\ref{thm:geometric_transgression} establishes the existence and descent of the relative primitives via an integral formula, evaluating the power $F_s^{j-1}$ requires managing the non-commutativity of the curvature $F_{\ttt}$ and the soldering power $\theta_\m^2$. In Subsection~\ref{sec:universal_closed_formula}, we expand this power via non-commutative symmetric word sums and evaluate the resulting parameter integrals using Euler's Beta integral, yielding a universal closed-form formula for all degrees $j \ge 1$. 

\subsection{Proof of the Universal Closed Formula}\label{sec:universal_closed_formula}

The simplification $F_s = (1 - s)(F_{\ttt} - s\theta_\m^2)$ established in Remark~\ref{rmk:curvature_factorization} allows us to deduce a closed-form expression of the relative transgression primitive $\alpha_j$.

\begin{definition}[Non-Commutative Symmetric Word Sums]
	Let $\mathcal{A}$ be an associative algebra (here, the algebra of matrix-valued differential forms $\Omega^*(K, \mathfrak{gl}_k)$). For any two elements $A, B \in \mathcal{A}$ and non-negative integers $n \ge k \ge 0$, we define the \textit{complete non-commutative word sum} $S_{n, k}(A, B) \in \mathcal{A}$ as the sum over all $\binom{n}{k}$ words of length $n$ containing exactly $k$ factors of $B$ and $n - k$ factors of $A$:
	\[S_{n, k}(A, B) := \sum_{\substack{w \in \{A, B\}^n \\ \vert{}w\vert{}_B = k}} w .\]
	By convention, $S_{0, 0}(A, B) := {1}$.
\end{definition}
In particular, for an arbitrary scalar parameter $s \in \CC$, the $n$-th power of binomial $(A - sB)$ expands as follows:
\[(A - sB)^n = \sum_{k=0}^n (-s)^k S_{n, k}(A, B) .\]

\begin{theorem}[Closed Formula for Transgression Primitives]\label{thm:general_primitive}
	For every integer $j \ge 1$, the relative
 	transgression primitive $\alpha_j \in (\Lambda^{2j-1}\m^\vee)^T \cong {}^K\Omega^{2j-1}(K/T)$, satisfying 
 	$\dif_\m \alpha_j = -p_j(F_{\ttt}),$ 
 	admits the following, closed-form expansion:
	\[\alpha_j = \sum_{k=1}^{j-1} (-1)^k\inv{\binom{j+k}{k}} \tr\big(\theta_\m \wedge S_{j-1, k}(F_{\ttt}, \theta_\m^2)\big) .\]
\end{theorem}

\begin{proof}
	By Theorem~\ref{thm:geometric_transgression}, the relative transgression primitive is given by the following integral:
	\[\alpha_j = j \int_0^1 \tr(\theta_\m \wedge F_s^{j-1}) \, \dif s .\]
	Recalling from Remark~\ref{rmk:curvature_factorization} that substituting $[\theta_\m, \theta_\m]_\m = 2\theta_\m^2 + 2F_{\ttt}$ into the interpolating curvature $F_s$ from Proposition~\ref{prop:interpolating_curvature} leads to the simplified expression
	$F_s=(1 - s)(F_{\ttt} - s\theta_\m^2)$.

	To compute $\alpha_j$, we next need to address the $(j - 1)$-th power:
	\[F_s^{j-1} = (1 - s)^{j-1} (F_{\ttt} - s\theta_\m^2)^{j-1} .\]
	The first term is a scalar; the second requires further attention. 	
	Because $F_{\ttt} \in \Omega^2(K, \ttt)$ and $\theta_\m^2 \in \Omega^2(K, \mathfrak{gl}_k)$ do not commute under matrix multiplication, we expand the power using the non-commutative symmetric word sums $S_{j-1, k}(F_{\ttt}, \theta_\m^2)$:
	\[(F_{\ttt} - s\theta_\m^2)^{j-1} = \sum_{k=0}^{j-1} (-s)^k S_{j-1, k}(F_{\ttt}, \theta_\m^2) .\]
	
	Next, we wedge on the left with the soldering form $\theta_\m$, apply the trace and exploit its linearity to obtain:
	\[\tr(\theta_\m \wedge F_s^{j-1}) = (1 - s)^{j-1} \sum_{k=0}^{j-1} (-s)^k \tr\big(\theta_\m \wedge S_{j-1, k}(F_{\ttt}, \theta_\m^2)\big) .\]
	
	Finally, it remains to integrate in $s$ over the unit interval and multiply by $j$. We have conveniently separated the integration variable from the differential forms:
	\[\alpha_j = \sum_{k=0}^{j-1} \left( j (-1)^k \int_0^1 s^k (1 - s)^{j-1} \, \dif s \right) \tr\big(\theta_\m \wedge S_{j-1, k}(F_{\ttt}, \theta_\m^2)\big) .\]
	Note that the integral is Euler's Beta integral $\mathrm{B}(k + 1, j)$. Its value is well known:
	\[\int_0^1 s^k (1 - s)^{j-1} \, \dif s = \frac{\Gamma(k + 1)\Gamma(j)}{\Gamma(k + 1 + j)} = \frac{k!(j - 1)!}{(j + k)!} .\]
	Multiplying by $j$ and the sign $(-1)^k$ yields:
	\[j (-1)^k \frac{k!(j - 1)!}{(j + k)!} = (-1)^k \frac{j! k!}{(j + k)!} = \frac{(-1)^k}{\binom{j+k}{k}} = (-1)^k\inv{\binom{j+k}{k}} .\]

	For $k = 0$, the word sum consists of the single term $S_{j-1, 0}(F_{\ttt}, \theta_\m^2) = F_{\ttt}^{j-1}$. The corresponding term in the primitive is:
	\[\tr\big(\theta_\m \wedge S_{j-1, 0}(F_{\ttt}, \theta_\m^2)\big) = \tr(\theta_\m \wedge F_{\ttt}^{j-1}) .\]
	Hence by the Trace Orthogonality Lemma~\ref{lem:trace_orthogonality}, the $k = 0$ summand vanishes.
	
	The summation index therefore runs from $k = 1$ to $j - 1$:
	\[\alpha_j = \sum_{k=1}^{j-1} (-1)^k\inv{\binom{j+k}{k}} \tr\big(\theta_\m \wedge S_{j-1, k}(F_{\ttt}, \theta_\m^2)\big) .\]
	Because $\theta_\m$ and $F_{\ttt}$ are left-$K$-invariant and basic with respect to the right action of $T$, every constituent word is again basic and invariant. Hence $\alpha_j \in (\Lambda^{2j-1}\m^\vee)^T \cong {}^K\Omega^{2j-1}(K/T)$.
\end{proof}

Theorem~\ref{thm:general_primitive} provides the master formula for all degrees $j \ge 1$. To understand the explicit structure of these forms and observe how curvature terms enter with increasing degree, we now evaluate this formula explicitly for the low-degree cases $j = 2$ (the second primitive) and $j = 3$ (the third primitive) in Subsection~\ref{sec:low_degree_primitives}, before addressing the non-commutative threshold at $j = 4$ in Subsection~\ref{sec:fourth}.

\subsection{Explicit Low-Degree Primitives}\label{sec:low_degree_primitives}

We now evaluate the universal transgression formula of Theorem~\ref{thm:general_primitive} explicitly in low degrees. For $j=2$, the second primitive $\alpha_2$ contains zero curvature factors and is constructed purely from the soldering form. For $j=3$, the third primitive $\alpha_3$ is linear in the curvature $F_{\ttt}$, where cyclic trace permutations still commute trivially past forms of total even degree.

Our goal in this section is to explicitly compute the relative transgression primitive $\alpha_j := j \int_0^1 \tr(\theta_\m \wedge F_s^{j-1}) \, \dif s$. By Proposition \ref{prop:interpolating_curvature} and Theorem \ref{thm:geometric_transgression}, This reduces to integrating the trace of 
\[\theta_\m \wedge F_s^{j-1} = \theta_\m \wedge \left( (1-s^2)F_{\ttt} + \frac{1}{2}(s^2-s)[\theta_\m, \theta_\m]_{\m} \right)^{j-1}\in\quad \Omega^{2j-1}(K/T)\]
over the unit interval before multiplying by $j$, where the Trace Orthogonality Lemma~\ref{lem:trace_orthogonality} will help eliminate some of the terms. 
Since the trace invariants on $\mathfrak{sl}(k)$ begin at degree 2 ($f_1(X) = \tr(X) \equiv 0$), the base case is $j = 2$.

\begin{corollary}[The Second Primitive]\label{cor:j2_primitive}
	For $j = 2$, the transgression primitive evaluates explicitly to:
	\[\alpha_2 = -\frac{1}{3}\tr(\theta_\m^3) \in (\Lambda^3 \m^\vee)^T \cong {}^K\Omega^3(K/T) ,\]
	and satisfies the boundary equation:
	\[\dif_\m \left( \frac{1}{3}\tr(\theta_\m^3) \right) = \tr(F_{\ttt}^2) .\]
\end{corollary}

\begin{proof}
	Specializing the closed formula of Theorem~\ref{thm:general_primitive} to $j = 2$, the sum contains only the single term $k = 1$:
	\[\alpha_2 = (-1)^1 \inv{\binom{3}{1}} \tr\big(\theta_\m \wedge S_{1, 1}(F_{\ttt}, \theta_\m^2)\big) = -\frac{1}{3}\tr(\theta_\m \wedge \theta_\m^2) = -\frac{1}{3}\tr(\theta_\m^3) .\]
	Alternatively, we may derive this directly from the transgression integral of Theorem~\ref{thm:geometric_transgression}:
	By Theorem~\ref{thm:geometric_transgression}, the second primitive is given by:
	\[\alpha_2 = 2 \int_0^1 \tr(\theta_\m \wedge F_s) \, \dif s .\]
	Substituting the interpolating curvature $F_s = (1 - s^2)F_{\ttt} + \frac{1}{2}(s^2 - s)[\theta_\m, \theta_\m]_\m$, we get the following:
	\[\tr(\theta_\m \wedge F_s) = (1 - s^2)\tr(\theta_\m \wedge F_{\ttt}) + \frac{1}{2}(s^2 - s)\tr\big(\theta_\m \wedge [\theta_\m, \theta_\m]_\m\big) .\]
	By Lemma~\ref{lem:trace_orthogonality}, the term $\tr(\theta_\m \wedge F_{\ttt})$ vanishes. 

	To express the remaining summand in terms of $\theta_\m^3$, note that $[\theta_\m, \theta_\m] = 2\theta_\m^2$.
	This follows immediately from fact that the left hand side is the graded commutator of  differential forms evaluated on the same 1-form twice.
	Projecting this bracket along $\kk = \ttt \oplus \m$ and using $[\theta_\m, \theta_\m]_{\ttt} = -2F_{\ttt}$ leads to:
	\[2\theta_\m^2 = -2F_{\ttt} + [\theta_\m, \theta_\m]_\m \implies [\theta_\m, \theta_\m]_\m = 2\theta_\m^2 + 2F_{\ttt} .\]
	Wedging on the left with $\theta_\m$ and taking the trace:
	\[\tr\big(\theta_\m \wedge [\theta_\m, \theta_\m]_\m\big) = 2\tr(\theta_\m^3) + 2\tr(\theta_\m \wedge F_{\ttt}) = 2\tr(\theta_\m^3) ,\]
	where the second term vanishes again by Lemma~\ref{lem:trace_orthogonality}. Evaluating the integral is elementary:
	\[\int_0^1 (s^2 - s) \, \dif s = \left[ \frac{s^3}{3} - \frac{s^2}{2} \right]_0^1 = -\frac{1}{6} .\]
	Combining these factors, we obtain the desired form of the primitive:
	\[\alpha_2 = 2 \cdot \frac{1}{2}\left(-\frac{1}{6}\right) \cdot 2\tr(\theta_\m^3) = -\frac{1}{3}\tr(\theta_\m^3) .\]
\end{proof}

Although the boundary relation $\dif_\m \alpha_2 = -p_2(F_{\ttt}) = -\tr(F_{\ttt}^2)$ is guaranteed by Theorem~\ref{thm:geometric_transgression}, we also verify it directly to illustrate the algebraic behavior of the relative differential $\dif_\m$ prior to the root-lattice analysis in Section~\ref{sec:root}. Recall that the relative differential satisfies $\dif_\m \theta_\m = -\frac{1}{2}[\theta_\m, \theta_\m]_\m = R - \theta_\m^2$, where $R := -F_{\ttt} \in \ttt$. 

Applying the graded Leibniz rule and the cyclic property of the trace:
\begin{align*}
	\dif_\m \tr(\theta_\m^3) &= \tr\big((\dif_\m \theta_\m)\theta_\m^2 - \theta_\m(\dif_\m \theta_\m)\theta_\m + \theta_\m^2(\dif_\m \theta_\m)\big) \\
	&= \tr\big((\dif_\m \theta_\m)\theta_\m^2\big) - (-1)^{1 \cdot (2+1)}\tr\big((\dif_\m \theta_\m)\theta_\m^2\big) + (-1)^{2 \cdot 2}\tr\big((\dif_\m \theta_\m)\theta_\m^2\big) \\
	&= 3\tr\big((\dif_\m \theta_\m)\wedge \theta_\m^2\big) .
\end{align*}
Substituting $\dif_\m \theta_\m = R - \theta_\m^2$ leads to:
\[\dif_\m\left(-\frac{1}{3}\tr(\theta_\m^3)\right) = -\tr\big((R - \theta_\m^2)\theta_\m^2\big) = -\tr(R \wedge \theta_\m^2) + \tr(\theta_\m^4) .\]
Because $\theta_\m$ is a $1$-form, we have $\tr(\theta_\m^4) = (-1)^{1 \cdot 3}\tr(\theta_\m^4) = -\tr(\theta_\m^4) = 0$ by graded commutativity. Substituting $\theta_\m^2 = R - \dif_\m \theta_\m$ into the remaining term:
\[-\tr(R \wedge \theta_\m^2) = -\tr\big(R \wedge (R - \dif_\m \theta_\m)\big) = -\tr(R^2) + \tr(R \wedge \dif_\m \theta_\m) .\]
Since $R = -F_{\ttt} \in \ttt$ and $\dif_\m \theta_\m \in \m$, Lemma~\ref{lem:trace_orthogonality} implies $\tr(R \wedge \dif_\m \theta_\m) = 0$. 
Finally, plugging $R^2 = (-F_{\ttt})^2 = F_{\ttt}^2$ back in leads to the proven result:
\[\dif_\m\left(-\frac{1}{3}\tr(\theta_\m^3)\right) = -\tr(F_{\ttt}^2) \iff \dif_\m\left(\frac{1}{3}\tr(\theta_\m^3)\right) = \tr(F_{\ttt}^2) .\]

\begin{corollary}[The Third Primitive]\label{cor:j3_primitive}
	For $j = 3$, the transgression primitive on the flag manifold $K/T$ evaluates to the invariant $5$-form:
	\[\alpha_3 = -\frac{1}{2}\tr(\theta_\m^3 \wedge F_{\ttt}) + \frac{1}{10}\tr(\theta_\m^5) \in (\Lambda^5 \m^\vee)^T \cong {}^K\Omega^5(K/T) ,\]
	and satisfies the transgression boundary equation:
	\[\dif_\m \alpha_3 = -p_3(F_{\ttt}) = -\tr(F_{\ttt}^3) .\]
\end{corollary}

\begin{proof}
	Specializing Theorem~\ref{thm:general_primitive} to $j = 3$ gives the formula directly without recalculating Beta integrals:
	\begin{align*}
		\alpha_3 &= (-1)^1 \inv{\binom{4}{1}} \tr\big(\theta_\m \wedge S_{2, 1}(F_{\ttt}, \theta_\m^2)\big) + (-1)^2 \inv{\binom{5}{2}} \tr\big(\theta_\m \wedge S_{2, 2}(F_{\ttt}, \theta_\m^2)\big) \\
		&= -\frac{1}{4} \tr\big(\theta_\m \wedge (F_{\ttt}\theta_\m^2 + \theta_\m^2 F_{\ttt})\big) + \frac{1}{10}\tr(\theta_\m \wedge \theta_\m^4) \\
		&= -\frac{1}{4} \cdot 2\tr(\theta_\m^3 \wedge F_{\ttt}) + \frac{1}{10}\tr(\theta_\m^5) = -\frac{1}{2}\tr(\theta_\m^3 \wedge F_{\ttt}) + \frac{1}{10}\tr(\theta_\m^5) .
	\end{align*}
	Alternatively, evaluating the integral directly from Theorem~\ref{thm:geometric_transgression}:
	By Theorem~\ref{thm:geometric_transgression}, the third primitive is given by:
	\[\alpha_3 = 3 \int_0^1 \tr(\theta_\m \wedge F_s^2) \, \dif s .\]
	By Remark~\ref{rmk:curvature_factorization}, the interpolating curvature factors as $F_s = (1 - s)(F_{\ttt} - s\theta_\m^2)$.
	Squaring this expression yields:
	\[F_s^2 = (1 - s)^2 \big( F_{\ttt}^2 - s(F_{\ttt} \wedge \theta_\m^2 + \theta_\m^2 \wedge F_{\ttt}) + s^2 \theta_\m^4 \big) .\]
	Wedging on the left with $\theta_\m$ and taking the matrix trace:
	\begin{itemize}
		\item By Lemma~\ref{lem:trace_orthogonality}, $\tr(\theta_\m \wedge F_{\ttt}^2) = 0$ due to the orthogonality of $\m$ and $\ttt$.
		\item By the graded cyclic permutation of the trace,
		\[\tr(\theta_\m \wedge F_{\ttt} \wedge \theta_\m^2) = (-1)^{3 \cdot 2} \tr(\theta_\m^2 \wedge \theta_\m \wedge F_{\ttt}) = \tr(\theta_\m^3 \wedge F_{\ttt}) ,\]
		and $\tr(\theta_\m \wedge \theta_\m^2 \wedge F_{\ttt}) = \tr(\theta_\m^3 \wedge F_{\ttt})$. Thus the cross-term simplifies to $-2s\tr(\theta_\m^3 \wedge F_{\ttt})$. Integrating over $s \in [0, 1]$:
		\[3 \int_0^1 -2s(1 - s)^2 \, \dif s = -6 \left( \frac{1}{2} - \frac{2}{3} + \frac{1}{4} \right) = -6 \cdot \frac{1}{12} = -\frac{1}{2} .\]
		\item The pure soldering term $\theta_\m \wedge \theta_\m^4 = \theta_\m^5$ yields the scalar integral:
		\[3 \int_0^1 s^2(1 - s)^2 \, \dif s = 3 \left( \frac{1}{3} - \frac{1}{2} + \frac{1}{5} \right) = 3 \cdot \frac{1}{30} = \frac{1}{10} .\]
	\end{itemize}
	Summing these contributions yields $\alpha_3 = -\frac{1}{2}\tr(\theta_\m^3 \wedge F_{\ttt}) + \frac{1}{10}\tr(\theta_\m^5)$, which satisfies $\dif_\m \alpha_3 = -\tr(F_{\ttt}^3)$ by Theorem~\ref{thm:geometric_transgression}.
\end{proof}

The explicit primitives in low degrees demonstrate a clear progression in curvature dependence: for $j=2$, the second primitive $\alpha_2 = -\frac{1}{3}\tr(\theta_\m^3)$ is purely unperturbed by curvature; for $j=3$, the third primitive $\alpha_3$ contains a single curvature factor $F_{\ttt}$, where cyclic trace permutations still commute trivially past forms of total even degree. The non-commutative threshold occurs at $j=4$, where quadratic curvature $F_{\ttt}^2$ generates words such as $F_{\ttt} \theta_\m^2 F_{\ttt}$ that fail to commute under the trace, leading to the necklace splitting analyzed in Subsection~\ref{sec:fourth}. 

\subsection{The Fourth Primitive and Necklace Splitting}\label{sec:fourth}

We now analyze the structural transition that takes place at degree $j = 4$. The presence of quadratic curvature $F_\ttt^2$ coupled with $[F_\ttt, \theta_\m^2] \neq 0$ shatters the cyclic symmetry of the word sums, causing words of identical letter content to split into independent necklaces.

\begin{corollary}[The Fourth Primitive]\label{cor:j4_primitive}
	For $j = 4$, the geometric transgression primitive evaluates to the invariant $7$-form:
	\begin{equation*}
		\alpha_4 = -\frac{2}{5}\tr(\theta_\m^3 \wedge F_{\ttt}^2) - \frac{1}{5}\tr(\theta_\m \wedge F_{\ttt} \wedge \theta_\m^2 \wedge F_{\ttt}) 
		+ \frac{1}{5}\tr(\theta_\m^5 \wedge F_{\ttt}) - \frac{1}{35}\tr(\theta_\m^7),
	\end{equation*}
	living in $(\Lambda^7 \m^\vee)^T \cong {}^K\Omega^7(K/T)$ and satisfying the transgression boundary equation:
	\[\dif_\m \alpha_4 = -p_4(F_{\ttt}) = -\tr(F_{\ttt}^4) .\]
\end{corollary}

\begin{proof}
	We specialize the closed formula of Theorem~\ref{thm:general_primitive} to $j = 4$:
	\[\alpha_4 = \sum_{k=1}^3 (-1)^k\inv{\binom{4+k}{k}} \tr\big(\theta_\m \wedge S_{3, k}(F_{\ttt}, \theta_\m^2)\big) .\]
	The scalar coefficients evaluate to:
	\[(-1)^1\inv{\binom{5}{1}} = -\frac{1}{5}, \qquad (-1)^2\inv{\binom{6}{2}} = \frac{1}{15}, \qquad (-1)^3\inv{\binom{7}{3}} = -\frac{1}{35} .\]
	
	Because the total differential form degree $2j - 1 = 7$ is odd, moving an arbitrary block of wedged-together forms of total degree $d$ from back to front under the trace incurs the positive sign $(-1)^{d(7 - d)} = +1$, since exactly one of $d$ or $7 - d$ is even. Writing $A := F_{\ttt}$ and $B := \theta_\m^2$, we classify the words in $S_{3, k}(A, B)$, where $k$ is the number of $B$'s, under cyclic trace permutations:
	\begin{enumerate}[(i)]
		\item Case $k = 1$:
		The sum $S_{3, 1}(A, B)$ contains $\binom{3}{1} = 3$ words:
		\[S_{3, 1}(A, B) = BAA + ABA + AAB .\]
		Wedging on the left with $\theta_\m$:
		\begin{itemize}
			\item $\theta_\m \wedge BAA = \theta_\m^3 \wedge F_{\ttt}^2$.
			\item $\theta_\m \wedge AAB = \theta_\m \wedge F_{\ttt}^2 \wedge \theta_\m^2$. Permuting the $2$-form $\theta_\m^2$ to the front under the trace yields:
			\[\tr(\theta_\m \wedge F_{\ttt}^2 \wedge \theta_\m^2) = (-1)^{2 \cdot 5} \tr(\theta_\m^2 \wedge \theta_\m \wedge F_{\ttt}^2) = \tr(\theta_\m^3 \wedge F_{\ttt}^2) .\]
			\item $\theta_\m \wedge ABA = \theta_\m \wedge F_{\ttt} \wedge \theta_\m^2 \wedge F_{\ttt}$. Because $F_{\ttt}$ and $\theta_\m^2$ do not commute, this word is an irreducible alternating necklace.
		\end{itemize}

		While the cyclic words $BAA$ and $AAB$ merge under the trace to produce $2\tr(\theta_\m^3 \wedge F_{\ttt}^2)$, the alternating word $\theta_\m \wedge ABA$ cannot be reordered into $\theta_\m^3 \wedge F_{\ttt}^2$ without commuting $F_{\ttt}$ past $\theta_\m^2$. Because $[F_{\ttt}, \theta_\m^2] \neq 0$, it forms an independent trace invariant, giving rise to necklace splitting.

		The $k = 1$ contribution evaluates to:
		\[-\frac{1}{5}\Big(2\tr(\theta_\m^3 \wedge F_{\ttt}^2) + \tr(\theta_\m \wedge F_{\ttt} \wedge \theta_\m^2 \wedge F_{\ttt})\Big) = -\frac{2}{5}\tr(\theta_\m^3 \wedge F_{\ttt}^2) - \frac{1}{5}\tr(\theta_\m \wedge F_{\ttt} \wedge \theta_\m^2 \wedge F_{\ttt}) .\]
		
		\item Case $k = 2$:
		The sum $S_{3, 2}(A, B)$ contains $\binom{3}{2} = 3$ words:
		\[S_{3, 2}(A, B) = BBA + BAB + ABB .\]
		Wedging with $\theta_\m$:
		\begin{itemize}
			\item $\theta_\m \wedge BBA = \theta_\m^5 \wedge F_{\ttt}$.
			\item $\theta_\m \wedge BAB = \theta_\m^3 \wedge F_{\ttt} \wedge \theta_\m^2$. Permuting $\theta_\m^2$ past $\theta_\m^3 \wedge F_{\ttt}$ yields $\tr(\theta_\m^5 \wedge F_{\ttt})$.
			\item $\theta_\m \wedge ABB = \theta_\m \wedge F_{\ttt} \wedge \theta_\m^4$. Permuting $\theta_\m^4$ past $\theta_\m \wedge F_{\ttt}$ yields $\tr(\theta_\m^5 \wedge F_{\ttt})$.
		\end{itemize}
		All three words are cyclically equivalent under the trace. Multiplying by the scalar factor yields the contribution:
		\[3 \cdot \frac{1}{15} \tr(\theta_\m^5 \wedge F_{\ttt}) = \frac{1}{5}\tr(\theta_\m^5 \wedge F_{\ttt}) .\]
		
		\item Case $k = 3$: 
		The unique word in $S_{3, 3}(A, B)$ is $BBB = \theta_\m^6$. Wedging with $\theta_\m$ gives $\theta_\m^7$, which contributes:
		\[-\frac{1}{35}\tr(\theta_\m^7) .\]
	\end{enumerate}
	Summing the three components yields the stated expression for $\alpha_4$. The boundary equation $\dif_\m \alpha_4 = -\tr(F_{\ttt}^4)$ follows from Theorem~\ref{thm:general_primitive}.
\end{proof}

The leading term $-\frac{1}{35}\tr(\theta_\m^7)$ is proportional to the bare projection of the absolute primitive $\eta_4\vert{}_\m$, while the three lower-degree terms provide the non-commutative curvature perturbation required to satisfy the transgression equation $\dif_\m \alpha_4 = -p_4(F_{\ttt})$. This serves as the basis of a more detailed structural analysis in Subsection~\ref{sec:perturbed_mc}.

\subsection{Primitives as Curvature-Perturbed Maurer--Cartan Forms}\label{sec:perturbed_mc}

For $k\geq j>3$, the primitives correspond up to scaling to a perturbed version of the absolute primitive $\eta_j = \frac{j}{2j-1} \tr(\theta_\mathrm{MC}^{2j-1})$ from Proposition~\ref{primform}. 

The explicit forms we derived for the absolute primitives $\eta_j\in\Lambda^{2j-1}\g_\CC^\vee$ and $\alpha_j\in(\Lambda^{2j-1}\m^\vee)^T$ explain the non-commutative threshold at $j=4$ (and the multiplicity barrier at $k=4$) in terms of a perturbation by the canonical curvature $F_{\ttt}$. From Proposition~\ref{primform}, we have:
\[\eta_j=\frac{j}{2j-1}\tr(\theta_\mathrm{MC}^{2j-1}) ,\]
and from Theorem~\ref{thm:general_primitive}, we have:
\[\alpha_j = \sum_{k=1}^{j-1} (-1)^k\inv{\binom{j+k}{k}} \tr\big(\theta_\m \wedge S_{j-1, k}(F_{\ttt}, \theta_\m^2)\big) .\]
Isolating the final summand, this becomes:
\[\alpha_j = (-1)^{j-1}\inv{\binom{2j-1}{j-1}} \tr(\theta_\m^{2j-1}) + \sum_{k=1}^{j-2} \frac{(-1)^k}{\binom{j+k}{k}} \tr\big(\theta_\m \wedge S_{j-1, k}(F_{\ttt}, \theta_\m^2)\big) .\]
The leading term is proportional to the bare projection $\eta_j\vert{}_\m$, while the lower-order summands ($k < j - 1$) provide the necessary perturbation for the transgression equation $\dif \alpha_j = -p_j(F_{\ttt})$ to be fulfilled. We thereby obtain the desired Chern class, which the bare projection $\eta_j$ could not accomplish. In particular, $\dif_\m(\eta_j\vert{}_\m)$ contained no curvature factors, highlighting the necessity of the perturbation. As we saw in the previous section, the case $k=3$ is the last step at which the perturbation is trivial (recall that $\alpha_3=0$ in this case). 

More precisely, on the flag manifold of $\mathrm{SU}(3)$ ($k=3$), the second primitive $\alpha_2 = -\frac{1}{3}\tr(\theta_\m^3)$ contains no curvature factor, while $\alpha_3 \equiv 0$ vanishes because there are no non-zero $T$-invariant $5$-forms on $\mathrm{SU}(3)/T$, i.e., $(\Lambda^5 \m^\vee)^T = \{0\}$. Thus, non-vanishing curvature perturbations first manifest in non-trivial primitives for $\mathrm{SU}(k)$ with $k \ge 4$.

Beyond the flat $k=3$ case, the perturbative terms no longer vanish and curvature plays a role.

This completes the realization of the relative transgression primitives $\alpha_j$ as non-trivial curvature perturbations of the projected Maurer--Cartan forms. While static projection to $K$ yields closed absolute cocycles ($\dif_\wedge \eta_j = 0$), dynamic interpolation along $\theta_s$ restores the curvature target $\dif_\m \alpha_j = -p_j(F_\ttt)$ on the base $K/T$.

\section{Geometric Visualization on the \texorpdfstring{$A_2$}{A2} Root Lattice}\label{sec:root}

Having established the algebraic existence, descent, and closed Beta-integral formula for the relative transgression primitives $\alpha_j$ in Section~\ref{sec:uni}, we now ground the abstract differential geometry in the geometry of the root lattice. For the rank-two system $A_2$ ($\mathfrak{sl}_3(\CC)$), the relative Chevalley--Eilenberg complex $(\Lambda^\bullet \m^\vee)^T$ admits an intriguing pictorial dictionary on the planar hexagonal root lattice:
\begin{enumerate}[(1)]
	\item $T$-invariant cochains correspond bijectively to closed polygonal directed loops.
	\item The relative differential $\dif_\m$ acts physically as an edge-fracturing operator shattering root vectors across Weyl parallelograms.
	\item The second primitive $\alpha_2 = \Omega_- - \Omega_+$ manifests as the oriented difference of  triangular loops whose boundary under $\dif_\m$ is a sum of parallelograms, witnessing the exactness of $c_2(E)$.
\end{enumerate}

\subsection{Root Forms, Weights, and the Loop Dictionary}

\begin{proposition}[Primitive Forms as Root Lattice Loops]\label{prop:closedloops}
	When restricting the primitive forms $\eta_j$ to the orthogonal complement $\m$ of the Cartan subalgebra $\h$ (so that $\g = \h \oplus \m$), the non-vanishing terms correspond to (sums of) closed directed loops of length $2j-1$ in the $A_{k-1}$ root lattice.
\end{proposition}

\begin{proof}
	For notational convenience, we work with the normalized primitive form:
	\[\tilde{\eta}_j := \frac{2j-1}{j} \eta_j = \tr(\theta_\mathrm{MC}^{2j-1}) .\]
	Let us evaluate the restriction of $\tilde{\eta}_j$ to $\m_\CC$. In the root space decomposition  $\m_\CC \cong \bigoplus_{\alpha \in \Phi} \slk_\alpha$, the restricted Maurer--Cartan form expands as
	\[\theta_\mathrm{MC}\vert{}_{\m_\CC} = \theta_\m = \sum_{\alpha \in \Phi} e_\alpha \otimes \omega^\alpha ,\]
	where $e_\alpha$ spans the root space $\slk_\alpha$ and $\omega^\alpha \in \m_\CC^\vee$ is the dual basis element satisfying $\omega^\alpha(e_\beta) = \delta^\alpha_\beta$. Substituting this into the trace yields:
	\begin{align}
		\tilde{\eta}_j|_{\m_\CC} &= \tr\left( \left( \sum_{\alpha \in \Phi} e_\alpha \otimes \omega^\alpha \right)^{2j-1} \right) \nonumber \\
		&= \sum_{\alpha_1, \dots, \alpha_{2j-1} \in \Phi} \tr(e_{\alpha_1} e_{\alpha_2} \cdots e_{\alpha_{2j-1}}) \, \omega^{\alpha_1} \wedge \omega^{\alpha_2} \wedge \dots \wedge \omega^{\alpha_{2j-1}} . \label{eq:primitive_expansion}
	\end{align}
	To determine which summands survive, we evaluate the algebraic and exterior factors:
	\begin{enumerate}[(1)]
		\item {Exterior Factor:} Because each $\omega^{\alpha_i}$ is a 1-form, the anti-commutativity of the wedge product implies that any term containing a repeated root $\alpha_i = \alpha_k$ (for $i \neq k$) vanishes, so non-vanishing terms must consist of $2j-1$ mutually distinct roots.
		
		\item {Matrix Factor:} In the standard defining representation $\slk \subset \mathfrak{gl}(k)$, a root $\alpha \in \Phi$ has the form $\alpha = L_a - L_b$ for distinct indices $a, b \in \{1, \dots, k\}$, with corresponding root vector given by the elementary matrix unit $e_\alpha = E_{ab}$. For a sequence of roots $\alpha_i = L_{a_i} - L_{b_i}$, the matrix product expands via unit multiplication $E_{ab} E_{cd} = \delta_{bc} E_{ad}$:
		\[e_{\alpha_1} e_{\alpha_2} \cdots e_{\alpha_{2j-1}} = E_{a_1 b_1} E_{a_2 b_2} \cdots E_{a_{2j-1} b_{2j-1}} = \left( \prod_{i=1}^{2j-2} \delta_{b_i, a_{i+1}} \right) E_{a_1 b_{2j-1}} .\]
		Taking the matrix trace yields:
		\[\tr(e_{\alpha_1} e_{\alpha_2} \cdots e_{\alpha_{2j-1}}) = \left( \prod_{i=1}^{2j-2} \delta_{b_i, a_{i+1}} \right) \tr(E_{a_1 b_{2j-1}}) = \left( \prod_{i=1}^{2j-2} \delta_{b_i, a_{i+1}} \right) \delta_{b_{2j-1}, a_1} .\]
	\end{enumerate}
	Consequently, the trace is non-zero (and equal to $1$) if and only if the index sequence closes into a directed loop on the vertex set $\{1, \dots, k\}$:
	\[b_1 = a_2, \quad b_2 = a_3, \quad \dots, \quad b_{2j-2} = a_{2j-1}, \quad b_{2j-1} = a_1 .\]
	The roots comprising this sequence take the form:
	\[\alpha_1 = L_{a_1} - L_{a_2}, \quad \alpha_2 = L_{a_2} - L_{a_3}, \quad \dots, \quad \alpha_{2j-1} = L_{a_{2j-1}} - L_{a_1} .\]
	Summing these roots gives the zero weight constraint:
	\[\sum_{i=1}^{2j-1} \alpha_i = (L_{a_1} - L_{a_2}) + (L_{a_2} - L_{a_3}) + \dots + (L_{a_{2j-1}} - L_{a_1}) = 0 .\]
	Geometrically, concatenating the root vectors $\alpha_1, \dots, \alpha_{2j-1}$ head-to-tail in the $A_{k-1}$ root lattice forms a closed polygonal path, or loop, of length $2j-1$ returning to the origin.
\end{proof}

In summary, restricting the absolute primitive $\tau(f_j) = \eta_j$ to the reductive complement $\m_\CC$ yields a superposition of closed directed root cycles:
\[\eta_j\vert{}_{\m_\CC} \in (\Lambda^{2j-1}\m_\CC^\vee)^T .\]
For instance, when $j = 2$, this restriction produces the primitive 3-form $\eta_2\vert{}_{\m_\CC} \in (\Lambda^3 \m_\CC^\vee)^T$, which manifests geometrically as the isotropic sum of oriented triangles Subsection~\ref{sec:coord_alpha2}.

\begin{remark}[Summary: Paths and Loops in the Root Lattice]\label{loops0}
	We now synthesize the geometric dictionary connecting the algebraic cochain complex $C^\bullet = (\Lambda^\bullet \m^\vee)^T$ directly with the Euclidean geometry of the root lattice $\Lambda_R \subset \h^*_\mathbb{R}$:
	The algebraic nature of this complex admits a direct, visual translation into the geometry of the root lattice $\Lambda_R \subset \h^*_\mathbb{R}$.
	
	\begin{enumerate}[(i)]
		\item \textit{Root Forms as Lattice Paths:} Each root $\alpha \in \Phi$ is a vector in the Euclidean space $\h^*_\mathbb{R}$. We associate the dual root form $\omega^\alpha \in \m^\vee$ with this vector. A generic basic $n$-form given by the wedge product $\omega^{\alpha_1} \wedge \dots \wedge \omega^{\alpha_n}$ geometrically corresponds to a sequence of root vectors. When arranged head-to-tail, they trace out a directed path in the root lattice starting from the origin: $0 \to \alpha_1 \to \alpha_1+\alpha_2 \to \dots \to \sum \alpha_i$.
		
		\item\textit{Invariance as Closed Loops:} Our complex restricts to the $T$-invariant subspace. A wedge product of forms has weight zero (is $T$-invariant) if and only if the sum of its roots vanishes: $\sum \alpha_i = 0$. Geometrically, this requires the head-to-tail sequence of root vectors to return to the origin. Therefore, the $T$-invariant cochain space $C^n$ is generated by the \textit{closed polygonal loops} of length $n$ drawn on the root lattice.
		
		\item\textit{The Edge-Fracturing Differential:} The Chevalley-Eilenberg differential $\dif_\m$ is governed by the dual of the Lie bracket. For a basis 1-form, it is given by:
		\[\dif_\m (\omega^\gamma) = -\frac{1}{2} \sum_{\alpha+\beta=\gamma} N_{\alpha, \beta} \, \omega^\alpha \wedge \omega^\beta ,\]
		where $N_{\alpha,\beta}$ are the structure constants. Geometrically, $\dif_\m$ takes a single root vector $\gamma$ (a straight line segment in the lattice) and ``fractures'' it, replacing it with the linear combination of all valid two-step paths $(\alpha, \beta)$ that traverse the parallelogram formed by $\alpha$ and $\beta$. Extending this via the graded Leibniz rule, the differential systematically shatters the edges of our root polygons into finer configurations.
		
		\item\textit{Weyl Equivariance and Representation Theory:} The Weyl group $W = S_k$ acts on the root lattice $\Lambda_R$ by isometries (specifically, by permuting the basis weights $L_i \mapsto L_{\sigma(i)}$). Geometrically, this corresponds to reflections and rotations of the lattice across the root hyperplanes. Because the Lie bracket is $W$-equivariant, the edge-fracturing operator $\dif_\m$ commutes with these symmetries: $\dif_\m(\sigma \cdot \omega) = \sigma \cdot \dif_\m(\omega)$. Consequently, if we take a closed root polygon and reflect it across a hyperplane, its boundary fractures in the mirrored configuration. 
		
		This visualizes the block-diagonalization of the complex into irreducible $S_k$-representations: decomposing the cochain spaces into Specht modules is geometrically equivalent to forming highly symmetric or anti-symmetric superpositions of root polygons.
	\end{enumerate}
	
	This intrinsic geometric model provides the coordinate realizations of the topological invariants of $G/B$. 
	
	While computing the ranks of the matrices $\dif_\m$ becomes combinatorially explosive for large rank $k$, the root-lattice framing provides a concrete geometric realization of the abstract algebraic primitives constructed in Section~\ref{sec:uni}, as we now establish explicitly for $A_2$.
\end{remark}

\subsection{The Canonical Connection and Curvature in Coordinates}\label{sec:canonical_curvature_coord}

We now specialize the general differential-geometric framework to the rank-two root system $A_2$, corresponding to the flag manifold $K/T = \mathrm{SU}(3)/\mathrm{S}(\mathrm{U}(1)^3)$.

We extend differential forms to complex coefficients and work in the complexified relative Chevalley--Eilenberg complex $(\Lambda^\bullet_\CC \m_\CC^\vee)^T$ with Cartan subalgebra $\h = \ttt_\CC$.

Recall the standard Chevalley basis conventions for $\mathfrak{sl}_3(\CC)$: the simple roots are $\alpha_1, \alpha_2$, the highest root is $\alpha_3 := \alpha_1 + \alpha_2$, and the coroots $h_i = [e_{\alpha_i}, e_{-\alpha_i}] \in \h$ span the Cartan subalgebra. By the reductive Maurer--Cartan, the canonical curvature $F_{\ttt} = -\frac{1}{2}[\theta_\m, \theta_\m]_{\ttt}$ restricts to opposite root pairs $[e_\alpha, e_{-\alpha}] = h_\alpha$. Extending scalars to $\h$ and expressing the result in terms of the basic $2$-folds $\nu_\alpha := \omega^\alpha \wedge \omega^{-\alpha} \in C^2$, the curvature $F_{\ttt} \in \Omega^2(K/T; \CC) \otimes \h$ evaluates to:
\[F_{\ttt} = -\sum_{\alpha \in \Phi^+} h_\alpha \otimes \nu_\alpha = -(h_1 \nu_{\alpha_1} + h_2 \nu_{\alpha_2} + h_3 \nu_{\alpha_3}) .\]

Expressed as a diagonal matrix of $2$-forms on the flag manifold:
\[F_{\ttt} = \begin{pmatrix} -(\nu_{\alpha_1} + \nu_{\alpha_3}) & 0 & 0 \\ 0 & \nu_{\alpha_1} - \nu_{\alpha_2} & 0 \\ 0 & 0 & \nu_{\alpha_2} + \nu_{\alpha_3} \end{pmatrix} .\]
We begin by evaluating the invariant polynomials $p_j(F_{\ttt}) = \tr(F_{\ttt}^j)$.
Because $\operatorname{tr}(h_\alpha) = 0$, the first trace vanishes, i.e., $f_1(F_{\ttt}) = \tr(F_{\ttt}) = 0$. 

Write $\alpha_1:=\alpha_{12}:=L_1-L_2,$ $\alpha_2:=\alpha_{23}:=L_2-L_3$. Then $\Phi^+=\{\alpha_1,\alpha_2,(\alpha_1+\alpha_2)\}.$ Consider the following illustration for clarity. The simple roots are marked in white, the hyperplane of positive roots is marked in dark gray. 
\begin{center}
	\begin{tikzpicture}[scale=2.5] 
		\tikzset{grid/.style={black}}
		\begin{rootSystem}{A}
			\roots
			\positiveRootHyperplane
			\simpleroots
			
			\node[above, font=\small] at \Root{1}{0} {$\alpha_2$};
			\node[right, font=\small] at \Root{0}{1} {$\alpha_1$};
			\node[above right, font=\small] at \Root{1}{1} {$\alpha_1 + \alpha_2$};

			
			\foreach \x/\y in {1/0,-1/1, 0/-1}{\wt{\x}{\y}}
			\node[above, font=\small] at \weight{1}{0} {\small\(L_1\)};
			\node[above, font=\small] at \weight{-1}{1} {\small\(L_2\)};
			\node[above, font=\small] at \weight{0}{-1} {\small\(L_3\)};
			
		\end{rootSystem}
	\end{tikzpicture}
\end{center}

For $j = 2$, we compute the characteristic $4$-form $\tr(F_{\ttt}^2)$. Since $\omega^\alpha \wedge \omega^\alpha = 0$, the squares of the $2$-folds vanish: $\nu_{\alpha_k}^2 = (\omega^{\alpha_k} \wedge \omega^{-\alpha_k})^2 = 0$. Squaring the diagonal entries yields:
\begin{align*}
	(F_{\ttt})_{11}^2 &= (\nu_{\alpha_1} + \nu_{\alpha_3})^2 = 2\nu_{\alpha_1} \wedge \nu_{\alpha_3} , \\
	(F_{\ttt})_{22}^2 &= (\nu_{\alpha_1} - \nu_{\alpha_2})^2 = -2\nu_{\alpha_1} \wedge \nu_{\alpha_2} , \\
	(F_{\ttt})_{33}^2 &= (\nu_{\alpha_2} + \nu_{\alpha_3})^2 = 2\nu_{\alpha_2} \wedge \nu_{\alpha_3} .
\end{align*}
Summing these up yields the trace in coordinates:
\begin{equation}\label{eq:coord_trace_F2}
	p_2(F_{\ttt}) = \tr(F_{\ttt}^2) = 2\big(\nu_{\alpha_1} \wedge \nu_{\alpha_3} - \nu_{\alpha_1} \wedge \nu_{\alpha_2} + \nu_{\alpha_2} \wedge \nu_{\alpha_3}\big) \in (\Lambda^4 \m^\vee)^T .
\end{equation}
    
This computation connects directly with the loop model of Remark~\ref{loops0}. On the planar hexagonal root lattice of the $A_2$-system, the differential-geometric objects acquire explicit visual interpretations:
\begin{itemize}
	\item The Soldering Form $\theta_\m$ is the superposition of all directed radial unit steps from the center. Each basis $1$-form $\omega^\alpha \in \m^\vee$ in the coordinate description $\theta_\m = \sum_{\alpha \in \Phi} e_\alpha \otimes \omega^\alpha$ represents a single directed segment traversing the root vector $\alpha$.
	
	\item
	The connection $1$-form $\theta_{\ttt} \in \Omega^1(K, \ttt)$ on the principal $T$-bundle $\pi \colon K \to K/T$ takes values in the real abelian Lie algebra $\ttt = \operatorname{Lie}(T)$. Recall how under the adjoint action, each $\xi \in \ttt$ acts on the one-dimensional root space $\g_\alpha \subset \m_\CC$ by scalar multiplication with $\alpha(\xi) \in i\RR$, the differential of the character 
	\[\chi_\alpha \colon T \to \mathrm{U}(1), \qquad \chi_\alpha(\exp \xi) = e^{\alpha(\xi)} ,\]
	defining the standard $\mathrm{U}(1)$-action on $\g_\alpha$. Under the canonical decomposition of the complexified tangent bundle into associated complex line bundles under the canonical decomposition into complex line bundles established in Theorem~\ref{thm:borel_isomorphism}(i) (equation~\eqref{vbdec}):
	\[\mathrm{T}^{1,0}(K/T) \cong K \times_T \m^- \cong \bigoplus_{\alpha \in \Phi^+} \mathbb{L}_{-\alpha}, \qquad \mathbb{L}_{-\alpha} := K \times_{\chi_{-\alpha}} \CC .\]
	The connection $\theta_{\ttt}$ induces a covariant derivative on each line bundle $\mathbb{L}_{-\alpha}$. The connection $1$-form on the associated principal $\mathrm{U}(1)$-bundle is given by evaluating the root $\alpha \in \h^\vee$ on the Cartan connection,
	\[\alpha(\theta_{\ttt}) \in \Omega^1(K, i\RR) ,\]
	with curvature $2$-form given by $\alpha(F_{\ttt}) \in \Omega^2(K/T, i\RR)$.
	
	\item The Curvature $F_{\ttt} = -\sum_{\alpha \in \Phi^+} h_\alpha \otimes \nu_\alpha$  is the superposition of the length-$2$ degenerate cycles from the center to a root and back. Each basic $2$-fold $\nu_\alpha = \omega^\alpha \wedge \omega^{-\alpha}$ represents such a cycle $0 \to \alpha \to 0$ traversing a root and returning along its negative. 
	Evaluating the curvature on these horizontal generators yields:
	\[F_{\ttt}(e_\alpha, e_{-\alpha}) = -h_\alpha \in \ttt .\]	
	By the Ambrose--Singer theorem, the holonomy Lie algebra $\mathfrak{hol}_p(\theta_{\ttt})$ at $p\in K$ is generated by the values of $(F_{\ttt})_q$  on horizontal vectors $\xi_1, \xi_2 \in \mathcal{H}:= \ker((\theta_{\ttt})_q)  = (L_q)_* \m$ for $q\in H(p)$, where $H(p)$ is the set of points in $K$ that can be linked to $p$ via a horizontal path. 
		
	Because $[e_\alpha, e_{-\alpha}] = h_\alpha$, the horizontal distribution satisfies the bracket-generating condition $\m + [\m, \m] = \kk$. By the Chow--Rashevsky theorem, this implies that $H(p) = K$.
	Consequently, the holonomy Lie algebra at $p\in K$ is generated by the pushforwards of all the horizontal generators, namely $(L_q)_*(h_\alpha$). In particular $\theta_{\ttt}$ has maximal holonomy: $\mathfrak{hol}_p(\theta_{\ttt})\cong\ttt$ for all $p\in K$, and integrating yields $\mathrm{Hol}(\theta_{\ttt}) = T$.
		
	In other words, the coordinate expression $F_{\ttt}= -\sum_{\alpha \in \Phi^+} h_\alpha \otimes \nu_\alpha$ links the holonomy given by the coroot $-h_\alpha \in \ttt$ to the $2$-step cycle $0 \to \alpha \to 0$ in the root lattice.
\end{itemize}

\subsection{The Second Primitive as  Triangle Loops}\label{sec:coord_alpha2}

For $j=2$, we expand $\tr(\theta_\m^3) \in \Lambda^3 \m^\vee$ in the root basis:
\[\tr(\theta_\m^3) = \sum_{\alpha, \beta, \gamma \in \Phi} \tr(e_\alpha e_\beta e_\gamma) \, \omega^\alpha \wedge \omega^\beta \wedge \omega^\gamma .\]
By Proposition~\ref{prop:closedloops} (specialized to $j = 2$), the matrix trace $\tr(e_\alpha e_\beta e_\gamma)$ vanishes unless the roots form a closed directed loop, enforcing the zero-weight condition:
\[\tr(e_\alpha e_\beta e_\gamma) \neq 0 \implies \alpha + \beta + \gamma = 0 ,\]
which recovers our initial character-theoretic description of the invariant forms.

On the $A_2$ root lattice, three non-zero roots sum to zero if and only if they form a closed, directed triangular loop. There are exactly two such fundamental triangles, corresponding to the two orientations:
\begin{itemize}
	\item {Counter-clockwise cycle:} The roots are $\{\alpha_1, \alpha_2, -\alpha_3\}$. Evaluating the matrix products, we get:
	\[e_{\alpha_1} e_{\alpha_2} e_{-\alpha_3} = E_{12} E_{23} E_{31} = E_{11} \implies \tr(e_{\alpha_1} e_{\alpha_2} e_{-\alpha_3}) = 1 .\]
	The cyclic permutations $(e_{\alpha_2}, e_{-\alpha_3}, e_{\alpha_1})$ and $(e_{-\alpha_3}, e_{\alpha_1}, e_{\alpha_2})$ yield $E_{22}$ and $E_{33}$, each with trace $1$. Permuting $1$-forms cyclically preserves the wedge product, so the three cyclic permutations combine to yield $3\Omega_+$, where $\Omega_+ := \omega^{\alpha_1} \wedge \omega^{\alpha_2} \wedge \omega^{-\alpha_3}$.
	All non-cyclic permutations vanish.
	
	\item {Clockwise cycle:} The opposite roots are  $\{-\alpha_1, -\alpha_2, \alpha_3\}$. Here:
	\[e_{-\alpha_2} e_{-\alpha_1} e_{\alpha_3} = E_{32} E_{21} E_{13} = E_{33} \implies \tr(e_{-\alpha_2} e_{-\alpha_1} e_{\alpha_3}) = 1 .\]
	The cyclic permutations yield $E_{22}$ and $E_{11}$, each with trace $1$. In terms of $\Omega_- := \omega^{-\alpha_1} \wedge \omega^{-\alpha_2} \wedge \omega^{\alpha_3}$, the cyclic sum equates to $-3\Omega_-$.
\end{itemize}
Summing these two cycles yields the complete coordinate formula for $\tr(\theta_\m^3)$:
\begin{equation*}
	\tr(\theta_\m^3) = 3(\Omega_+ - \Omega_-) = 3\big(\omega^{\alpha_1} \wedge \omega^{\alpha_2} \wedge \omega^{-\alpha_3} - \omega^{-\alpha_1} \wedge \omega^{-\alpha_2} \wedge \omega^{\alpha_3}\big) .
\end{equation*}
By Corollary~\ref{cor:j2_primitive}, the second transgression primitive is $\alpha_2 = -\frac{1}{3}\tr(\theta_\m^3)$. Hence the factor of $3$ neatly cancels:
\begin{equation*}
	\alpha_2 = \Omega_- - \Omega_+ = \omega^{-\alpha_1} \wedge \omega^{-\alpha_2} \wedge \omega^{\alpha_3} - \omega^{\alpha_1} \wedge \omega^{\alpha_2} \wedge \omega^{-\alpha_3} \in (\Lambda^3 \m^\vee)^T .
\end{equation*}
In words, the second primitive $\alpha_2$ is the difference between the clockwise and counter-clockwise oriented triangle forms on the root lattice.

These oriented triangles are represented by the generators $\Omega_-$ and $\Omega_+$ of the space $C^3$ of $T$-invariant 3-forms. Under the spatial reflection and rotation symmetries of the $A_2$ root hexagon (the Weyl group $S_3$), the space of invariant 3-forms decomposes into symmetric and antisymmetric lines: $C^3 = \CC(\Omega_- - \Omega_+) \oplus \CC(\Omega_- + \Omega_+) \cong \mathbf{1} \oplus \mathbf{1}'$ (which we formalize representation-theoretically in Subsection~\ref{sec:master}). The primitive $\alpha_2 = \Omega_- - \Omega_+$ spans the one-dimensional subspace carrying the trivial representation $\mathbf{1}$. As we shall see, the relative differential $\dif_2$ is an $S_3$-module homomorphism between $C^2 \cong \mathbf{1}' \oplus \mathbf{2}$ and $C^3 \cong \mathbf{1} \oplus \mathbf{1}'$, mapping $\mathbf{1}'$ to $\mathbf{1}'$. 
Having identified the second primitive $\alpha_2$ as the oriented difference of triangle loops $\Omega_- - \Omega_+$, we now evaluate the action of the relative differential $\dif_\m$ on these triangles. In Subsection~\ref{sec:edge_fracturing}, we demonstrate how the edge-fracturing operator shatters them into the parallelogram configurations of $p_2(F_{\ttt})$, visually witnessing the transgression relation $\dif_\m \alpha_2 = -p_2(F_{\ttt})$ in coordinates.

\subsection{The Relative Differential as an Edge-Fracturing Operator}\label{sec:edge_fracturing}

Let us directly verify the transgression equation $\dif_\m \alpha_2 = -\tr(F_{\ttt}^2)$. From the structure constants of $\mathfrak{sl}_3(\CC)$ associated with the chosen basis of $\mathfrak{sl}_3(\CC)$, the relative differential $\dif_\m \omega^\gamma = -\frac{1}{2}\sum_{\alpha+\beta=\gamma} N_{\alpha, \beta} \, \omega^\alpha \wedge \omega^\beta$ acts on root $1$-forms by:
\begin{alignat*}{3}
	\dif_\m \omega^{\alpha_1} &= \omega^{-\alpha_2} \wedge \omega^{\alpha_3}, &\qquad \dif_\m \omega^{\alpha_2} &= -\omega^{-\alpha_1} \wedge \omega^{\alpha_3}, &\qquad \dif_\m \omega^{\alpha_3} &= -\omega^{\alpha_1} \wedge \omega^{\alpha_2} , \\
	\dif_\m \omega^{-\alpha_1} &= -\omega^{\alpha_2} \wedge \omega^{-\alpha_3}, &\qquad \dif_\m \omega^{-\alpha_2} &= \omega^{\alpha_1} \wedge \omega^{-\alpha_3}, &\qquad \dif_\m \omega^{-\alpha_3} &= \omega^{-\alpha_1} \wedge \omega^{-\alpha_2} .
\end{alignat*}
Applying the graded Leibniz rule to $\Omega_+ = \omega^{\alpha_1} \wedge \omega^{\alpha_2} \wedge \omega^{-\alpha_3}$:
\begin{align*}
	\dif_\m \Omega_+ &= (\dif_\m \omega^{\alpha_1}) \wedge \omega^{\alpha_2} \wedge \omega^{-\alpha_3} - \omega^{\alpha_1} \wedge (\dif_\m \omega^{\alpha_2}) \wedge \omega^{-\alpha_3} + \omega^{\alpha_1} \wedge \omega^{\alpha_2} \wedge (\dif_\m \omega^{-\alpha_3}) \\
	&= (\omega^{-\alpha_2} \wedge \omega^{\alpha_3}) \wedge \omega^{\alpha_2} \wedge \omega^{-\alpha_3} + \omega^{\alpha_1} \wedge (\omega^{-\alpha_1} \wedge \omega^{\alpha_3}) \wedge \omega^{-\alpha_3} \\
	&\quad + \: \omega^{\alpha_1} \wedge \omega^{\alpha_2} \wedge (\omega^{-\alpha_1} \wedge \omega^{-\alpha_2}) \\
	&= \nu_{\alpha_2} \wedge \nu_{\alpha_3} + \nu_{\alpha_1} \wedge \nu_{\alpha_3} - \nu_{\alpha_1} \wedge \nu_{\alpha_2} .
\end{align*}

Comparing this with Equation~\eqref{eq:coord_trace_F2} shows that $\dif_\m \Omega_+ = \frac{1}{2}\tr(F_{\ttt}^2)$. An analogous computation for the clockwise triangle then yields $\dif_\m \Omega_- = -\frac{1}{2}\tr(F_{\ttt}^2)$. Differentiating the full primitive $\alpha_2 = \Omega_- - \Omega_+$ reproduces the transgression boundary equation, as expected:
\[ \dif_\m \alpha_2 = \dif_\m \Omega_- - \dif_\m \Omega_+ = -\frac{1}{2}\tr(F_{\ttt}^2) - \frac{1}{2}\tr(F_{\ttt}^2) = -\tr(F_{\ttt}^2) = -p_2(F_{\ttt}) . \]

Under the relative differential $\dif_\m$, the triangle loops are fractured along their edges: $\dif_\m(\Omega_\pm)$ breaks each directed triangle into the three parallelogram configurations $\nu_{\alpha_1} \wedge \nu_{\alpha_3}$, $\nu_{\alpha_1} \wedge \nu_{\alpha_2}$, and $\nu_{\alpha_2} \wedge \nu_{\alpha_3}$.
This allows us to visually interpret the vanishing of the second Chern class guaranteed by Borel's theorem: $0 = c_2(E) = \left(\frac{1}{2\pi i}\right)^2 [p_2(F_{\ttt})] \in H^4(K/T;\CC)$ means that the superposition of parallelogram cycles $p_2(F_{\ttt})$ results from the Chevalley--Eilenberg differential breaking the edges of the triangle superpositions in $C^3$.

\begin{remark}[Triviality of the Third Primitive on $A_2$]\label{rmk:a2_cubic_triviality}
	In contrast to the non-trivial second primitive $\alpha_2$, the third primitive $\alpha_3$ vanishes on the flag manifold $K/T = \mathrm{SU}(3)/T$:
	\[ \alpha_3 \equiv 0 \in (\Lambda^5 \m^\vee)^T = C^5 . \]
	This vanishing can be understood through three complementary viewpoints:
	\begin{enumerate}[(i)]
		\item \textit{Cochain Duality:} By the upcoming Proposition~\ref{signtensor}, the cochain spaces satisfy $C^{N-n} \cong C^n \otimes \mathbf{1}'$ for $N = \dim_\RR(K/T) = 6$. Since there are no non-zero $T$-invariant $1$-forms on the flag manifold ($C^1 = (\m^\vee)^T = \{0\}$), duality forces
		\[ C^5 \cong C^1 \otimes \mathbf{1}' = \{0\} . \]
		\item \textit{Weight Cancellation on the Root Lattice:} As established in Remark~\ref{loops0}, a monomial $\omega^{\beta_1} \wedge \dots \wedge \omega^{\beta_5}$ is $T$-invariant if and only if its roots sum to zero. The root system $\Phi$ of $A_2$ contains exactly six roots whose total sum vanishes: $\sum_{\alpha \in \Phi} \alpha = 0$. Any non-zero $5$-form must choose five distinct roots, thus omitting exactly one root $\gamma \in \Phi$. Its net weight is therefore $0 - \gamma = -\gamma \neq 0$, preventing the closure of any $5$-step polygonal loop on the planar hexagonal lattice.
		\item \textit{Vanishing of the Target Form:} Evaluating the restricted invariant polynomial $p_3(F_{\ttt}) = \tr(F_{\ttt}^3)$ directly on the diagonal curvature matrix from Subsection~\ref{sec:canonical_curvature_coord} yields
		\[ (F_{\ttt})_{11}^3 = -(\nu_{\alpha_1} + \nu_{\alpha_3})^3 = -2(\nu_{\alpha_1} \wedge \nu_{\alpha_3}) \wedge (\nu_{\alpha_1} + \nu_{\alpha_3}) = 0 , \]
		since $\nu_\alpha^2 = 0$. Analogously, $(F_{\ttt})_{22}^3 = (F_{\ttt})_{33}^3 = 0$. Hence $p_3(F_{\ttt}) \equiv 0 \in C^6$, and the relative transgression equation $\dif_\m \alpha_3 = -p_3(F_{\ttt})$ is satisfied as the identity $0 = 0$.
	\end{enumerate}
\end{remark}

After this concrete coordinate visualization on the $A_2$ root lattice, we turn in Section~\ref{sec:homo} to the global homological problem: resolving the coinvariant algebra and deriving the Poincar\'e polynomials and Euler characteristics for arbitrary rank $k$ without encountering matrix-rank ambiguities.

\section{Recovery of Poincar\'e Polynomials}\label{sec:homo}

In the preceding sections, we developed the explicit differential geometry of the canonical fibration $K \to K/T$, constructed the relative transgression primitives $\alpha_j$, and provided their visual interpretation on the $A_2$ root lattice. In this final section, we recover the Poincar\'e polynomials and the coinvariant algebra presentation of $H^*(K/T; \CC)$ without relying on undetermined matrix ranks. We proceed along two complementary paths:
\begin{enumerate}[(1)]
	\item An additive, representation-theoretic framework utilizing symmetric functions and the Frobenius characteristic map to systematically compute the isotypic decompositions of the cochain spaces $C^n$ into Specht modules. Here, the $k=4$ multiplicity barrier appears.
	\item A multiplicative Koszul DGA resolution $K_\bullet(p_2, \dots, p_k)$ that algebraically mirrors the Weil algebra transgression and directly proves the acyclicity and Hilbert series of the flag manifold, circumventing the $k=4$ multiplicity barrier.
\end{enumerate}

\subsection{Weyl Equivariance}\label{sec:weyl_equivariance}

The relative Chevalley--Eilenberg complex $C^\bullet = (\Lambda^\bullet \m^\vee)^T$ carries an action of the Weyl group $W = N_K(T)/T \cong S_k$. To exploit this symmetry, we first establish that this action commutes with the relative differential $\dif_\m$.%

\begin{proposition}\label{diffcomm}
	The action of the Weyl group $W=S_k$ commutes with the relative Chevalley--Eilenberg differential $\dif_\m$.
\end{proposition}

\begin{proof}
	We must first describe the action of $W=N_K(T)/T$. Let $w\in W$ and let $\tilde{w}\in N_K(T)$ be a lift to the normalizer of the maximal torus $T$. The lift acts on $\g$ via $\Ad_{\tilde{w}}$, a Lie algebra automorphism. By definition of the normalizer of $T$, the action of  $\Ad_{\tilde{w}}$ preserves the Cartan subalgebra $\ttt=\mathrm{Lie}(T)$ (and its complexification $\h$). Since $\m$ is the orthogonal complement of $\ttt$ under the ($\Ad$-invariant) Killing form, it also preserves $\m$. The action therefore descends to $\m$, i.e. it commutes with the projection onto the orthogonal complement, and can be extended naturally via pullbacks to an action on $\Lambda^\bullet\m^\vee$. For example, for $\mu\in\Lambda^1\m^\vee$ and $\xi\in\m$:
	\[(\tilde{w}\cdot\mu)(\xi)=\mu(\Ad_{\tilde{w}^{-1}}(\xi)) .\]
	Clearly, this action of $N_K(T)$ on $\Lambda^\bullet\m^\vee$ descends to an action of $W=N_K(T)/T$ on $\left(\Lambda^\bullet\m^\vee\right)^T$. For $\mu \in \left(\Lambda^1\m^\vee\right)^T$:
	\[(w\cdot\mu)(\xi):=\mu(\Ad_{\tilde{w}^{-1}}(\xi)) .\]
	
	To show the commutativity with the differential, we begin by working with the lift on $N_K(T)$ and add $T$-invariance at the end. Since the differential is fully determined by its definition on $1$-forms (as the dual of the Lie bracket) and the graded Leibniz rule (which only involves signs and brackets), it is sufficient to show commutativity for $\mu\in\Lambda^1\m^\vee$:
	\begin{align*}
		\dif_\m(\tilde{w}\cdot\mu)(\xi_1,\xi_2)&=-(\tilde{w}\cdot\mu)([\xi_1,\xi_2]) \\
		&=-\mu([\Ad_{\tilde{w}^{-1}}(\xi_1),\Ad_{\tilde{w}^{-1}}(\xi_2)]) \\
		&=\dif_\m\mu(\Ad_{\tilde{w}^{-1}}(\xi_1),\Ad_{\tilde{w}^{-1}}(\xi_2)) \\
		&=\tilde{w}\cdot(\dif_\m\mu)(\xi_1,\xi_2) .
	\end{align*}
	By definition of the action of $W=N_K(T)/T$, this behavior descends to the complex of invariant forms.
\end{proof}

\begin{corollary}\label{cor:weyl_block_diagonalization}
	The complexes of cochains $C^\bullet$, of coboundaries $B^\bullet$, of cocycles $Z^\bullet$ and the cohomology $H^\bullet=Z^\bullet/B^\bullet$ decompose into subcomplexes corresponding to the decomposition of $W$ into irreducible subrepresentations. In particular, the differential block-diagonalizes according to this decomposition.
\end{corollary}

\begin{remark}[Representation Theory of Symmetric Groups]\label{skrep}
	The representation theory of symmetric groups is very well known. A good reference is Chapter 5 of \cite{etingof2011}. The classification theorem for the irreducible representations of $S_k$ states that the irreps of $S_k$ are given  by the so-called \textit{Specht modules} $V_\lambda$, where $\lambda\vdash k$ is a partition of $k$. In particular, writing $C^n = (\Lambda^n \m^\vee)^T$
	we get subcomplexes $C^\bullet_\lambda, H^\bullet_\lambda$, and so on. 
	Under this decomposition, each cochain space decomposes into Specht modules: $C^n \cong \bigoplus_{\lambda \vdash k} m_\lambda(n) V_\lambda$.
	
	In the case of $G=W=S_k$, we denote these multiplicities by $m_\lambda$ for the partitions $\lambda\vdash k$. Summing over the conjugacy classes of $S_k$, the formula for the multiplicities of the irreps in $C^n$ becomes:
	\begin{equation}\label{charproj}
		m_\lambda(n) := m_{V_\lambda}= \frac{1}{|W|} \sum_{c} |c| \chi_{C^n}(c) \overline{\chi_\lambda(c)}=\frac{1}{k!}\sum_{\nu \vdash k} |c_\nu| \chi_{C^n}(\nu) \chi_\lambda(\nu) ,
	\end{equation}
	where we omitted the complex conjugation since all the involved characters are real.
\end{remark}

To compute the cochain characters $\chi_{C^n}(\nu)$ and their multiplicities $m_\lambda(n)$ across all degrees $n$ simultaneously, we package them into generating functions via the classical Frobenius characteristic map, leading directly to the master generating function in Subsection~\ref{sec:master}.

\subsection{Symmetric Functions and the Master Generating Function}\label{sec:master}

We begin by establishing the module-theoretic duality between the relative cochain spaces of complementary degrees.

\begin{proposition}\label{signtensor}
	For every $0\leq n\leq N=k(k-1)$, there is an isomorphism of $S_k$-modules:
	\begin{equation*}
		C^{N-n}\cong C^n\otimes\mathbf{1'}.
	\end{equation*}
\end{proposition}
\begin{proof}
	Recall the non-degenerate pairing 
	\[P:C^n\times C^{N-n}\longrightarrow \CC,\]
	induced by the exterior product given by $\omega\wedge\nu=P(\omega,\nu)\Omega$ for $N=k(k-1)$ and $\Omega=\bigwedge_{\alpha\in\Phi}\omega^\alpha\in C^N$. We now consider the associated non-degenerate pairing 
	\[P_\wedge:C^n\times C^{N-n}\longrightarrow C^N\]
	given simply by the wedge product. We note that $P_\wedge$ is $W$-invariant, since the Weyl group acts via pullback:
	\[\sigma\cdot(\alpha\wedge\beta)=(\sigma\cdot\alpha)\wedge(\sigma\cdot\beta),\]
	where $\sigma\in S_k=W$. 
	The latter pairing yields a map given by:
	\begin{align*}
		C^{N-n}&\longrightarrow\Hom_\CC(C^n,C^N)\\
		\beta&\longmapsto(\alpha\mapsto\alpha\wedge\beta).
	\end{align*}
	The non-degeneracy of the pairing implies the injectivity of this map between finite dimensional vector spaces. The map is thus an isomorphism. Together with the canonical identification $\Hom_\CC(C^n,C^N)\cong(C^n)^\vee\otimes C^N$, we conclude that
	\begin{equation}\label{dualeq}
		C^{N-n}\cong (C^n)^\vee\otimes C^N.
	\end{equation}
	Now let us take closer look at the top cochain space $C^N=\langle\Omega\rangle_\CC$. An element $\sigma$ of the Weyl group acts on $\Omega$ by permuting the $\omega^\alpha$:
	\begin{equation}\label{sigom}
		\sigma\cdot\Omega=\bigwedge_{\alpha\in\Phi}\omega^{\sigma(\alpha)}=\det(\sigma\mid_{\m^\vee})\Omega,
	\end{equation}
	where we used the definition of the determinant in the last step. The determinant of a permutation is $\pm1$, in particular $\det(\sigma\mid_{\m^\vee})=(\det(\sigma\mid_{\m^\vee}))^{-1}=\det(\sigma\mid_{\m})$. 
	Note that we cannot simply conclude that $\sigma\cdot\Omega=\sgn(\sigma)\Omega$, since the Weyl group $S_k$ acts on $k(k-1)$ roots. However, the equality holds regardless. Indeed, $\sigma\mid_\g$ block-diagonalizes according to the decomposition $\g=\h\oplus\m$ and therefore $\det(\sigma\mid_\g)=\det(\sigma\mid_\h)\cdot \det(\sigma\mid_\m)$. As we saw in the proof of Proposition \ref{diffcomm}, the Weyl group $W=N(T)/T$ acts on $\g$ via the adjoint representation and therefore $\det(\sigma\mid_\g)=1$. Consequently, $\det(\sigma\mid_\m)=(\det(\sigma\mid_\h))^{-1}$.
	
	Let us compute $\det(\sigma\mid_\h)$. Recall that in our setting, $\h$ consists of traceless diagonal matrices of size $k$. These form a $(k-1)$-dimensional subspace of $\CC^k$ and the action of $S_k$ on $\h$ via adjunction corresponds to the standard action of $S_k$ on $\CC^k$. More precisely, $\h$ corresponds to a hypersurface with orthogonal complement under the standard scalar product given by the span of $(1,\dots,1)\in\CC^k$. This span is clearly invariant under the standard action of $S_k$, so the determinant of the restriction of any permutation is equal to $1$. On $\CC^k$, the determinant of a permutation is clearly given by $\sgn(\sigma)$. As above it follows by block-diagonalization that the determinant on $\h$ is given by $\sgn(\sigma)$, and therefore $\det(\sigma\mid_\m)=(\sgn(\sigma))^{-1}=\sgn(\sigma).$
		
	For Weyl groups other than $S_k$, one can employ a similar argument to show that $\det(\sigma\mid_\h)=\sgn(\sigma)$: On $\h$, the Weyl group acts via reflections $s_\alpha$ with $\det(s_\alpha)=-1$. For a Weyl element $w= s_{i_1} \dots s_{i_{l(w)}}\in W$ of length $l(w)$, we get $\det(w\mid_\h)=\det(s_{i_1}) \cdots\det(s_{i_{l(w)}})=(-1)^{l(w)}=:\sgn(w)$.
	
	Together with equation (\ref{sigom}), this implies that $C^N=\mathbf{1'}$, i.e. the top cochain space is isomorphic as a $W$-module to the sign representation. Plugging this into equation (\ref{dualeq}), we obtain 
	\[C^{N-n}\cong (C^n)^\vee\otimes \mathbf{1'}.\]
	Finally, because every representation of the symmetric group $S_k$ over $\CC$ is self-dual ($(C^n)^\vee \cong C^n$), this yields the stated $S_k$-module isomorphism $C^{N-n} \cong C^n \otimes \mathbf{1'}$.
\end{proof}

\begin{corollary}[Poincaré Duality for Multiplicities]\label{pdm}
	Let $\lambda \vdash k$ be a partition and $\lambda'$ its conjugate partition (i.e. the partition corresponding to the transposed Young diagram). For every degree $0 \leq n \leq N$, the multiplicities satisfy:
	\[m_\lambda(N-n) = m_{\lambda'}(n).\]
\end{corollary}
\begin{proof}
	Decomposing the cochain spaces into irreducible $S_k$-representations (Specht modules), we have $C^n \cong \bigoplus_{\lambda \vdash k} m_\lambda(n) V_\lambda$. Tensoring this decomposition with the sign representation $\mathbf{1'}$ yields:
	\[C^n \otimes \mathbf{1'} \cong \bigoplus_{\lambda \vdash k} m_\lambda(n) (V_\lambda \otimes \mathbf{1'}).\]
	By standard representation theory of the symmetric group (see, e.g., Exercise~4.3 in \cite{Fulton1991RepresentationTA}), tensoring a Specht module with the sign representation dualizes the partition: $V_\lambda \otimes \mathbf{1'} \cong V_{\lambda'}$. 
	Reindexing the direct sum via the involution $\lambda \mapsto \lambda'$ on the set of partitions of $k$, we obtain:
	\[C^n \otimes \mathbf{1'} \cong \bigoplus_{\lambda \vdash k} m_\lambda(n) V_{\lambda'} = \bigoplus_{\lambda \vdash k} m_{\lambda'}(n) V_\lambda.\]
	Applying the isomorphism $C^{N-n} \cong C^n \otimes \mathbf{1'}$ from Proposition \ref{signtensor}, we equate the dimensions of the respective $\lambda$-isotypic components to conclude that $m_\lambda(N-n) = m_{\lambda'}(n)$.
\end{proof}

Previously, we employed a geometric approach to the Chevalley-Eilenberg complex: We began by trying to isolate the invariant cochains before applying the Chevalley-Eilenberg differential. However, as indicated by the proof of Proposition \ref{diffcomm}, it can be fruitful to employ a more algebraic approach by first working with all the cochains and only applying the constraint of $T$-invariance at the end. We thereby trade the geometric loop picture for more efficient computations, in particular of the characters $\chi_n$, which by definition are given by the traces of the elements $\sigma\in W$ acting on $C^n$. Consider the full space of cochains $V^n:=\Lambda^n\m^\vee$. Then $C^n=(V^n)^T$. 
The root space decomposition $\m_\CC \cong \bigoplus_{\alpha \in \Phi} \mathfrak{g}_\alpha$ naturally induces a weight space decomposition on $V^n$.

The weight of a wedge product in $V^n$ is given by the sum of the weights of its constituent forms from $\m^\vee$. Let us denote this decomposition by
\[V^n=V^n_0\oplus V^n_{\mu_1}\oplus  V^n_{\mu_2}\oplus\dots.\] 
Now the action of $W=S_k$ maps weight spaces into weight spaces: $\sigma(V^n_{\mu})=V^n_{\sigma(\mu)}$. In a suitable basis, $\sigma$ is block-diagonalized w.r.t. the weight space decomposition of $V^n$. Only the blocks with non-trivial diagonal entries, namely with $\sigma(\mu)=\mu$, contribute to the trace. Let us keep track of these computations by introducing a grading by formal variables: Let $h:=\exp(\xi)\in T$ for $\xi\in\h$ be a formal element of the maximal torus and let $z_i:=e^{L_i(\xi)}$. Then for $\mu=\sum_ia_iL_i\in \h$, the torus element $h$ acts on $V^n_\mu$ via multiplication by the monomial $z^\mu:=\prod_i z_i^{a_i}$. For example, $\mu=\alpha_1=L_1-L_2$ corresponds to the monomial $z_1z_2^{-1}$, and a form $\omega^{-\alpha} \wedge \omega^{-\beta}$ corresponds to $\mu = -\alpha - \beta$. We get the following trace identity:
\begin{equation}\label{trtr}
	\tr(\sigma\circ h(z)\mid_{V^n})=\sum_{\mu=\sigma(\mu)}z^\mu\tr(\sigma\mid_{V^n_\mu}).
\end{equation}
In order to extract the invariant part $C^n=V^n_0$, we isolate the corresponding term in the direct sum decomposition. This corresponds to taking the constant term $\mathrm{CT}_z$ in $z$:
\begin{equation}\label{charct}
	\chi_n(\sigma)=\tr(\sigma\mid_{V^n_0})=\mathrm{CT}_z \left[\tr(\sigma\circ h(z)\mid_{V^n})\right].
\end{equation}
The traces $\tr(\sigma\circ h\mid_{V^n})$ can be assembled elegantly in terms of their generating function $T_\sigma(t,z)$:
\begin{equation}\label{genfc}
	T_\sigma(t,z):=\sum_{n=0}^{k(k-1)}\tr(\sigma\circ h(z)\mid_{V^n})t^n=\sum_{n=0}^{k(k-1)}\tr(\sigma\circ h(z)\mid_{\Lambda^n\m^\vee})t^n.
\end{equation}
To rewrite this further, we use the following basic lemma.
\begin{lemma}\label{lem:det_trace_id}
	Let $V$ be a complex vector space of dimension $N$ and let $f\in\mathrm{End}_\CC(V)$. Then the following identity holds:
	\[\sum_{n=0}^N\tr(\Lambda^nf)t^n=\det(I+tf),\]
	where $I$ is the identity endomorphism on $V$. 
\end{lemma}
\begin{proof}
	Since diagonalizable endomorphisms are dense in $\mathrm{End}_\CC(V)$, we assume w.l.o.g. that $f$ is diagonalizable with eigenvalues $\lambda_1,\dots,\lambda_N$. Then $\tr(\Lambda^n f)=e_n(\lambda_1,\dots,\lambda_N)$. After observing that 
	\[\det(I + tf) = \prod_{i=1}^N (1 + t\lambda_i) = \sum_{n=0}^N e_n(\lambda_1,\dots,\lambda_N) t^n,\]
	we can conclude the proof.
\end{proof}

Applying Lemma~\ref{lem:det_trace_id} to equation (\ref{genfc}) yields:
\[T_\sigma(t,z)=\det(I+t(\sigma\circ h(z))\mid_{\m^\vee}).\]
Combining this with equation (\ref{charct}), we obtain an expression for the generating function of the characters $\chi_n(\sigma)$:
\begin{equation}\label{chargen}
	X_\sigma(t)=\sum_{n=0}^{k(k-1)} \chi_n(\sigma) t^n = \mathrm{CT}_z \left[ \det\left(I + t(\sigma\circ h(z))\mid_{\m^\vee}\right) \right],
\end{equation}
where we used the linearity of $\mathrm{CT}_z$. Because $\chi_n$ is a class function, $X_\sigma(t)$ depends only on the cycle type $\rho = \rho(\sigma) \vdash k$ of the permutation; we shall denote it unambiguously by $X_\rho(t)$.

Let us now further develop equation (\ref{charproj}) for the multiplicities $m_\lambda$ of the irreps, which presently still relies on explicit character tables for  $S_k$ to plug in the $\chi_\lambda$. We want instead to translate all of this into the language of polynomials. 
Let us denote the generating function of the multiplicities as follows:
\[M_\lambda(t) := \sum_{n=0}^{k(k-1)} m_\lambda(n) t^n.\]
We can rephrase Corollary \ref{pdm} in terms of this generating function.
\begin{corollary}[Poincaré Duality for Multiplicity Generating Functions]\label{pdm2}
	The multiplicity generating functions satisfy the following equation:
	\[M_\lambda(t) = t^N M_{\lambda'}(t^{-1}).\]
\end{corollary}
\begin{proof}
	To derive the identity for the generating functions, we substitute the relation from Corollary \ref{pdm} into the definition of $M_\lambda(t)$:
	\[M_\lambda(t) = \sum_{n=0}^N m_\lambda(n) t^n = \sum_{n=0}^N m_{\lambda'}(N-n) t^n.\]
	Substituting the index $j = N - n$ gives $n = N - j$. Since the summation runs over the same bounds $0 \leq j \leq N$, we factor out $t^N$ to obtain:
	\[M_\lambda(t) = \sum_{j=0}^N m_{\lambda'}(j) t^{N-j} = t^N \sum_{j=0}^N m_{\lambda'}(j) (t^{-1})^j = t^N M_{\lambda'}(t^{-1}).\]
	This completes the proof.
\end{proof}
Plugging equation (\ref{charproj}) into the definition of $M_\lambda(t)$ yields:
\begin{align*}
	M_\lambda(t)&= \sum_{n=0}^{k(k-1)}\left(\frac{1}{k!}\sum_{\nu \vdash k} |c_\nu| \chi_n(\nu) \chi_\lambda(\nu)\right)t^n\\
	&=\frac{1}{k!} \sum_{\nu \vdash k} |c_\nu| \chi_\lambda(\nu) \left( \sum_{n=0}^{k(k-1)} \chi_n(\nu) t^n \right)\\
	&=\frac{1}{k!} \sum_{\nu \vdash k} |c_\nu| X_\nu(t) \chi_\lambda(\nu).
\end{align*}
Equation (\ref{chargen}) gives us a nice expression for $X_\nu$. The corresponding tool for the $\chi_\lambda$ is the Frobenius characteristic map, an isometric isomorphism between the ring of characters of $S_k$ and the ring of symmetric functions (a proof of this property can be found under (7.3) in \cite{macdo}).

\begin{definition}[Frobenius Characteristic Map and Master Generating Function]\label{def:frob_master}
	Let $\Lambda_{\mathrm{sym}} = \bigoplus_{k \ge 0} \Lambda_{\mathrm{sym}}^k$ denote the graded ring of symmetric functions over $\CC$. The \textit{Frobenius characteristic map} 
	\[\operatorname{ch} \colon R(S_k) \otimes_\ZZ \CC \xrightarrow{\sim} \Lambda_{\mathrm{sym}}^k\]
	is the canonical isometric isomorphism assigning to each class function $f$ the symmetric function
	\begin{equation}\label{chardef}
		\operatorname{ch}(f) := \frac{1}{k!}\sum_{\sigma\in S_k}f(\sigma)p_{\rho(\sigma)}(x) = \sum_{\rho \vdash k} z_\rho^{-1} f(\rho) p_\rho(x) ,
	\end{equation}
	where $p_\rho(x) = p_{\rho_1} p_{\rho_2} \cdots$ is the power-sum symmetric function, and $z_\rho = \prod_{i \ge 1} i^{m_i} m_i!$ is the order of the centralizer of a permutation of cycle type $\rho = (1^{m_1} 2^{m_2} \dots)$. Under $\operatorname{ch}$, the irreducible character $\chi_\lambda$ of the Specht module $V_\lambda$ maps to the Schur function: $\operatorname{ch}(\chi_\lambda) = s_\lambda(x)$. In particular, the transition between power sums and Schur functions is governed by the character table:
	\begin{equation}\label{powerschur}
		p_\rho(x) = \sum_{\lambda \vdash k} \chi_\lambda(\rho) s_\lambda(x) .
	\end{equation}
	Extending scalars to the polynomial ring $\Lambda_{\mathrm{sym}}[t]$, we define the \textit{master generating function} of the relative cochain complex $C^\bullet$ as:
	\begin{equation}\label{masterfunc}
		\mathcal{F}(t, x) := \sum_{\rho \vdash k} z_\rho^{-1} X_\rho(t) p_\rho(x) \in \Lambda_{\mathrm{sym}}^k[t] .
	\end{equation}
\end{definition}

\begin{proposition}\label{frobprop}
	The master generating function expands into the Schur basis via the multiplicity generating functions:
	\[\mathcal{F}(t, x) = \sum_{\lambda \vdash k} M_\lambda(t) s_\lambda(x).\]
\end{proposition}
\begin{proof}
	By equation (\ref{chardef}), given the character $\chi_n \in R^k$, we obtain:
	\[\mathrm{ch}(\chi_n) = \sum_{\rho \vdash k} z_\rho^{-1} \chi_n(\rho) p_\rho(x).\]
	Multiplying by $t^n$ and summing over all degrees $0 \le n \le k(k-1)$ gives:
	\[\sum_{n=0}^{k(k-1)} \mathrm{ch}(\chi_n) t^n = \sum_{\rho \vdash k} z_\rho^{-1} \left( \sum_{n=0}^{k(k-1)} \chi_n(\rho) t^n \right) p_\rho(x) = \sum_{\rho \vdash k} z_\rho^{-1} X_\rho(t) p_\rho(x) = \mathcal{F}(t, x).\]
	In turn, decomposing $C^n \cong \bigoplus_{\lambda \vdash k} m_\lambda(n) V_\lambda$ gives $\chi_n = \sum_\lambda m_\lambda(n) \chi_\lambda$. Applying $\mathrm{ch}$ and summing over $n$ yields:
	\[\sum_{n=0}^{k(k-1)} \mathrm{ch}(\chi_n) t^n = \sum_{\lambda \vdash k} \left( \sum_{n=0}^{k(k-1)} m_\lambda(n) t^n \right) s_\lambda(x) = \sum_{\lambda \vdash k} M_\lambda(t) s_\lambda(x).\]
	Equating both evaluations of $\sum_n \mathrm{ch}(\chi_n) t^n$ concludes the proof.
\end{proof}

\begin{corollary}\label{schurcor}
	The multiplicity generating function transforms into the character generating function as follows:
	\[M_\lambda(t) = \sum_{\rho \vdash k} z_\rho^{-1} \chi_\lambda(\rho) X_\rho(t).\]
\end{corollary}
\begin{proof}
	 Substituting equation (\ref{powerschur}) into the definition of the master generating function yields:
	\[\mathcal{F}(t, x) = \sum_{\rho \vdash k} z_\rho^{-1} X_\rho(t) \left( \sum_{\lambda \vdash k} \chi_\lambda(\rho) s_\lambda(x) \right).\]
	Interchanging the order of summation, we obtain:
	\[\mathcal{F}(t, x) = \sum_{\lambda \vdash k} \left( \sum_{\rho \vdash k} z_\rho^{-1} \chi_\lambda(\rho) X_\rho(t) \right) s_\lambda(x).\]
	By Proposition \ref{frobprop}, we also have $\mathcal{F}(t, x) = \sum_{\lambda \vdash k} M_\lambda(t) s_\lambda(x)$. Because the Schur functions $s_\lambda(x)$ form a basis of the ring of symmetric functions, their coefficients must be equal, concluding the proof.
\end{proof}

The transformation in Corollary~\ref{schurcor} can be executed directly via the character table matrix $C_k = (\chi_\lambda(\rho))$ and the diagonal matrix $D = \operatorname{diag}(z_\rho^{-1})$ via $M(t) = C_k D X(t)$. Alternatively, if character tables are not precomputed, one can invert the transition via the Jacobi--Trudi identity $s_\lambda = \det(h_{\lambda_i - i + j})$ and Newton's identities expressing $h_m$ in terms of power sums.

\begin{example}[Betti Numbers for $k=3$ via the Master Generating Function]\label{k3decomp2}
	For $k=3$, the symmetric group $S_3$ has three conjugacy classes corresponding to the partitions $\rho \in \{(1,1,1), (2,1), (3)\}$. The corresponding centralizer factors are $z_{(1,1,1)}^{-1} = \frac{1}{6}$, $z_{(2,1)}^{-1} = \frac{1}{2}$, and $z_{(3)}^{-1} = \frac{1}{3}$.
	
	The three irreducible representations of $S_3$ correspond to the Schur functions indexed by these partitions: the trivial representation $\mathbf{1}$ corresponds to $s_{(3)}$, the sign representation $\mathbf{1'}$ corresponds to $s_{(1,1,1)}$, and the standard representation $\mathbf{2}$ corresponds to $s_{(2,1)}$.
	
	We compute the generating functions $X_\rho(t)$ directly via equation (\ref{chargen}). Ordering the basis of $\m^\vee$ as $(e_1, \dots, e_6) = (\omega_{\alpha_1}, \omega_{-\alpha_1}, \omega_{\alpha_2}, \omega_{\alpha_3}, \omega_{-\alpha_2}, \omega_{-\alpha_3})$, the matrix $I + t(\sigma \circ h(z))\mid_{{\m}^\vee}$ evaluates for each cycle type to:
	\begin{align*}
		\rho(e) = (1,1,1): &\quad \operatorname{diag}\left( 1 + t \frac{z_1}{z_2}, 1 + t \frac{z_2}{z_1}, 1 + t \frac{z_2}{z_3}, 1 + t \frac{z_3}{z_2}, 1 + t \frac{z_1}{z_3}, 1 + t \frac{z_3}{z_1} \right), \\
		\rho((12)) = (2,1): &\quad \begin{pmatrix} 1 & t \frac{z_2}{z_1} & 0 & 0 & 0 & 0 \\ t \frac{z_1}{z_2} & 1 & 0 & 0 & 0 & 0 \\ 0 & 0 & 1 & t \frac{z_1}{z_3} & 0 & 0 \\ 0 & 0 & t \frac{z_2}{z_3} & 1 & 0 & 0 \\ 0 & 0 & 0 & 0 & 1 & t \frac{z_3}{z_1} \\ 0 & 0 & 0 & 0 & t \frac{z_3}{z_2} & 1 \end{pmatrix}, \\
		\rho((123)) = (3): &\quad \begin{pmatrix} 1 & 0 & 0 & 0 & 0 & t \frac{z_3}{z_1} \\ 0 & 1 & 0 & t \frac{z_1}{z_3} & 0 & 0 \\ t \frac{z_1}{z_2} & 0 & 1 & 0 & 0 & 0 \\ 0 & 0 & 0 & 1 & t \frac{z_3}{z_2} & 0 \\ 0 & t \frac{z_2}{z_1} & 0 & 0 & 1 & 0 \\ 0 & 0 & t \frac{z_2}{z_3} & 0 & 0 & 1 \end{pmatrix}.
	\end{align*}
	Computing the constant terms with respect to $z = (z_1, z_2, z_3)$ of their determinants yields:
	\begin{align*}
		X_{(1,1,1)}(t) &= 1 + 3t^2 + 2t^3 + 3t^4 + t^6,\\
		X_{(2,1)}(t) &= 1 - t^2 + t^4 - t^6,\\
		X_{(3)}(t) &= 1 + 2t^3 + t^6.
	\end{align*}
	We now construct the master generating function $\mathcal{F}(t,x)$ using equation (\ref{masterfunc}):
	\begin{align*}
		\mathcal{F}(t, x) &= \frac{1}{6} X_{(1,1,1)}(t) p_{(1,1,1)} + \frac{1}{2} X_{(2,1)}(t) p_{(2,1)} + \frac{1}{3} X_{(3)}(t) p_{(3)} \\[1ex]
		&= \left(\frac{1}{6} p_{(1,1,1)} + \frac{1}{2} p_{(2,1)} + \frac{1}{3} p_{(3)}\right) + \left(\frac{1}{2} p_{(1,1,1)} - \frac{1}{2} p_{(2,1)}\right) t^2 \\
		&\quad + \left(\frac{1}{3} p_{(1,1,1)} + \frac{2}{3} p_{(3)}\right) t^3 + \left(\frac{1}{2} p_{(1,1,1)} + \frac{1}{2} p_{(2,1)}\right) t^4 \\
		&\quad + \left(\frac{1}{6} p_{(1,1,1)} - \frac{1}{2} p_{(2,1)} + \frac{1}{3} p_{(3)}\right) t^6.
	\end{align*}
	Substituting the basis transitions $p_\rho = \sum \chi_\lambda(\rho) s_\lambda$, the coefficients resolve into positive integer combinations of the Schur polynomials:
	\[\mathcal{F}(t, x) = s_{(3)} + (s_{(1,1,1)} + s_{(2,1)}) t^2 + (s_{(3)} + s_{(1,1,1)}) t^3 + (s_{(3)} + s_{(2,1)}) t^4 + s_{(1,1,1)} t^6 .\]
	By Proposition \ref{frobprop}, reading off the coefficients yields the multiplicity polynomials:
	\begin{align*}
		M_{(3)}(t) &= 1 + t^3 + t^4, \\
		M_{(1,1,1)}(t) &= t^2 + t^3 + t^6, \\
		M_{(2,1)}(t) &= t^2 + t^4.
	\end{align*}
	Notice how $M_{(1,1,1)}(t) = t^6 M_{(3)}(t^{-1})$ and $M_{(2,1)}(t) = t^6 M_{(2,1)}(t^{-1})$, confirming the Poincaré duality of Corollary~\ref{pdm2}. The coefficients of $t^n$ recover the cochain decompositions:
	\begin{alignat*}{4}
		C^0 &\cong \mathbf{1}, &\qquad C^2 &\cong \mathbf{1'} \oplus \mathbf{2}, &\qquad C^4 &\cong \mathbf{1} \oplus \mathbf{2}, &\qquad C^6 &\cong \mathbf{1'}, \\
		C^1 &\cong \mathbf{0}, & C^3 &\cong \mathbf{1} \oplus \mathbf{1'}, & C^5 &\cong \mathbf{0}. & &
	\end{alignat*}
	Because each non-zero multiplicity is exactly $1$, Schur's Lemma implies that the differentials between isotypic components are either zero or isomorphisms. A non-zero boundary $\dif_2(\mathbf{1'}) = \mathbf{1'}$ and $\dif_3(\mathbf{1}) = \mathbf{1}$ immediately forces $H^2 \cong \mathbf{2}, H^4 \cong \mathbf{2}$, and $H^0 \cong \mathbf{1}, H^6 \cong \mathbf{1'}$, fully recovering the Betti numbers $(1, 0, 2, 0, 2, 0, 1)$ without matrix rank calculations.
\end{example}

However, as the rank increases, this purely additive representation-theoretic framework encounters the aforementioned barrier. 

\begin{example}[The Multiplicity Barrier for $k=4$]\label{k4decomp}
	For $k=4$, the symmetric group $S_4$ has five conjugacy classes, corresponding to the partitions of $4$. The irreducible representations are: $\mathbf{1}$ $((4))$, $\mathbf{1'}$ $((1,1,1,1))$, $\mathbf{2}$ $((2,2))$, $\mathbf{3}$ $((3,1))$, and $\mathbf{3'}$ $((2,1,1))$. 
	
	The character table is given as follows:
	\begin{table}[h]
		\centering
		\renewcommand{\arraystretch}{1.2}
		\begin{tabular}{c|ccccc}
			\hline
			Size & $1$ & $6$ & $8$ & $6$ & $3$ \\
			$S_4$ & $e$ & $(12)$ & $(123)$ & $(1234)$ & $(12)(34)$ \\
			\hline
			$\mathbf{1}$ & $1$ & $1$ & $1$ & $1$ & $1$ \\
			$\mathbf{1'}$ & $1$ & $-1$ & $1$ & $-1$ & $1$ \\
			$\mathbf{2}$ & $2$ & $0$ & $-1$ & $0$ & $2$ \\
			$\mathbf{3}$ & $3$ & $1$ & $0$ & $-1$ & $-1$ \\
			$\mathbf{3'}$ & $3$ & $-1$ & $0$ & $1$ & $-1$ \\
			\hline
		\end{tabular}
		\label{tab:s4_characters}
	\end{table}
	\FloatBarrier
	Applying the constant term extraction algorithm on $\m^\vee$ for $\mathfrak{sl}_4(\CC)$ yields the Specht module decomposition of $C^\bullet$:
	\begin{align*}
		C^{0}&\cong\mathbf{1} ,\\
		C^{1}&\cong\mathbf{0} ,\\
		C^{2}&\cong\mathbf{3}\oplus\mathbf{3'} ,\\
		C^{3}&\cong\mathbf{1}\oplus\mathbf{3}\oplus\mathbf{3'}\oplus\mathbf{1'} ,\\
		C^{4}&\cong\mathbf{1}\oplus3\mathbf{3}\oplus3\mathbf{2}\oplus\mathbf{3'}\oplus2\mathbf{1'} ,\\
		C^{5}&\cong\mathbf{1}\oplus3\mathbf{3}\oplus2\mathbf{2}\oplus3\mathbf{3'}\oplus\mathbf{1'} ,\\
		C^{6}&\cong\mathbf{1}\oplus5\mathbf{3}\oplus5\mathbf{3'}\oplus\mathbf{1'} ,\\
		C^{7}&\cong\mathbf{1}\oplus3\mathbf{3}\oplus2\mathbf{2}\oplus3\mathbf{3'}\oplus\mathbf{1'} ,\\
		C^{8}&\cong2\mathbf{1}\oplus\mathbf{3}\oplus3\mathbf{2}\oplus3\mathbf{3'}\oplus\mathbf{1'} ,\\
		C^{9}&\cong\mathbf{1}\oplus\mathbf{3}\oplus\mathbf{3'}\oplus\mathbf{1'} ,\\
		C^{10}&\cong\mathbf{3}\oplus\mathbf{3'} ,\\
		C^{11}&\cong\mathbf{0} ,\\
		C^{12}&\cong\mathbf{1'} .
	\end{align*}
	Consider for example the restriction of the Chevalley-Eilenberg differential $\dif_5: C^5 \to C^6$ to the $\mathbf{3}$-isotypic component. By Schur's Lemma, any $S_4$-equivariant map between these spaces is governed by a matrix of complex scalars:
	\[\mathrm{Hom}^{S_4}\left( V_{\mathbf{3}}^{\oplus 3}, V_{\mathbf{3}}^{\oplus 5} \right) \cong \CC^{5 \times 3}.\] 
	Because the multiplicities exceed $1$, Schur's Lemma no longer determines whether the map has maximal rank ($3$) or drops rank. The master generating function $\mathcal{F}(t,x)$ operates entirely within the additive category of $S_4$-modules: it computes dimensions and traces, but it cannot determine the rank of this $5 \times 3$ block.
\end{example}

The obstruction encountered in Example \ref{k4decomp} exposes the intrinsic limits of the character-theoretic approach. By relying on multiplicity generating functions, we implicitly applied a forgetful functor: we reduced the Chevalley-Eilenberg complex $C^\bullet$ from a Differential Graded Algebra (DGA) to a sequence of graded $W$-modules.
This allowed us to achieve an isotypic decomposition of the cochain spaces and block-diagonalize the differentials. However, it does not capture the behavior within the isotypic blocks, which is indispensable for determining Betti numbers. 

The differential $\dif_\m$ on $C^\bullet = (\Lambda^\bullet {\m}^\vee)^T$ is not a collection of independent linear maps, but a graded derivation induced in degree zero by the dual of the Lie bracket and defined on higher degrees via the graded Leibniz rule:
\[\dif_\m(\eta \wedge \beta) = (\dif_\m\eta) \wedge \beta - (-1)^{\vert{}\eta\vert{}} \eta \wedge (\dif_\m\beta).\] 
For instance, if $\alpha$ is an exact boundary ($\alpha = \dif_\m\eta$) and $\beta$ is a closed form ($\dif_\m\beta = 0$), their wedge product is automatically an exact boundary in a higher degree: $\alpha \wedge \beta = \dif_\m(\eta \wedge \beta)$. Hence the matrix entries in our diagonal blocks are not arbitrary degrees of freedom; they are constrained by the ideal generated by these wedge products. 

The additive formalism fails because it ignores this multiplicative structure. To capture these relations and resolve the cohomology without undetermined matrix ranks, we must abandon the category of modules and work with the full differential graded algebra. In Subsection~\ref{sec:koszul_dga}, we introduce the Koszul DGA $K_\bullet(p_2, \dots, p_k)$, which algebraically mirrors the transgression relations and recovers the Poincaré polynomial of the flag manifold. 

\subsection{The Koszul Differential Graded Algebra}\label{sec:koszul_dga}

While the master generating function $\mathcal{F}(t, x)$ operates within the additive category of $S_k$-modules, computing the cohomology ring $H^*(K/T; \CC)$ requires capturing the multiplicative relations among boundaries. We achieve this by introducing the Koszul differential graded algebra $K_\bullet(p_2, \dots, p_k)$, which algebraically models the restriction of the Weil algebra $W(\kk)$ to the Cartan subalgebra $\h$ and establishes the resolution of the coinvariant algebra.

We now introduce an algebraic framework that models the differential on invariant forms by setting up a Koszul complex.
This clarifies the precise algebraic roles played by the invariant polynomials and the primitive exterior generators in the computation of the cohomology ring. Our standard reference for Koszul complexes is Chapter 4.5 of \cite{Weibel_1994}.

\begin{definition}[The Koszul DGA]\label{def:koszul}
	Let $\h$ be the Cartan subalgebra of $\g = \mathfrak{sl}(k, \CC)$, with coordinate weight functionals $L_1, \dots, L_k \in \h^\vee$ satisfying $\sum_{i=1}^k L_i = 0$, so that the symmetric algebra is
	\[S(\h^\vee) \cong \CC[L_1, \dots, L_k] / \langle \sum_{i=1}^k L_i \rangle .\]
	Let $p_j := f_j\vert{}_\h = \sum_{i=1}^k L_i^j \in S^j(\h^\vee)^W$ ($j = 2, \dots, k$) be the $r = k-1$ restricted fundamental Weyl-invariant polynomials, and let $V_{\mathrm{odd}} = \bigoplus_{j=2}^k \CC \eta_j \cong \CC^{k-1}$ be a graded vector space spanned by formal odd symbols $\eta_j$.
	
	The \textit{Koszul differential graded algebra} modeling the flag manifold is defined as the tensor product:
	\begin{equation}\label{koszuldecomp}
		K_\bullet(p_2, \dots, p_k) := S(\h^\vee) \otimes_\CC \Lambda^\bullet_\CC(\eta_2, \dots, \eta_k) ,
	\end{equation}
	equipped with the differential $\dif_\kappa$ of Koszul degree $-1$ uniquely determined by:
	\begin{align}
		\dif_\kappa(p \otimes 1) &= 0 \quad \text{for all } p \in S(\h^\vee) = K_0 , \label{deg0def} \\
		\dif_\kappa(1 \otimes \eta_j) &= p_j \otimes 1 \in S^j(\h^\vee) \otimes 1 \subset K_0 \quad (j = 2, \dots, k) , \label{deg1def}
	\end{align}
	and extended to all higher exterior products via the graded Leibniz rule:
	\begin{equation}\label{eq:koszul_leibniz}
		\dif_\kappa(a \wedge b) = \dif_\kappa(a) \wedge b + (-1)^{|a|} a \wedge \dif_\kappa(b) .
	\end{equation}
\end{definition}

Under this identification, the polynomial ring $S(\h^\vee)$ is concentrated in Koszul degree $0$, while each generator $\eta_j \in \Lambda^1(V_{\mathrm{odd}})$ carries Koszul degree $1$ and topological degree $2j-1$. The polynomials $p_j \in S^j(\h^\vee)$ carry Koszul degree $0$ and topological degree $2j$, reflecting the doubling of degrees in the Borel isomorphism. 
The exterior generators $\eta_j$ act as algebraic surrogates for the universal Chern--Simons cochains $Tf_j \in W^{2j-1}(\kk)$ and their projections $\eta_j = \pi(Tf_j)$, enforcing the transgression relation $\dif_\kappa(1 \otimes \eta_j) = p_j \otimes 1$ directly on the symmetric algebra $S(\h^\vee)$, algebraically mirroring the geometric transgression $\dif_\m \alpha_j = -p_j(F_\ttt)$ on the base. 

Applying the graded Leibniz rule \eqref{eq:koszul_leibniz} iteratively to a general basis element $1 \otimes (\eta_{i_1} \wedge \dots \wedge \eta_{i_p}) \in K_p$ yields:
\[\dif_\kappa(1 \otimes \eta_{i_1} \wedge \dots \wedge \eta_{i_p}) = \sum_{j=1}^p (-1)^{j+1} p_{i_j} \otimes (\eta_{i_1} \wedge \dots \wedge \widehat{\eta_{i_j}} \wedge \dots \wedge \eta_{i_p}) ,\]
which matches the classical definition of the Koszul boundary operator (cf.\ Chapter 4.5 in \cite{Weibel_1994}). 

\begin{remark}[Restriction of the Weil Algebra]\label{weilres}
	The decomposition \eqref{koszuldecomp} and the differential rules \eqref{deg0def} and \eqref{deg1def} directly mirror the structure of the Weil algebra $W(\kk) = S(\kk^\vee) \otimes \Lambda(\kk^\vee)$ introduced in Definition~\ref{def:weil-algebra}. 
	
	Recall from Section~\ref{sec:weil} that the total Weil differential is the sum $\dw = \dif_\text{K} + \dce$. When restricting the underlying Lie algebra to the abelian Cartan subalgebra $\h$, the bracket vanishes ($[\h, \h] = 0$), which implies $\dif_\wedge = 0$ on $\Lambda(\h^\vee)$. Consequently, the curvature variable simplifies to $\hat{l} = \bar{l} = l \otimes 1$, the Chevalley--Eilenberg differential vanishes ($\dce = 0$), and the total Weil differential on $W(\h)$ reduces to the Koszul differential $\dif_\text{K}$.
	In addition:
	\begin{itemize}
		\item Equation \eqref{deg0def} corresponds to the horizontal closure condition $\dif_\text{K}\hat{l} = 0$ in \eqref{kosz}, reflecting that all symmetric polynomials in $S(\h^\vee) \otimes 1$ are $\dif_\kappa$-closed.
		\item Equation \eqref{deg1def} is the polynomial analogue of the horizontal shift $\dif_\text{K}\tilde{l} = \hat{l}$ in \eqref{kosz}, mapping each primitive exterior generator $\eta_j \in \Lambda^1(V_{\mathrm{odd}})$ to its transgressive partner, the invariant polynomial $p_j \in S^j(\h^\vee)$.
	\end{itemize}
	Thus, $K_\bullet(p_2, \dots, p_k)$ can be viewed as the minimal DGA submodel of the Weil algebra generated by the primitive forms and their corresponding transgression targets. In other words, it isolates the geometric structure we are interested in from the general Weil algebra setup.
\end{remark}

Having defined the  Koszul DGA, the key question is whether it provides a resolution of the coinvariant algebra. In Subsection~\ref{sec:regular_sequences}, we prove that the restricted invariants $p_2, \dots, p_k$ form a regular sequence in $S(\h^\vee)$, ensuring that the Koszul homology is concentrated entirely in degree zero.

\subsection{Regular Sequences and Homological Acyclicity}\label{sec:regular_sequences}

To determine the homology of the  Koszul differential graded algebra $K_\bullet(p_2, \dots, p_k)$ introduced in Definition~\ref{def:koszul}, we must verify that its defining invariant generators possess no hidden algebraic dependencies. The acyclicity of a Koszul complex in positive degrees is governed by the regularity of its generating sequence.

\begin{proposition}[Regular Sequence in the Cartan Subalgebra]\label{regseqH}
	The restricted fundamental invariants $p_2, \dots, p_k$ form a \textit{regular sequence} in the symmetric algebra $S(\h^\vee)$. That is, the Weyl ideal $I_W = \langle p_2, \dots, p_k \rangle$ is a proper ideal, and for each $j \in \{2, \dots, k\}$, the image of $p_j$ is not a zero-divisor in the quotient ring $S(\h^\vee) / \langle p_2, \dots, p_{j-1} \rangle$ (with the convention $\langle p_2, \dots, p_1 \rangle := (0)$). 
\end{proposition}

\begin{proof}
	First, we examine the structure of the invariant subring $S(\h^\vee)^W$. An element $H \in \h$ is represented by a traceless diagonal matrix with coordinate functionals $L_1, \dots, L_k \in \h^\vee$ satisfying $\sum_{i=1}^k L_i = 0$. The Weyl group $W \cong S_k$ acts on $\h$ by permuting these coordinate functionals.
	
	The restricted invariants $p_j = \sum_{i=1}^k L_i^j \in S^j(\h^\vee)$ for $j = 2, \dots, k$ are the power-sum symmetric polynomials. On the ambient polynomial ring $\CC[L_1, \dots, L_k]$, the symmetric polynomials are freely generated by the elementary symmetric polynomials. Via the previously encountered Newton's identities, they are equivalently generated by the power sums $p_1, \dots, p_k$. Restricting to $\h = \ker(p_1)$ sets the linear relation $p_1 = \sum_{i=1}^k L_i = 0$. In accordance with Theorem (A) of Chevalley \cite{cheval}, which establishes that the invariant ring of any finite reflection group is a free polynomial algebra generated by $\dim \h = k-1$ generators, we obtain:
	\[S(\h^\vee)^W = \CC[p_2, \dots, p_k] .\]
	Because $p_2, \dots, p_k$ serve as algebraically independent coordinate variables in this polynomial ring, they trivially form a regular sequence within $S(\h^\vee)^W$.

	Next, we extend this regularity from the invariant subring to the full symmetric algebra $S(\h^\vee)$. In proving Theorem (B) (and Lemma 3) of \cite{cheval}, Chevalley established that if $A_1, \dots, A_N$ (where $N = |W| = k!$) are homogeneous polynomials whose residue classes form a basis of the quotient space $S(\h^\vee) / I_W$, then every $P \in S(\h^\vee)$ can be uniquely expressed as:
	\begin{equation}\label{chevfree}
		P = \sum_{j=1}^N c_j A_j \quad \text{with } c_j \in S(\h^\vee)^W ,
	\end{equation}
	and that the representation of $W$ on $S(\h^\vee) / I_W$ is equivalent to the regular representation. As formulated in Example 1 of \cite{kostant1963lie}, this unique linear expansion means that $S(\h^\vee)$ is a free module over its invariant subring $S(\h^\vee)^W$ of rank $|W|$.

	Because $S(\h^\vee) \cong \bigoplus_{m=1}^N S(\h^\vee)^W A_m$ is a direct sum of free rank-one modules over $S(\h^\vee)^W$, multiplication by $p_j$ acts diagonally on each direct summand:
	\[ \big(S(\h^\vee) / \langle p_2, \dots, p_{j-1} \rangle S(\h^\vee)\big) \cong \bigoplus_{m=1}^N \big(S(\h^\vee)^W / \langle p_2, \dots, p_{j-1} \rangle\big) A_m . \]
	Because $p_j$ is not a zero-divisor on the base quotient ring $S(\h^\vee)^W / \langle p_2, \dots, p_{j-1} \rangle$, it is not a zero-divisor on the direct sum.
	
	Because a sequence that is regular on a commutative base ring remains regular on any free module over that ring, the restricted fundamental invariants $p_2, \dots, p_k$ form a regular sequence in the total space $S(\h^\vee)$.
\end{proof}

\begin{corollary}[Acyclicity of the Koszul Complex over ${\h}$]\label{acycH}
	Because the restricted fundamental invariants $p_2, \dots, p_k$ form a regular sequence in $S({\h}^\vee)$, the associated Koszul complex $K_\bullet(p_2, \dots, p_k)$ is an acyclic resolution of the coinvariant algebra (see, e.g., Corollary 4.5.4 in \cite{Weibel_1994}). Specifically, the homology is trivial in all positive Koszul degrees:
	\[H_i(K_\bullet(p_2, \dots, p_k)) = 0 \quad \text{for } i \ge 1 ,\]
	and the degree-zero homology yields the coinvariant ring:
	\[H_0(K_\bullet(p_2, \dots, p_k)) = S({\h}^\vee) / \langle p_2, \dots, p_k \rangle = S({\h}^\vee) / I_W .\]
\end{corollary}

\begin{remark}[Finite Free Resolution of the Coinvariant Algebra]\label{rmk:finite_free_res}
	Geometrically, Corollary~\ref{acycH} establishes that the augmented complex
	\begin{equation}\label{freeres}
		0 \longrightarrow K_{k-1} \xrightarrow{\ \dif_\kappa\ } K_{k-2} \xrightarrow{\ \dif_\kappa\ } \dots \xrightarrow{\ \dif_\kappa\ } K_1 \xrightarrow{\ \dif_\kappa\ } K_0 \xrightarrow{\ \pi\ } S(\h^\vee)/I_W \longrightarrow 0
	\end{equation}
	forms a finite free resolution of the coinvariant algebra $S(\h^\vee)/I_W$ as an $S(\h^\vee)$-module, where each module $K_p = S(\h^\vee) \otimes_\CC \Lambda^p_\CC(V_{\mathrm{odd}})$ has free rank $\binom{k-1}{p}$. Under Borel's degree-doubling isomorphism $S(\h^\vee)/I_W \cong H^*(K/T; \CC)$, this complex provides an exact algebraic surrogate for the relative Chevalley--Eilenberg complex $(\Lambda^\bullet \m^\vee)^T$.
\end{remark}

Having established that $K_\bullet(p_2, \dots, p_k)$ provides an acyclic resolution of the coinvariant algebra, we now exploit its graded free structure. In Subsection~\ref{sec:poincare_euler}, we apply the Euler--Poincar\'e principle to its homogeneous components, directly recovering the classical product formula for the Poincar\'e polynomial and Euler characteristic of the flag manifold.

\subsection{Poincar\'e Polynomials and Euler Characteristics}\label{sec:poincare_euler}

Having established the acyclicity of the  Koszul differential graded algebra $K_\bullet(p_2, \dots, p_k)$ in Subsection~\ref{sec:regular_sequences}, we now exploit its finite free resolution to compute the cohomology of the flag manifold. By applying the Euler--Poincar\'e principle to the homogeneous components of the resolution, we obtain a direct, purely algebraic derivation of the Poincar\'e polynomial and Euler characteristic of $K/T$, bypassing combinatorial matrix-rank computations entirely and demonstrating how the degrees $d_j = j$ of the fundamental invariants completely dictate the topology of the flag manifold.

Evaluating the total dimension of the relative cohomology $H^*(K/T)$ corresponds to computing its Euler characteristic, since all non-zero Betti numbers occur in even degrees. An advantage of the Koszul setup above is that we can directly read off the Poincar\'e polynomial $P_{K/T}(t) = \sum_{m \ge 0} \dim H^m(K/T; \CC) t^m$.
	
Recall Borel's presentation of the cohomology ring from \cite{Borel1953SurLC}: 
\[H^*(K/T, \CC) \cong S(\h^\vee) / I_W = S(\h^\vee) / \langle p_2, \dots, p_k \rangle ,\] 
where $p_j = \mathrm{res}(f_j) \in S^j(\h^\vee)^W$ are the restricted fundamental invariants. This establishes that the topology of the flag manifold is encoded in the coinvariant algebra. With the help of the free resolution developed above, we can readily read off the Poincaré polynomial of the flag manifold.

Recall that the Hilbert series of a non-negatively graded commutative ring $R = \bigoplus_{n \ge 0} R_n$ is defined by the formal power series
\[\operatorname{HS}_R(q) = \sum_{n=0}^\infty \dim(R_n) q^n .\]
We compute the Hilbert series of the coinvariant ring $R = S(\h^\vee)/I_W$ using its defining regular sequence. For $\g = \mathfrak{sl}(k, \CC)$, the Cartan subalgebra $\h$ has dimension $r = k-1$. In the standard polynomial grading with $\deg(L_i) = 1$, the Hilbert series of the free polynomial ring $S(\h^\vee)$ in $k-1$ variables is:
\[\operatorname{HS}_{S(\h^\vee)}(q) = \frac{1}{(1-q)^{k-1}} ,\]
where each variable contributes a geometric series $1 + q + q^2 + \dots = \frac{1}{1-q}$, multiplied together due to the algebraic independence of the generators.

Topologically, Borel's degree doubling ($\deg_{\mathrm{dR}} = 2\deg_{\mathrm{poly}}$) assigns degree $2$ to each polynomial generator. The topological Hilbert series of the unconstrained polynomial algebra is therefore obtained via the substitution $q = t^2$:
\[\operatorname{HS}_{S(\h^\vee)}(t^2) = \frac{1}{(1-t^2)^{k-1}} .\]
This is the primary ingredient of the Hilbert series of the quotient ring; we now take into account the ideal $I_W$ generated by the restricted fundamental invariants.
	
\begin{proposition}[Poincar\'e Polynomial of the Flag Manifold]\label{prop:poincare_flag}
	The Poincar\'e polynomial of the flag manifold $K/T \cong G/B$ is given by:
	\begin{equation*}
		P_{K/T}(t) = \frac{\prod_{j=2}^k (1 - t^{2j})}{(1-t^2)^{k-1}} = \prod_{j=2}^k \frac{1 - t^{2j}}{1-t^2} = \prod_{j=2}^k \left( 1 + t^2 + t^4 + \dots + t^{2(j-1)} \right) .
	\end{equation*}
\end{proposition}
Here we invoke the classical diffeomorphism $K/T \cong G/B$, where $G = \mathrm{SL}(k, \CC)$ is a complex semisimple Lie group with Borel subgroup $B$ (the subgroup of upper-triangular matrices), and $K = \mathrm{SU}(k)$ is a maximal compact subgroup with maximal torus $T = K \cap B$ (diagonal unitary matrices).
\begin{proof}
	To compute the Hilbert series of the quotient $S(\h^\vee)/I_W$, we exploit the fact that the restricted fundamental invariants $p_2, \dots, p_k$ form a regular sequence in $S(\h^\vee)$ (Proposition~\ref{regseqH}). As established in Corollary~\ref{acycH} and Remark~\ref{rmk:finite_free_res}, the restricted Koszul complex $K_\bullet(p_2, \dots, p_k)$ is acyclic in positive degrees ($H_{p > 0}(K_\bullet(p)) = 0$), providing a finite minimal free resolution of the coinvariant algebra:
	\begin{equation}\label{freeres_aug}
		0 \longrightarrow K_{k-1} \xrightarrow{\ \dif_\kappa\ } \dots \xrightarrow{\ \dif_\kappa\ } K_1 \xrightarrow{\ \dif_\kappa\ } K_0 \xrightarrow{\ \pi\ } S(\h^\vee)/I_W \longrightarrow 0 , 
	\end{equation}
	where $K_0 = S(\h^\vee)$ and the higher modules are given by $K_p = S(\h^\vee) \otimes_\CC \Lambda^p_\CC(\eta_2, \dots, \eta_k)$ as in \eqref{koszuldecomp}. Minimality follows from the definition of the Koszul differential as mapping into the maximal ideal $S_+(\h^\vee)$.

	Applying the Euler--Poincar\'e principle to a bounded exact sequence of graded vector spaces implies that the alternating sum of dimensions vanishes degree by degree. 
	While each exterior generator $\eta_j$ carries Koszul homological degree $1$, the Koszul differential $\dif_\kappa(1 \otimes \eta_j) = p_j \otimes 1$ must be a homogeneous map of algebraic degree zero. Consequently, each $\eta_j$ is assigned the polynomial degree of its target invariant polynomial, namely $\deg_{\mathrm{poly}}(\eta_j) = \deg(p_j) = j$. With this grading in place, applying the Euler--Poincar\'e principle degree-by-degree to the homogeneous components of the free resolution \eqref{freeres} expresses the Hilbert series of the coinvariant algebra as an alternating sum:
	\begin{equation}\label{hilbs}
		\operatorname{HS}_{S(\h^\vee)/I_W}(q) = \sum_{l=0}^{k-1} (-1)^l \operatorname{HS}_{K_l}(q) .
	\end{equation}
	
	By definition, the module $K_l = S(\h^\vee) \otimes_\CC \Lambda^l_\CC(\eta_2, \dots, \eta_k)$ is a direct sum of shifted copies of the free ring $S(\h^\vee)$. Its Hilbert series is the series of the free ring multiplied by the sum of the degree shifts of all possible $l$-fold wedge products $\eta_{i_1} \wedge \dots \wedge \eta_{i_l}$:
	\[\operatorname{HS}_{K_l}(q) = \operatorname{HS}_{S(\h^\vee)}(q) \sum_{2 \le i_1 < \dots < i_l \le k} q^{\deg(p_{i_1}) + \dots + \deg(p_{i_l})} .\]
	
	Substituting this into the Euler--Poincar\'e alternating sum \eqref{hilbs} allows us to factor out $\operatorname{HS}_{S(\h^\vee)}(q)$, leaving a signed sum of all possible degree combinations:
	\[\operatorname{HS}_{S(\h^\vee)/I_W}(q) = \operatorname{HS}_{S(\h^\vee)}(q) \sum_{l=0}^{k-1} (-1)^l \sum_{2 \le i_1 < \dots < i_l \le k} q^{\deg(p_{i_1}) + \dots + \deg(p_{i_l})} .\]
	Recognizing this alternating sum over subsets as the expanded form of a product of binomials, the formula collapses to:
	\[\operatorname{HS}_{S(\h^\vee)/I_W}(q) = \operatorname{HS}_{S(\h^\vee)}(q) \prod_{j=2}^k \left(1 - q^{\deg(p_j)}\right) .\]
	
	The restricted fundamental invariants $p_j = \mathrm{res}(f_j)$ have polynomial degrees $d_j = j$. Substituting the Hilbert series of the free ring $\operatorname{HS}_{S(\h^\vee)}(q) = (1-q)^{-(k-1)}$ into the product yields the algebraic Hilbert series of the coinvariant algebra:
	\[\operatorname{HS}_{S(\h^\vee)/I_W}(q) = \frac{\prod_{j=2}^k (1 - q^j)}{(1-q)^{k-1}} = \prod_{j=2}^k \frac{1 - q^j}{1-q} .\]
	
	Converting to the topological grading via Borel's degree-doubling substitution $q = t^2$, we obtain the Poincar\'e polynomial by grouping the $k-1$ numerator and denominator terms for each index $j$:
	\begin{equation*}
		P_{K/T}(t) = \operatorname{HS}_{S(\h^\vee)/I_W}(t^2) = \frac{\prod_{j=2}^k (1 - t^{2j})}{(1-t^2)^{k-1}} = \prod_{j=2}^k \frac{1 - t^{2j}}{1-t^2} .
	\end{equation*}
	To elaborate on the final equality, we apply the standard algebraic factorization for the difference of powers, $1 - x^d = (1-x)(1 + x + x^2 + \dots + x^{d-1})$ for $x = t^2$ and $d = j$:
	\begin{equation*}
		P_{K/T}(t) = \prod_{j=2}^k \left( 1 + t^2 + t^4 + \dots + t^{2(j-1)} \right) .
	\end{equation*}
	Note that the degrees $d_j = j$ correspond directly to the exponents of the Weyl group $W = S_k$, which are given by $m_j = j-1$ for $j = 2, \dots, k$.
	
	This formula confirms that all the cohomological information of the flag manifold is encoded by the degrees of the fundamental invariants.
\end{proof}

\begin{corollary}[Euler Characteristic of the Flag Manifold]\label{cor:euler_flag}
	Evaluating the Poincar\'e polynomial at $t=1$ avoids any singularity due to the expanded sum and immediately yields the Euler characteristic of the flag manifold. For $G = \mathrm{SL}(k, \CC)$, the Weyl group is the symmetric group $S_k$, and we directly obtain:
	\[\chi(K/T) = P_{K/T}(1) = \prod_{j=2}^k j = k! = |S_k| .\]
	This recovers the classical result that the Euler characteristic is the order of the Weyl group.
\end{corollary}

In Borel's classical description \cite{Borel1953SurLC} of the coinvariant algebra, the generators of the Weyl-invariant ideal $I_W$ are the elementary symmetric polynomials $e_2, \dots, e_k$. As established in Remark~\ref{newton4} and Subsection~\ref{sec:master}, these generators can be swapped for the power-sum invariants $p_2, \dots, p_k$ via Newton's identities, yielding $I_W = \langle e_2, \dots, e_k \rangle = \langle p_2, \dots, p_k \rangle \subset S(\h^\vee)$.

While the minimal Koszul resolution algebraically establishes the coinvariant presentation and derives the Poincar\'e polynomial from the degrees of the fundamental invariants, the individual Betti numbers can also be recovered through harmonic analysis. In Subsection~\ref{sec:kostant_hodge}, we also examine Kostant's classical 1961 Lie-algebraic solution, which uses finite-dimensional Hodge theory on the nilpotent radical to index the cohomology classes directly by the lengths of Weyl group elements.

\subsection{Kostant's Highest Weight Solution}\label{sec:kostant_hodge}

While the minimal Koszul resolution algebraically establishes the coinvariant presentation and derives the Poincar\'e polynomial from the degrees of the fundamental invariants, the individual Betti numbers can also be recovered through representation theory and harmonic analysis. In 1961, Kostant \cite{Kostant61} determined the Lie algebra cohomology of the nilpotent radical $\m^+$ using finite-dimensional Hodge theory. By endowing the cochain complex $C^\bullet(\m^+, V)$ with the positive-definite Hermitian inner product induced by the Killing form and the Cartan involution, the computation of cohomology reduces to identifying the harmonic space $\mathcal{H}^*(\m^+, V) = \ker(L)$, where $L = \dif \delta + \delta \dif$ is the algebraic Laplacian.

More precisely, let ${\g}$ be a complex semisimple Lie algebra with root decomposition ${\g}={\m}^+\oplus{\h}\oplus{\m}^-$, and let $V$ be an irreducible, finite-dimensional ${\g}$-module with highest weight $\lambda$. Kostant's theorem asserts that the Lie algebra cohomology $H^{k}({\m}^+,V)$ splits (as an ${\h}$-module) into a direct sum of one-dimensional modules of multiplicity one. Furthermore, the corresponding weights are the elements of the form $w(\lambda+\sigma)-\sigma$, where $w\in W$ is any Weyl group element of length $k$ and $\sigma = \frac{1}{2}\sum_{\alpha\in\Phi^+}\alpha$ denotes the half-sum of the positive roots. 

Evaluating Kostant's general theorem (Theorem~5.14 and Corollary~5.15 in \cite{Kostant61}) for the trivial representation $V = \CC$ with highest weight $\lambda = 0$ yields the following formulation:

\begin{theorem}[Kostant's Theorem on Lie Algebra Cohomology]\label{kolialco} 
	Let ${\g}$ be a complex semisimple Lie algebra with root decomposition ${\g} = {\m}^+ \oplus {\h} \oplus {\m}^-$, and let $V$ be an irreducible, finite-dimensional ${\g}$-module with highest weight $\lambda$. Then the Lie algebra cohomology $H^k({\m}^+, V)$ decomposes into a direct sum of one-dimensional weight spaces (as ${\h}$-modules) of multiplicity one. Furthermore, the weights occurring in the decomposition $H^k({\m}^+, V)$ are the elements of the form $w(\lambda + \sigma) - \sigma$, where $w\in W$ is a Weyl group element of length $k$, and $\sigma=\frac{1}{2}\sum_{\alpha\in\Phi^+}\alpha$ denotes the half-sum of the positive roots.
\end{theorem}

As an immediate consequence of his Corollary 5.15, Kostant recovers a result first proven by Bott in \cite{Bott57}: Because the cohomology decomposes into one-dimensional weight spaces indexed by Weyl group elements of length $k$, the total dimension of $H^k({\m}^+, V)$ is simply equal to the number of elements in the Weyl group of length $k$, and furthermore, this dimension is constant and independent of the highest weight $\lambda$.

\begin{example}[Kostant's Theorem for $\mathfrak{sl}(k)$ and $k=3$]
	We apply Theorem \ref{kolialco} to our case, namely ${\g} = \slk$ and the trivial representation $V = \CC$ (meaning the highest weight is $\lambda = 0$). 
	
	The weights of the harmonic cocycles governing the cohomology are  $w(\sigma) - \sigma$ for $w \in S_k$. Recall that the positive roots are $\Phi^+ = \{L_i - L_j \mid 1 \le i < j \le k\}$, making the half-sum of positive roots:
	\[ \sigma = \frac{1}{2}\sum_{i<j} (L_i - L_j) = \sum_{i=1}^k \left(\frac{k+1}{2} - i\right) L_i. \]
	The formula used in the second step can be verified by induction, noticing that $L_{i+1}$ has one fewer positive appearance and one more negative appearance than $L_i$.
	For $k=3$, the Weyl group $S_3$ has $3!=6$ elements, and $\sigma = L_1 - L_3 = \alpha_1 + \alpha_2$. The lengths $n(w)$ of the Weyl elements then dictate the degrees of the non-vanishing relative Chevalley-Eilenberg cohomology groups $H^n({\m}, \CC)$:
	\begin{itemize}
		\item {Length 0:} $w = e$. The weight is $e(\sigma) - \sigma = 0$. This corresponds to the constants and generates the $1$-dimensional $H^0$.
		\item {Length 1:} $w = (12) = s_1$ and $w = (23) = s_2$. The respective weights evaluate to $s_1(\sigma) - \sigma = -\alpha_1$ and $s_2(\sigma) - \sigma = -\alpha_2$. These correspond  to the dual forms $\omega_{-\alpha_1}$ and $\omega_{-\alpha_2}$, yielding the $2$-dimensional $H^1(\m^+, \CC)$
		\item {Length 2:} $w = (123) = s_1 s_2$ and $w = (132) = s_2 s_1$. The respective weights evaluate to $-2\alpha_1 - \alpha_2$ and $-\alpha_1 - 2\alpha_2$. In $\Lambda^2 (\m^+)^\vee$, these weights single out  the wedge products $\omega_{-\alpha_1} \wedge \omega_{-(\alpha_1+\alpha_2)}$ and $\omega_{-\alpha_2} \wedge \omega_{-(\alpha_1+\alpha_2)}$. This yields the $2$-dimensional $H^2(\m^+, \CC)$.
		\item {Length 3:} $w = (13) = s_1 s_2 s_1$. The weight is $-2\alpha_1 - 2\alpha_2 = -2(L_1 - L_3)$. This singles out the top-degree volume form $\omega_{-\alpha_1} \wedge \omega_{-\alpha_2} \wedge \omega_{-(\alpha_1+\alpha_2)}$, yielding the $1$-dimensional $H^3(\m^+, \CC)$.
	\end{itemize}
	Taking into account that the topological degree in the flag manifold $K/T$ doubles the Lie algebra cohomology degree (as established in Theorem~\ref{thm:borel_isomorphism}(ii)), the sequence of dimensions $(1, 2, 2, 1)$  matches our established Betti numbers $(b_0, b_2, b_4, b_6)$ from Example~\ref{k3decomp2}. 
\end{example}

\subsection{Comparison with Hodge-Theoretic and Equivariant Frameworks}
Kostant's highest-weight solution demonstrates that each cohomology class admits a unique harmonic representative supported on the nilpotent radical. We compare this analytic framework with the algebraic DGA model and equivariant GKM theory.

\begin{remark}[The Limits of Hodge Theory]\label{rmk:limits_hodge}
	Kostant's approach elegantly bypasses the matrix rank problem to compute the Betti numbers. However, the method is fundamentally analytic: The space of harmonic forms does not constitute a subring, as the wedge product of two harmonic cocycles is generally not harmonic. Consequently, the Hodge-theoretic approach is blind to the multiplicative cup product structure of the cohomology. This method therefore does not help us understand the structure of the cohomology ring discovered by Borel, namely the isomorphism to the coinvariant ring. Nor does it provide a mechanism to explicitly construct the exact polynomial boundaries that would explain which loops in our root lattice do not contribute to the cohomology. This highlights the fundamental role of the DGA framework developed in Subsections~\ref{sec:koszul_dga}--\ref{sec:poincare_euler}: while Hodge theory computes the Betti numbers as vector-space dimensions, the Koszul DGA internalizes the ring structure of the coinvariant algebra.
\end{remark}

\begin{remark}[Contrast with GKM Theory]\label{rmk:gkm_contrast}
	While the cohomology of the flag manifold $K/T$ admits a well-known combinatorial description via Goresky-Kottwitz-MacPherson (GKM) theory, that approach is (at least a priori) different from the loop model presented here. GKM theory computes equivariant cohomology by evaluating polynomials on the 1-skeleton of the $T$-action, leading to a graph whose vertices are the $|W|=k!$ elements of the Weyl group. As shown in \cite{crooks2016}, the $1$-skeleton of the flag manifold under the torus action corresponds to the GKM graph, whose vertices are indexed by $W$.
	In contrast, our model operates directly on the cochain level of the relative Chevalley-Eilenberg complex, where the combinatorics take place within the $k(k-1)$ edges of the root lattice. The utility of our root-loop description lies not merely in computing the Betti numbers, but in providing explicit, coordinate differential-form primitives witnessing the vanishing of the characteristic classes on the base.
\end{remark}

This concludes our construction of explicit transgression primitives, bridging the differential geometry of the canonical fibration $K \to K/T$, the combinatorics of directed root-lattice loops, and the coinvariant algebra of the flag manifold.

\end{document}